\documentclass[a4paper, 11pt, twoside, leqno]{article}

\usepackage[T1]{fontenc}
\usepackage[utf8x]{inputenc}
\usepackage[english]{babel}
\usepackage{lmodern}               
\usepackage{amsmath,amsthm,amssymb}
\usepackage{mathrsfs,esint,bbm,stmaryrd}
\usepackage{enumitem,graphicx,xcolor,subcaption}
\usepackage[textwidth=16.1cm,centering,headheight=24pt]{geometry}
\usepackage{authblk}
\usepackage{subcaption}
\usepackage{hyperref}

\newtheorem{theorem}{Theorem}           

\newtheorem{lemma}[theorem]{Lemma}
\newtheorem{prop}[theorem]{Proposition}

\newtheorem{mainthm}{Theorem}           

\newtheorem{mainprop}[mainthm]{Proposition} %

\theoremstyle{definition}              
\newtheorem{definition}{Definition}

\theoremstyle{remark}                  
\newtheorem{step}{Step}

\newtheorem{remark}{Remark}
\newtheorem{example}{Example}

\DeclareMathOperator{\dist}{dist}                                   
\DeclareMathOperator{\sign}{sign}                                   
\DeclareMathOperator{\spt}{spt}                                     
\DeclareMathOperator{\curl}{\curl}                                  

\DeclareMathOperator{\hc}{hc}

\newcommand{\abs}[1]{\left| #1 \right|}                             
\newcommand{\norm}[1]{\left\| #1 \right\|}                          
\newcommand{\one}{\mathbbm{1}}                                      
\newcommand{\mres}
{\mathbin{\vrule height 1.6ex depth 0pt width 0.13ex\vrule height 0.13ex depth 0pt width 1.3ex}}
\newcommand{\csubset}{\subset\!\subset}                             

\DeclareMathAlphabet{\mathpzc}{OT1}{pzc}{m}{it}

\renewcommand{\d}{\mathrm{d}}

\newcommand{\N}{\mathbb{N}}       
\newcommand{\R}{\mathbb{R}}
\newcommand{\Z}{\mathbb{Z}}

\newcommand{\M}{\mathbb{M}}
\newcommand{\F}{\mathbb{F}}
\newcommand{\I}{\mathbb{I}}
\renewcommand{\SS}{\mathbb{S}}
\newcommand{\G}{\mathbf{G}} 

\newcommand{\e}{\mathbf{e}}

\newcommand{\NN}{\mathscr{N}}     

\newcommand{\GG}{\mathscr{G}}
\renewcommand{\H}{\mathscr{H}}
\renewcommand{\L}{\mathscr{L}}
\newcommand{\eps}{\varepsilon}

\SetSymbolFont{stmry}{bold}{U}{stmry}{m}{n}  

\newcommand{\PR}{\mathbb{R}\mathrm{P}^2}  

\newcommand{\RR}{\varrho}
\newcommand{\X}{\mathscr{X}}
\newcommand{\GN}{\pi_{k-1}(\NN)} 
\newcommand{\Sg}{\mathfrak{S}}
\newcommand{\nablaT}{\nabla_{\top}}

\renewcommand{\S}{\mathbf{S}}

\definecolor{lightblue}{rgb}{0.22,0.45,0.70}   
\definecolor{darkgray}{gray}{0.4}    
\definecolor{lightgray}{gray}{0.8}
\newcommand{\BBB}{}

\title{Topological singular set of vector-valued maps, II: \\
$\Gamma$-convergence for Ginzburg-Landau type functionals}

\author{Giacomo Canevari
and Giandomenico Orlandi\thanks{
Dipartimento di Informatica, Universit\`a di Verona,
Strada le Grazie 15, 37134 Verona, Italy. \\
\emph{E-mail addresses}: \texttt{giacomo.canevari@univr.it},
\texttt{giandomenico.orlandi@univr.it}}}

\date{\today}

\begin{document}

\maketitle

\begin{abstract} 
 We prove a~$\Gamma$-convergence result for a class of Ginzburg-Landau type functionals
 with~$\NN$-well potentials, where~$\NN$ is a closed and $(k-2)$-connected submanifold of~$\R^m$,
 in arbitrary dimension. This class includes, for instance, the 
 Landau-de Gennes free energy for nematic liquid crystals.
 The energy density of minimisers, subject to Dirichlet boundary conditions, converges to 
 a generalised surface (more precisely, a flat chain with coefficients in~$\GN$)
 which solves the Plateau problem in codimension~$k$.
 The analysis relies crucially on the set of topological singularities, 
 that is, the operator~$\S$ we introduced in the companion paper~\cite{CO1}.
 
 
\end{abstract}

\section{Introduction}

Let~$n\geq 0$, $k\geq 2$, $m\geq 2$ be integers, and
let~$\Omega\subseteq\R^{n+k}$ be a bounded, smooth domain.
Let~$\eps>0$ be a small parameter. For~$u\in W^{1,k}(\Omega, \, \R^m)$, 
we define the functional
\begin{equation} \label{energy} 
 E_\eps(u) 
 := \int_\Omega \left(\frac{1}{k}\abs{\nabla u}^k + \frac{1}{\eps^k} f(u) \right) \!.
\end{equation}
Here, $f\colon\R^m\to\R$ is a non-negative, continuous potential,
whose zero-set~$\NN := f^{-1}(0)$ is assumed to be a smooth, 
compact, $(k-2)$-connected manifold without boundary. The aim of this paper 
is to understand the asymptotic behaviour
of the functionals~$E_\eps$ in the limit as~$\eps\to 0$,
by a $\Gamma$-convergence approach. Our analysis
builds upon the results obtained in a companion paper,~\cite{CO1}.

Functionals of the form~\eqref{energy},
which describe a kind of penalised $k$-harmonic 
map problem (see e.g.~\cite{ChenStruwe, LinWang}),
arise naturally in different contexts.
A well-known example is the Ginzburg-Landau 
functional, which corresponds to the case $k=m=2$ and~$f(u) := (\abs{u}^2-1)^2$,
so that the zero-set of~$f$ is the unit circle, $\NN = \SS^{1}\subseteq\R^2$.
The Ginzburg-Landau functional was originally introduced 
as a (simplified) model for superconductivity, but has attracted 
considerable attention in the mathematical community
since the pioneering work by Bethuel, Br\'ezis and H\'elein~\cite{BBH}.
Another example, arising from materials science,
is the Landau-de Gennes model for nematic liquid crystals
(in the so-called one-constant approximation, see e.g.~\cite{deGennes}).
In this case, $k=2$ and the zero-set of~$f$ is a real projective plane~$\NN = \PR$,
whose elements can be interpreted as the preferred configurations 
for the material. Functionals of the form~\eqref{energy}
have also applications to mesh generation in numerical analysis, via the so-called 
cross-field algorithms (see e.g.~\cite{CrossFields}). 

Minimisers of~\eqref{energy} subject to a boundary condition
$u_{|\partial\Omega} = v\in W^{1-1/k, k}(\partial\Omega, \, \NN)$
may not satisfy uniform energy bounds, due to topological obstructions
carried by the boundary datum~$v$. When this phenomenon occurs,
the energy of minimisers is of order~$\abs{\log\eps}$ 
(see e.g. \cite{BBH, BourgainBrezisMironescu2000, Riviere-DenseSubsets}
in case~$k=2$, $\NN=\SS^1$). A similar phenomenon arises 
for tangent vector fields on a closed manifold, due to the Poincar\'e-Hopf theorem
(see e.g.~\cite{IgnatJerrard}).
The analysis of the Ginzburg-Landau case shows that the energy 
of minimisers (and other critical points) concentrates, 
to leading order, on a $n$-dimensional surface;
see e.g. \cite{BBH,LinRiviere,BethuelBrezisOrlandi,SandierSerfaty}.
From a variational viewpoint, the Ginzburg-Landau
functional itself can be considered an approximation of 
an $n$-dimensional ``weighted area'' functional, in a sense that can 
be made precise by $\Gamma$-convergence~\cite{JerrardSoner-GL, ABO2, 
SandierSerfaty, AlicandroPonsiglione}.
Therefore, the Ginzburg-Landau functional and its variants
have been proposed as tools to construct ``weak minimal surfaces'' 
or, more precisely, stationary varifolds of codimension greater than~one 
\cite{AmbrosioSoner, LinRiviere-Quant, BethuelBrezisOrlandi, Stern, PigatiStern}.
Energy concentration results have also been established for Landau-de Gennes 
minimisers~\cite{MajumdarZarnescu, NguyenZarnescu, BaumanParkPhillips, GolovatyMontero, pirla, INSZ-Hedgehog, pirla3, INSZ-1/2, ContrerasLamy-conv}.
To our best knowledge, minimisers of functionals associated with more general manifolds~$\NN$, in the logarithmic energy regime, have been studied only in case~$n=0$, $k=2$ so far~\cite{pirla, MonteilRodiacVanSchaftingen, MonteilRodiacVanSchaftingen2}.

In this paper, we show that the re-scaled 
functionals~$\abs{\log\eps}^{-1} E_\eps$ do converge
to an $n$-dimensional weighted area functional, thus extending
the results in~\cite{ABO2, JerrardSoner-GL} to more general 
potentials~$f$. The key tool is the topological singular set 
of vector-valued maps, that is, the operator~$\S$ we introduced in~\cite{CO1},
which identifies the appropriate topology of the $\Gamma$-convergence.
The operator~$\S$ effectively serves as a replacement, or rather a generalisation,
of the distributional Jacobian, which is commonly used
when the distinguished manifold is a sphere,~$\NN=\SS^{k-1}$.
In order to overcome the algebraic issues that 
make the distributional Jacobian incompatible with the topology of
other manifolds~$\NN$, we work in the setting 
of \emph{flat chains} with coefficients in~$\GN$~\cite{Fleming}.
In the context of manifold-constrained problems, the 
use of flat chains with coefficients in an Abelian group
was proposed by Pakzad and Rivi\`ere~\cite{PakzadRiviere}
and traces its roots back in the earlier literature on the subject:
the very notion of ``minimal connection'', introduced by Brezis, 
Coron and Lieb~\cite{BrezisCoronLieb}, can be interpreted as 
the flat norm of the distributional Jacobian.

We state our main $\Gamma$-convergence result,
Theorem~\ref{th:main}, in Section~\ref{sect:setting},
after introducing some background and notation.
Here, we present an application (Theorem~\ref{th:mainmin} below)
to the asymptotic analysis of minimisers of~\eqref{energy}
in the limit as~$\eps\to 0$.
We make the following assumptions on the potential~$f$:
\begin{enumerate}[label=(H\textsubscript{\arabic*}), ref=H\textsubscript{\arabic*}]
 \setcounter{enumi}{0}
 
 \item \label{hp:f} \label{hp:first} $f\in C^1(\R^m)$ and~$f\geq 0$. 
 
 \item \label{hp:N} The set $\NN := f^{-1}(0)\neq\emptyset$ is a 
 smooth, compact manifold without boundary. 
 Moreover, $\NN$ is  $(k-2)$-connected, that is
 $\pi_0(\NN) = \pi_{1}(\NN) = \ldots = \pi_{k-2}(\NN) = 0$, 
 and~$\pi_{k-1}(\NN)\neq 0$. In case~$k=2$, we also assume
 that~$\pi_1(\NN)$ is Abelian.
 
 \item \label{hp:non-degeneracy} 
 There exists a positive constant~$\lambda_0$
 such that $f(y) \geq \lambda_0\dist^2(y, \, \NN)$ for any~$y\in\R^m$.
\end{enumerate}
The assumption~\eqref{hp:N} is consistent with the setting of~\cite{CO1}
and is satisfied, for instance, when~$k=2$ and~$\NN=\SS^1$
(the Ginzburg-Landau case) or~$k=2$ and~$\NN=\PR$ (the Landau-de Gennes case).
The assumption~\eqref{hp:non-degeneracy} is both a non-degeneracy condition
around the minimising set~$\NN$ and a growth condition.

\begin{remark}
 {\BBB We do not expect the assumption~\eqref{hp:non-degeneracy} to be sharp.
 In fact, \eqref{hp:non-degeneracy} may probably be relaxed 
 so as to include potentials that behave as~$\dist^s(\cdot, \, \NN)$,
 for some~$s>2$, in a neighbourhood of~$\NN$.}
\end{remark}

We consider minimisers~$u_{\eps,\min}$ of~\eqref{energy},
subject to the boundary condition~$u = v$ on~$\partial\Omega$. 
On the boundary datum~$v$, we assume
\begin{enumerate}[label=(H\textsubscript{\arabic*}), ref=H\textsubscript{\arabic*}, resume]
 \item \label{hp:v} \label{hp:last}
 $v\in W^{1-1/k,k}(\partial\Omega,\,\NN)$
 --- that is, $v\in W^{1-1/k,k}(\partial\Omega,\, \R^m)$
 and~$v(x)\in\NN$ for~$\H^{n+k-1}$-a.e.~$x\in\partial\Omega$.
\end{enumerate}
Under the assumptions~\eqref{hp:first}--\eqref{hp:last},
the rescaled energy densities
\begin{equation*} 
 \mu_{\eps,\min} := \left( \frac{1}{k}|\nabla u_{\eps,\min}|^k
 + \frac{1}{\eps^k} f(u_{\eps,\min}) \right) \frac{\d x\mres\Omega}{\abs{\log\eps}}
\end{equation*}
have uniformly bounded mass
(see e.g.~Remark~\ref{rk:extension} below; here, 
$\d x \mres\Omega$ denotes the Lebesgue
measure restricted to~$\Omega$).
Up to extraction of a subsequence,
we may assume that $\mu_{\eps,\min}$ converges weakly$^*$ 
(as measures in~$\R^{n+k}$) 
to a non-negative measure~$\mu_{\min}$, as~$\eps\to 0$.
We provide a variational characterisation of~$\mu_{\min}$
in terms of flat chains with coefficients 
in~$(\GN, \, |\cdot|_*)$, where~$|\cdot|_*$
is a suitable norm, defined in Section~\ref{sect:setting} below.
(For instance, in case~$k=2$ and~$\NN=\SS^1$,
$\abs{d}_* = \pi\abs{d}$ for any~$d\in\pi_1(\SS^1)\simeq\Z$.)
We denote the mass of such a flat chain~$S$ by~$\M(S)$,
and the restriction of~$S$ to a set~$E$ by~$S\mres E$. We have

\begin{mainthm} \label{th:mainmin}
 Under the assumptions~\eqref{hp:first}--\eqref{hp:last},
 there exists a finite-mass~$n$-chain~$S_{\min}$,
 with coefficients in~$(\GN,\,|\cdot|_*)$ and 
 support in~$\overline{\Omega}$, such that 
 $\mu_{\min}(E) = \M(S_{\min}\mres E)$
 for any Borel set~$E\subseteq\R^{n+k}$.
 Moreover, $S_{\min}$ minimises the mass in its homology class 
 --- that is, for any~$(n+1)$-chain~$R$
 with coefficients in~$(\GN,\,|\cdot|_*)$ and 
 support in~$\overline{\Omega}$, we have
 \[
  \M(S_{\min}) \leq \M(S_{\min} + \partial R).
 \]
\end{mainthm}

In other words, in the limit as~$\eps\to 0$ the energy of minimisers
concentrates, to leading order, on the support of a flat chain~$S_{\min}$
that solves a homological Plateau problem. 
The homology class of~$S_{\min}$ is uniquely determined by 
the domain~$\Omega$ and the boundary datum~$v$ (that is, 
$S_{\min}$ belongs to the class~$\mathscr{C}(\Omega, \, v)$
defined by~\eqref{C_Omega_v} below). {\BBB We stress that 
Theorem~\ref{th:mainmin} does not require any topological assumption,
such as simply connectedness, on the domain~$\Omega$. However,
the homology class of~$S_{\min}$ does depend on the topology of the domain
and it can be described more easily if~$\Omega$ has a simple topology
(see the examples in Section~\ref{sect:setting} below).
On the other hand, the topological assumption~\eqref{hp:N}
on the manifold~$\NN$ is essential. An analogue of Theorem~\ref{th:mainmin} 
in case~$k=2$ and the fundamental group of~$\NN$ is non-Abelian
would already be of interest in terms of the applications;
manifolds with non-Abelian fundamental group arise quite
naturally, for instance, in materials science (e.g., as a
model for biaxial liquid crystals). Unfortunately,
the very statement of Theorem~\ref{th:mainmin} does not make sense
in the non-Abelian setting, because homology
requires the coefficient group to be Abelian.
Convergence results in case~$n=0$, $k=2$
(see e.g.~\cite{pirla, MonteilRodiacVanSchaftingen})
suggest that the energy concentration set may inherit some
minimality properties, even if~$\pi_1(\NN)$ is non-Abelian.
However, a general convergence result in the non-Abelian setting,
along the lines of Theorem~\ref{th:mainmin}, would presumably 
require some `ad-hoc' tools from Geometric Measure Theory.}

\begin{remark}
 {\BBB Theorem~A characterises the asymptotic behaviour
 of the energy of minimisers, to leading order:
 \[
  E_\eps(u_{\eps,\min}) = \M(S_{\min}) \abs{\log\eps} 
  + \mathrm{o}(\abs{\log\eps}) \qquad\textrm{as } \eps\to 0.
 \]
 In some cases, the next-to-leading order term can be characterised, too.
 For instance, when~$n=0$, $k=2$, the energy concentrates on 
 a finite number of points and the next-to-leading order term in the
 energy expansion is a `renormalised energy' which describes the
 interaction among the singular points. The renormalised energy
 was introduced, in the Ginzburg-Landau setting, 
 by Bethuel, Brezis and H\'elein~\cite{BBH}
 and it was extended very recently by Monteil, Rodiac and 
 Van Schaftingen~\cite{MonteilRodiacVanSchaftingen, 
 MonteilRodiacVanSchaftingen2} to more general functionals.
 This raises the question as to whether a renormalised energy
 may be derived in case~$n = 0$, $k > 2$. 
 A higher-order energy expansion
 for the three-dimensional Ginzburg-Landau functional 
 ($n = 1$, $k=2$, $\NN=\SS^1$) was obtained 
 by Contreras and Jerrard~\cite{ContrerasJerrard},
 in a setting where the energy concentrates on a cluster 
 of `nearly parallel' vortex filaments.}
\end{remark}

We deduce Theorem~\ref{th:mainmin} from our $\Gamma$-convergence
result, Theorem~\ref{th:main} in Section~\ref{sect:setting}.
The proof of the $\Gamma$-lower bound is based on the same 
strategy as in~\cite{ABO2}. However, the construction of a recovery sequence
is rather different from~\cite{ABO2}.
The main building block, Proposition~\ref{prop:dipole_simpl}
in Section~\ref{sect:dipole_simpl}, is inspired by the
``dipole construction'' \cite{BrezisCoronLieb, Bethuel1990, BethuelbrezisCoron}.
Here, dipoles are suitably inserted into a non-constant
and, in fact, singular background.

As an auxiliary result, we prove the following lower energy bound,
which may be of independent interest.

\begin{mainprop} \label{prop:JS}
 Suppose that~\eqref{hp:first}--\eqref{hp:last} hold.
 Let~$\Omega\subseteq\R^k$ be a bounded, Lipschitz domain
 that is homeomorphic to a ball. 
 Then, for any~$u\in W^{1,k}(\Omega,\,\R^m)$ 
 such that $u=v$ on~$\partial\Omega$, there holds
 \[
  E_\eps(u) \geq 
  \abs{\sigma}_*\abs{\log\eps} - C,
 \]
 where~$\sigma\in\GN$ is the homotopy class of~$v$
 and~$C$ is a {\BBB positive} constant that depends only on~$\Omega$, $v$.
\end{mainprop}

If~$\Omega\subseteq\R^k$ is homeomorphic to a ball 
and~$v\in W^{1-1/k,k}(\partial\Omega, \, \NN)$,
the homotopy class of~$v$ can be defined as in~\cite{BN1}.
In the Ginzburg-Landau case, this inequality was proved by Sandier~\cite{Sandier}
(with~$k=2$) and Jerrard~\cite{Jerrard}; for the Landau-de Gennes functional,
see e.g.~\cite{BaumanParkPhillips,pirla3}. The proof of 
Proposition~\ref{prop:JS} in contained in Appendix~\ref{sect:jerrard}
(in fact, a slightly stronger statement is given there).

\begin{remark}
 {\BBB In case~$\sigma=0$, Proposition~\ref{prop:JS} does not provide
 any information. However, there could be critical points of
 the functional~$E_\eps$ whose energy diverges logaritmically
 even if the boundary datum is homotopically trivial. In other words,
 energy concentration may happen not only because of global topological
 contraints, but also for other reasons, such as symmetry.
 See, for instance, \cite{INSZ2020} for an analysis of two-dimensional
 Landau-de Gennes solutions ($n = 0$, $k=2$, $\NN=\PR$). }
\end{remark}

The paper is organised as follows.
In Section~\ref{sect:setting}
we recall some notation from~\cite{CO1}
and we state the main $\Gamma$-convergence result, Theorem~\ref{th:main}.
We prove the $\Gamma$-upper bound first, 
in Section~\ref{sect:limsup}, and give the proof of the
$\Gamma$-lower bound in Section~\ref{sect:liminf}.
Theorem~\ref{th:mainmin} is deduced
from Theorem~\ref{th:main} in Section~\ref{sect:mainmin}.
A series of appendices, with proofs of technical results, completes the paper.

\numberwithin{equation}{section}
\numberwithin{definition}{section}
\numberwithin{theorem}{section}
\numberwithin{remark}{section}
\numberwithin{example}{section}

\section{Setting of the problem and statement of the \texorpdfstring{$\Gamma$}{Gamma}-convergence result}
\label{sect:setting}

Throughout the paper, we will write $A\lesssim B$ as a shorthand for
$A\leq C B$, where $C$ is a positive constant that only depends on~$n$,
$k$, $f$, $\NN$, and~$\Omega$. If~$F\subseteq\R^{n+k}$ is a rectifiable set 
of dimension~$d$ and~$u\in W^{1,k}_{\mathrm{loc}}(\R^{n+k}, \, \R^m)$ we will write
\[
 E_\eps(u, \, F) := \int_{F}
 \left(\frac{1}{k} \abs{\nabla u}^k + \frac{1}{\eps^k} f(u) \right)\d\H^{d} \! .
\]
Additional notation will be set later on.
Throughout the paper, we assume 
that~\eqref{hp:first}--\eqref{hp:last} are satisfied.

\subsection{Choice of the norm on~\texorpdfstring{$\GN$}{the coefficient group}}

Under the assumption~\eqref{hp:N}, the group~$\GN$
is Abelian (and we use additive notation for the group operation).
We recall that a function $|\cdot|\colon\GN\to [0, \, +\infty)$ 
is called a norm if it satisfies the following properties:
\begin{enumerate}[label=(\roman*)]
 \item $|\sigma| = 0$ if and only if~$\sigma=0$
 \item $|-\sigma| = |\sigma|$ for any~$\sigma\in\GN$
 \item $|\sigma_1 + \sigma_2|\leq |\sigma_1|+|\sigma_2|$ for any~$\sigma_1$, $\sigma_2\in\GN$.
\end{enumerate}
As in~\cite{CO1}, we assume that the norm satisfies
\begin{equation} \label{discrete_norm-intro}
  \inf_{\sigma\in\GN\setminus\{0\}} \abs{\sigma} > 0,
\end{equation}
that is, $|\cdot|$ induces the discrete topology on~$\GN$.
\begin{remark}
 We do \emph{not} require that $|n\sigma| = n|\sigma|$ 
 for any~$n\in\N$, $\sigma\in\GN$;
 this is consistent with the theory of flat chains as developed 
 in~\cite{Fleming, White-Rectifiability}. 
\end{remark}

While the results of~\cite{CO1} hold for 
any norm on~$\GN$ that satifies~\eqref{discrete_norm-intro},
Theorem~\ref{th:mainmin} only holds for a specific choice of the norm.
Let us define such a norm, following
the approach in~\cite[Chapter~6]{Chiron}.
A natural attempt, motivated by the analogy with the
functional~\eqref{energy}, is to define
\begin{equation} \label{I_min-intro}
 E_{\min}(\sigma) := \inf\left\{
 \frac{1}{k}\int_{\SS^{k-1}}\abs{\nablaT v}^k
 \colon v\in W^{1,k}(\SS^{k-1}, \, \NN)\cap\sigma \right\} 
\end{equation}
for any~$\sigma\in\GN$.
Here~$\nablaT$ denotes the tangential gradient on~$\SS^{k-1}$,
that is, the restriction of the Euclidean gradient~$\nabla$
to the tangent plane to the sphere. Due to the compact 
embedding~$W^{1,k}(\SS^{k-1}, \NN)\hookrightarrow C(\SS^{k-1}, \, \NN)$,
the set~$W^{1,k}(\SS^{k-1}, \, \NN)\cap\sigma$ is 
sequentially $W^{1,k}$-weakly closed and hence,
the infimum in~\eqref{I_min-intro} is achieved.
However, the function~$E_{\min}$ fails to be a norm, in general, because
it may not satisfy the triangle inequality~(iii). To overcome this issue,
for any~$\sigma\in\GN$ we define 
\begin{equation} \label{group_norm-intro}
 |\sigma|_* := 
   \inf\left\{\sum_{i=1}^q E_{\min}(\sigma_i)\colon q\in\N, \ (\sigma_i)_{i=1}^q\in\GN^q,
   \ \sum_{i=1}^q \sigma_i = \sigma\right\} \!.
\end{equation}

\begin{prop} \label{prop:group_norm}
 The function $|\cdot|_*$ is a norm on~$\GN$ that
 satisfies~\eqref{discrete_norm-intro}
 and~$\abs{\sigma}_*\leq E_{\min}(\sigma)$ for any~$\sigma\in\GN$.
 The infimum in~\eqref{group_norm-intro} is achieved,
 for any~$\sigma\in\GN$. Moreover, the set
 \begin{equation} \label{generators-intro}
  \Sg := \left\{\sigma\in\GN\colon 
  |\sigma|_* = E_{\min}(\sigma) \right\}
 \end{equation}
 is finite, and for any~$\sigma\in\GN$ there exists 
 a decomposition~$\sigma = \sum_{i=1}^q\sigma_i$
 such that $|\sigma|_* = \sum_{i=1}^q|\sigma_i|_*$ 
 and~$\sigma_i\in\Sg$ for any~$i$.
\end{prop}

The proof of this result will be given in Appendix~\ref{sect:norm}.
In case~$\NN=\SS^{k-1}$, the group $\pi_{k-1}(\SS^{k-1})$ 
is isomorphic to~$\mathbb{Z}$, $\Sg = \{-1, \, 0, \, 1\}$, 
and for any~$d\in\Z$ we have
\[
 \abs{d}_* = (k-1)^{k/2}\L^k(B^k_1) \abs{d},
\]
where~$\L^k(B^k_1)$ is the Lebesgue measure of the unit ball in~$\R^k$
and~$\abs{d}$ is the standard absolute value of~$d$
(see Example~\ref{example:sphere}).

\begin{remark}
 {\BBB When~$k=2$, the infimum in~\eqref{I_min-intro} is achieved by a minimising 
 geodesic in the homotopy class~$\sigma$, parametrised by multiples of arc-length.
 As a consequence, $E_{\min}(\sigma)$ is --- up to a multiplicative constant ---
 the length squared of a minimising geodesic in the class~$\sigma$,
 and~$E_{\min}^{1/2}$ is a norm on~$\pi_1(\NN)$. However, $E_{\min}^{1/2}$
 may not coincide with~$|\cdot|_*$, not even up to a multiplicative constant.
 For instance, when~$\NN$ is the flat torus, 
 $\NN = \R^2/(2\pi\Z)^2 = \SS^1\times\SS^1$, we have $\pi_1(\NN) \simeq\Z\times\Z$,
 \[
  E_{\min}^{1/2}(d_1, \, d_2) = \pi^{1/2}\left(d_1^2 + d_2^2\right)^{1/2}
  \qquad \textrm{and} \qquad
  \abs{(d_1, \, d_2)}_* = \pi\left(|d_1| + |d_2|\right)
 \]
 for any~$(d_1, \, d_2)\in\Z\times\Z$. 
 We did not investigate whether, for arbitrary~$k> 2$ and~$\NN$,
 $E_{\min}^{1/k}$ is a norm on~$\GN$.}
\end{remark}

\subsection{Notation for flat chains}

We follow the notation adopted in~\cite[Section~2]{CO1}.
In particular,
we denote by $\F_q(\R^{n+k}; \, \GN)$ the space of flat 
$q$-dimensional chains in~$\R^{n+k}$ 
with coefficients in the normed group~$(\GN, \, |\cdot|_*)$.
We denote the flat norm by~$\F$, and the mass by~$\M$. 
The support of a flat chain~$S$ is denoted by~$\spt S$.
The restriction of~$S$ to a Borel set~$E\subseteq\R^{n+k}$
is denoted $S\mres E$. Given~$f\in C^1(\R^{n+k}, \, \R^{n+k})$,
we write~$f_{*}S$ for the push-forward of~$S$ through~$f$.
(The reader is referred e.g. to \cite{Fleming, White-Rectifiability} 
for the definitions of these objects.) 

Given a domain~$\Omega\subseteq\R^{n+k}$,
we define $\F_q(\overline{\Omega}; \, \GN)$ as the set
of flat chains such that $\spt S\subseteq\overline{\Omega}$.
We also define $\M_q(\overline{\Omega}; \, \GN)$ as the set
of flat chains $S\in\F_q(\overline{\Omega}; \, \GN)$ such that $\M(S)<+\infty$.
We will say that two chains $S_1$,
$S_2\in \M_q(\overline{\Omega}; \, \GN)$ are 
\emph{cobordant in~$\overline{\Omega}$} if and only if there exists 
a finite-mass chain $R\in\M_{q+1}(\overline{\Omega}; \, \GN)$ such that 
\[
 S_2 - S_1 = \partial R.
\]
In this case, we write $S_1\sim_{\overline{\Omega}} S_2$.
The cobordism in~$\overline{\Omega}$ defines an equivalence relation
on the space of finite-mass chains, $\M_q(\overline{\Omega}; \, \GN)$.
Moreover, due to the isoperimetric inequality (see e.g.~\cite[7.6]{FedererFleming}),
cobordism classes are closed with respect to the $\F$-norm.

The group of flat $q$-chains relative to a 
domain~$\Omega\subseteq\R^{n+k}$ is defined 
as the quotient
\[
 \F_q(\Omega; \, \GN) := \F_q(\R^{n+k}; \, \GN)/\{S\in\F_q(\R^{n+k}; \, \GN)\colon
 \spt S \subseteq \R^{n+k}\setminus\Omega\} \! .
\]
To avoid notation, the equivalence class of a chain~$S\in\F_q(\R^{n+k}; \, \GN)$
will still be denoted by~$S$.
The quotient norm may equivalently be rewritten as
\begin{equation} \label{flat_relative}
 \begin{split}
  \F_\Omega(S) &= \inf\big\{\M(P\mres\Omega) + \M(Q\mres\Omega)
  \colon P\in\F_{q+1}(\R^{n+k}; \, \GN), \\
  &\qquad Q\in\F_q(\R^{n+k}; \, \GN), \ 
  \spt(S - \partial P - Q)\subseteq\R^{n+k}\setminus\Omega\big\} 
 \end{split}
\end{equation} 
(see \cite[Section~2.1]{CO1}).

For any~$S\in \F_n(\Omega; \, \GN)$ and~$R\in\F_k(\R^{n+k}; \, \Z)$
such that~$\M(R) + \M(\partial R)<+\infty$, $\spt R \subseteq\Omega$,
and $\spt(\partial S)\cap \spt R = \spt S\cap\spt(\partial R) = \emptyset$,
we denote the intersection index of~$S$ and~$R$ 
{\BBB (as defined in~\cite[Section~2.1]{CO1})} by~$\I(S, \, R)\in\GN$.
{\BBB For instance, if~$S$ is carried by a $n$-polyhedron
with constant multiplicity~$\sigma\in\GN$, $R$ is carried by
a~$k$-polyhedron with unit multiplicity and (the supports of) 
$S$, $R$ intersect transversally, then~$\I(S, \, R) = \pm\sigma$, 
where the sign depends on the relative orientation of~$S$ and~$R$.
The intersection index~$\I$ is a bilinear pairing and satisfies
suitable continuity properties (see e.g.~\cite[Lemma~8]{CO1}).}

\subsection{The topological singular set}

In~\cite{CO1}, we constructed the topological singular set, $\S_y(u)$,
for~$u\in (L^\infty\cap W^{1,k-1})(\Omega, \, \R^m)$ and~$y\in\R^m$. 
Here, we introduce a variant of that construction and define~$\S_y(u)$ in
case~$u\in W^{1,k}(\Omega, \, \R^m)$, without assuming that~$u\in L^\infty(\Omega, \, \R^m)$.
In both cases, the operator~$\S_y(u)$ generalises the Jacobian determinant of~$u$
--- and indeed, the Jacobian of~$u\colon\R^k\to\R^k$ is well-defined in a distributional sense
if~$u\in (L^\infty\cap W^{1,k-1})(\R^k, \, \R^k)$,
and in a pointwise sense if $u\in W^{1,k}(\R^k, \, \R^k)$.
The starting point of the construction is the following topological property.

\begin{prop}[\cite{HardtLin-Minimizing}]\label{prop:X}
 Under the assumption~\eqref{hp:N}, there
 exist a compact, \emph{polyhedral} complex~$\X\subseteq\R^m$ 
 of dimension~$m-k$ and a smooth map 
 $\RR\colon\R^m\setminus\X\to\NN$ such that $\RR(z) = z$ for any~$z\in\NN$, and
 \[
  |\nabla\RR(z)| \leq \frac{C}{\dist(z, \, \X)}
 \]
 for any~$z\in\R^m\setminus\X$ and some constant~$C=C(\NN, \, m, \, \X)>0$.
\end{prop}
This result, or variants thereof, was proved in
\cite[Lemma~6.1]{HardtLin-Minimizing}, 
\cite[Proposition~2.1]{BousquetPonceVanSchaftingen}, \cite[Lemma~4.5]{Hopper}.
\emph{While in our previous paper~\cite{CO1} we required~$\X$ to be a smooth complex, 
in this paper we require~$\X$ to be polyhedral,} 
because this will simplify some technical points in the proofs.

Let us fix once and for all a polyhedral complex~$\X$ and a map~$\RR$,
as in Proposition~\ref{prop:X}. Let~$\delta^*\in (0, \, \dist(\NN, \, \X))$
be fixed, and let $B^* := B^m(0, \, \delta^*)\subseteq\R^m$.
Let*
\begin{equation*} 
 Y := L^1(B^*, \, \F_{n}(\Omega; \, \GN)), \qquad
 \overline{Y} := L^1(B^*, \, \F_{n}(\overline{\Omega}; \, \GN))
\end{equation*}
be the set of Lebesgue-measurable
maps~$S\colon B^*\to\F_{n}(\Omega; \, \GN)$, respectively
$S\colon B^*\to\F_{n}(\overline{\Omega}; \, \GN)$
(we use the notation $y\in B^*\mapsto S_y$ in both cases), such that
\[
 \norm{S}_{Y} := \int_{B^*} \F_{\Omega}(S_y) \, \d y < +\infty,
 \ \textrm{ respectively } \
 \norm{S}_{\overline{Y}} := \int_{B^*} \F(S_y) \, \d y < +\infty.
\]
The sets~$Y$, $\overline{Y}$ are complete normed moduli,
with the norms~$\|\cdot\|_{Y}$, $\|\cdot\|_{\overline{Y}}$ respectively. 
The space $\F_{n}(\Omega; \, \GN)$, respectively~$\F_{n}(\overline{\Omega}; \, \GN)$,
embeds canonically into~$Y$, respectively~$\overline{Y}$.
If need be, we will identify a chain~$S\in\F_{n}(\overline{\Omega}; \, \GN)$
with an element of~$\overline{Y}$, i.e. the constant map~$y\mapsto S$.

By~\cite[Theorem~3.1]{CO1}, there exists a unique 
operator $\S\colon (L^\infty\cap W^{1,k-1})(\Omega, \, \R^m)\to Y$
that is continuous (if~$u_j\to u$ strongly in~$W^{1,k-1}(\Omega)$
and $\sup_j\|u_j\|_{L^\infty(\Omega)}<+\infty$, then $\S(u_j)\to\S(u)$ in~$Y$) 
and satisfies
\begin{enumerate}[label=(P\textsubscript{\arabic*}), ref=P\textsubscript{\arabic*}]
 \setcounter{enumi}{-1}
 \item \label{S:intersection} 
 for any smooth~$u$, a.e.~$y\in B^*$ and any~$R\in\F_{k}(\R^{n+k}; \, \Z)$
 such that $\M(R)+\M(\partial R)<+\infty$,
 $\spt(R)\subseteq\Omega$, $\spt(\partial R)\subseteq\Omega\setminus\spt\S_y(u)$, there holds
 \[
  \I(\S_y(u), \, R) = 
  \textrm{homotopy class of } \RR\circ(u - y) \textrm{ on } \partial R.
 \]
\end{enumerate}
{\BBB We recall that~$\I$ denotes the intersection index, 
defined as in~\cite[Section~2.1]{CO1}.}

\begin{prop} \label{prop:S}
 There exists a (unique) continuous 
 operator~$\overline{\S}\colon W^{1,k}(\Omega, \, \R^m)\to \overline{Y}$ 
 that satisfies~\eqref{S:intersection} and the 
 following properties:
 \begin{enumerate}[label=\emph{(P\textsubscript{\arabic*})},
 ref=P\textsubscript{\arabic*}]
  \item \label{S:Sbar} 
  For any~$u\in (L^\infty\cap W^{1,k})(\Omega, \, \R^m)$ and a.e~$y\in B^*$,
  $\overline{\S}_y(u) = \S_y(u)$ --- more precisely, the chain~$\overline{\S}_y(u)$
  belongs to the equivalence class~$\S_y(u)\in\F_{n}(\Omega; \, \GN)$.
  
  \item \label{S:mass} For any~$u\in W^{1,k}(\Omega, \, \R^m)$
  and any Borel subset~$E\subseteq\overline{\Omega}$, there holds
  \[
   \int_{B^*} \M(\overline{\S}_y(u)\mres E) \,\d y \lesssim \int_{E} \abs{\nabla u}^k.
  \]
  
  \item \label{S:cobord} If~$u_0$, $u_1\in W^{1,k}(\Omega, \, \R^m)$
  are such that $u_{0|\partial\Omega} = 
  u_{1|\partial\Omega}\in W^{1-1/k, k}(\partial\Omega, \, \NN)$
  (in the sense of traces), then
  $\overline{\S}_{y_0}(u_0) \sim_{\overline{\Omega}} \overline{\S}_{y_1}(u_1)$ for a.e.~$y_0$, $y_1\in B^*$.
\end{enumerate}
\end{prop}

The proof of Proposition~\ref{prop:S} will be given in Apprendix~\ref{appendix:S}.
Taking account of~\eqref{S:Sbar}, we abuse of notation
and write~$\S$ instead of~$\overline{\S}$ from now on.
As a consequence of~\eqref{S:cobord}, for any
boundary datum~$v\in W^{1-1/k, k}(\partial\Omega, \, \NN)$
there exists a unique cobordism 
class~$\mathscr{C}(\Omega, \, v)\subseteq\M_n(\overline{\Omega}; \, \GN)$
such that
\begin{equation} \label{C_Omega_v}
 \S_y(u)\in\mathscr{C}(\Omega, \, v) \quad
 \textrm{ for any } u\in W^{1,k}(\Omega, \, \R^m) \textrm{ with trace } v
 \textrm{ on } \partial\Omega
 \textrm{ and a.e. } y\in B^*.
\end{equation}

\subsection{The \texorpdfstring{$\Gamma$}{Gamma}-convergence result}

The main result of this paper is a generalisation of~\cite[Theorem~5.5]{ABO2}.
We let $W^{1,k}_v(\Omega, \, \R^m)$ denote the set of maps
$u\in  W^{1,k}(\Omega, \, \R^m)$ such that $u = v$ on~$\partial\Omega$
(in the sense of traces).

\begin{mainthm} \label{th:main}
 Suppose that the assumptions~\eqref{hp:first}--\eqref{hp:last}
 are satisfied. Then, the following properties hold.
 \begin{enumerate}[label=(\roman*)]
  \item \emph{Compactness and lower bound.}
  Let $(u_\eps)_{\eps > 0}$ be a sequence 
  in~$W^{1,k}_v(\Omega, \, \R^m)$ that satisfies
  $\sup_{\eps>0}\abs{\log\eps}^{-1} E_\eps(u_\eps) < +\infty$.
  Then, there exists a (non relabelled) countable
  subsequence and a finite-mass chain
  $S\in\mathscr{C}(\Omega, \, v)$ such that~$\S(u_\eps)\to S$ in~$\overline{Y}$
  and, for any open subset~$A\subseteq\R^{n+k}$,
  \[
   \M(S\mres A) \leq \liminf_{\eps\to 0}
   \frac{E_\eps(u_\eps, \, A\cap\Omega)}{\abs{\log\eps}} .
  \]
  
  \item \emph{Upper bound.} For any finite-mass chain $S\in\mathscr{C}(\Omega, \, v)$,
  there exists a sequence $(u_\eps)$ in~$W^{1,k}_v(\Omega, \, \R^m)$ 
  such that~$\S(u_\eps)\to S$ in~$\overline{Y}$ and
  \[
   \limsup_{\eps\to 0} \frac{E_\eps(u_\eps)}{\abs{\log\eps}}
   \leq 
   \M(S).
  \]
 \end{enumerate}
\end{mainthm}

Theorem~\ref{th:mainmin} follows almost immediately from
Theorem~\ref{th:main}, combined with general properties of
the $\Gamma$-convergence and standard facts in measure theory.
There is a variant of Theorem~\ref{th:main} for the problem with no boundary conditions,
which is analogue to~\cite[Theorem~1.1]{ABO2}.
We will say that a chain~$S$ is a finite-mass, $n$-dimensional relative boundary
if it has form~$S = (\partial R)\mres\Omega$, where~$R\in\M_{n+1}(\R^{n+k}; \, \GN)$
is such that~$\M(\partial R)<+\infty$.

\begin{mainprop} \label{prop:main-nobd}
 Suppose that the assumptions~\eqref{hp:first}--\eqref{hp:non-degeneracy}
 are satisfied. Then, the following properties hold.
 \begin{enumerate}[label=(\roman*)]
  \item \emph{Compactness and lower bound.}
  Let $(u_\eps)_{\eps > 0}$ be a sequence 
  in~$W^{1,k}(\Omega, \, \R^m)$ that satisfies
  $\sup_{\eps>0}\abs{\log\eps}^{-1} E_\eps(u_\eps) < +\infty$.
  Then, there exists a (non relabelled) countable
  subsequence and a finite-mass, $n$-dimensional relative 
  boundary~$S$ such that~$\S(u_\eps)\to S$ in~$Y$
  and, for any open subset~$A\subseteq\Omega$,
  \[
   \M(S\mres A) \leq \liminf_{\eps\to 0}
   \frac{E_\eps(u_\eps, \, A\cap\Omega)}{\abs{\log\eps}} .
  \]
  
  \item \emph{Upper bound.} For any finite-mass, $n$-dimensional
  relative boundary~$S$,
  there exists a sequence $(u_\eps)$ in~$W^{1,k}(\Omega, \, \R^m)$ 
  such that~$\S(u_\eps)\to S$ in~$Y$ and
  \[
   \limsup_{\eps\to 0} \frac{E_\eps(u_\eps)}{\abs{\log\eps}}
   \leq \M(S).
  \]
 \end{enumerate}
\end{mainprop}

Proposition~\ref{prop:main-nobd} is not quite informative as it stands, 
because minimisers of the functional~\eqref{energy} under no boundary conditions
are constant. However, since~$\Gamma$-convergence is stable with respect to continuous perturbations,
Proposition~\ref{prop:main-nobd} can be extended to non-trivial minimisation problems
with lower-order terms or under integral constraints, as long as these are
compatible with the topology of $\Gamma$-convergence.

\subsection{A few examples}

{\BBB We illustrate our results by means of a few simple examples. 
If~$A\subseteq\R^{n+k}$ is an~$n$-dimensional polyhedral (or smooth) set,
with a given orientation, the unit-multiplicity chain
carried by~$A$ will be denoted
$\llbracket A\rrbracket\in\M_{n}(\R^{n+k}; \, \Z)$.}

\begin{example} 
{\BBB First, we suppose the domain is the unit ball in the critical dimension,
i.e.~$n=0$ and~$\Omega = B^k$, and consider the 
target~$\NN=\SS^{k-1}\subseteq\R^k$. We need to 
identify the class~$\mathscr{C}(\Omega, \, v)$ defined by~\eqref{C_Omega_v}.
For simplicity, suppose that the boundary datum~$v\colon\partial B^k\to\SS^{k-1}$
is smooth, of degree~$d$.
(General data $v\in W^{1-1/k,k}(\partial B^k, \, \SS^{k-1})$
could also be considered, by appealing to Brezis and Nirenberg's 
theory of the degree in~VMO, \cite{BN1}). 
Let~$u\colon B^k\to\R^k$ be any smooth extension of~$v$.
Let~$y\in\R^k$ be a regular value for~$u$ (i.e., $\det\nabla u(x) \neq 0$
for any~$x\in u^{-1}(y)$) such that~$\abs{y} <1$. Then,
the inverse image~$u^{-1}(y)$ consists of a finite number points.
Let~$r>0$ be a sufficiently small radius. 
By definition of~$\S$, we have
\begin{equation*} 
 \S_y(u) = \sum_{x\in u^{-1}(y)} d(x)
 \llbracket x\rrbracket\in\M_0(\bar{B}^k; \, \Z),
\end{equation*}
where~$d(x)$ is the degree of the
map~$(u - y)/|u - y|\colon\partial B_r(x)\to\SS^{k-1}$.
The class~$\mathscr{C}(\Omega, \, v)$ consists of all and only the chains
that differ from~$\S_y(u)$ by a boundary. It is not difficult
to characterise~$\mathscr{C}(\Omega, \, v)$ using
the following topological property,
which holds true for any (normed, Abelian) coefficient group~$\G$
and any connected, open set~$D\subseteq\R^d$.

\medskip
\noindent
\textbf{Fact.}
 \emph{Let~$T$ be a $0$-chain of the 
 form~$T = \sum_{i=1}^q\sigma_j\llbracket z_i\rrbracket$,
 for~$z_j\in \bar{D}$, $\sigma_j\in\G$.
 Then, there exists~$R\in\M_1(\bar{D}; \, \G)$
 such that~$\partial R = T$ if and only if~$\sum_{j=1}^q \sigma_j = 0$.}
 
\medskip
For a proof of this fact, see e.g.~\cite[Proposition~2.7]{Hatcher}.
Now, Brouwer's theory of the degree (or Property~\eqref{S:intersection} above)
implies that
\begin{equation*} 
 \sum_{x\in u^{-1}(y)}
 d(x) = \sum_{x\in u^{-1}(y)} \sign(\det\nabla u(x)) = d,
\end{equation*}
therefore
\[
 \mathscr{C}(\Omega, \, v)
 = \left\{\sum_{j = 1}^q \sigma_j\llbracket z_j\rrbracket \colon 
 q\in\N, \ (\sigma_j)_{j=1}^q\in\Z^q, 
 \ (z_j)_{j=1}^q\in (\bar{B}^k)^q, \ \sum_{j=1}^q \sigma_j = d\right\} \! .
\]
In agreement with the Ginzburg-Landau theory,
mass-minimising chains in~$\mathscr{C}(\Omega, \, v)$
consist of exactly~$|d|$ points, with multiplicities equal to~$1$ or~$-1$
according to the sign of~$d$. This argument extends to more general 
manifolds~$\NN$, with no essentially change; we obtain
\[
 \mathscr{C}(\Omega, \, v)
 = \left\{\sum_{j = 1}^q \sigma_j\llbracket z_j\rrbracket \colon 
 q\in\N, \ (\sigma_j)_{j=1}^q\in(\GN)^q, 
 \ (z_j)_{j=1}^q\in (\bar{B}^k)^q, \ \sum_{j=1}^q \sigma_j = \sigma\right\} \! ,
\]
where~$\sigma\in\GN$ is the homotopy class of
the boundary datum~$v\colon\partial B^k\to\NN$.
Mass-minimising chains in~$\mathscr{C}(\Omega, \, v)$
have the form~$\sum_{j = 1}^q \sigma_j\llbracket z_j\rrbracket$,
where the multiplicities~$\sigma_j$ belong to the set~$\mathfrak{S}$
defined in~\eqref{generators-intro} and satisfy
$\sum_{j = 1}^q E_{\min}(\sigma_j) = \abs{\sigma}_*$.}
\end{example}

\begin{example}
{\BBB Next, we discuss the case~$n=1$, $\Omega = B^{k+1}$.
Suppose that the boundary datum~$v\colon \partial B^{k+1}\to\NN$
is smooth, except for finitely many isolated singularities
at the points~$x_1$, \ldots, $x_p$. Let~$D_1$, \ldots, $D_p$
be pairwise-disjoint closed geodesic disks in~$\partial B^{k+1}$,
centred at the points~$x_1$, \ldots, $x_p$. Each~$D_i$
is given the orientation induced by the outward-pointing unit normal to~$B^{k+1}$.
Using orientation-preserving coordinate charts, we may identify
$v_{|\partial D_i}\colon \partial D_i\to\NN$ with a map~$\SS^{k-1}\to\NN$;
the homotopy class of the latter is an element of~$\GN$, 
which we denote~$\sigma_i$. The coefficents~$\sigma_i$ must satisfy the
topological constraint
\begin{equation} \label{example2,1}
 \sum_{i=1}^p \sigma_i = 0.
\end{equation}
Indeed, let~$D^+\subseteq\partial B^{k+1}$ be a small geodesic disk
that does not contain any singular point~$x_i$,
and let~$D^- := \partial B^{k+1}\setminus D^+$. Topologically, 
$D^-$ is a disk which contains all the singular points of~$v$;
therefore, the homotopy class of~$v$ restricted
to~$\partial D^-$ is the sum of all the~$\sigma_i$'s above.
However, the homotopy class of~$v$ on~$\partial D^+$
must be trivial, because~$v$ is smooth in~$D^+$. 
Thus, \eqref{example2,1} follows.

We consider the chain
\[
 \S^{\mathrm{bd}}(v) := \sum_{i=1}^p \sigma_i \llbracket x_i\rrbracket
 \in\M_0(\partial\Omega; \, \GN).
\]
Thanks to~\eqref{example2,1}, $\S^{\mathrm{bd}}(v)$ is the boundary
of some~$1$-chain supported in~$\bar{B}^{k+1}$.
More precisely, let~$u\in W^{1,k}(B^{k+1}, \, \R^m)$ be any extension of~$v$.
The results of~\cite{CO1} (see, in particular, Proposition~1,
Proposition~3 and Lemma~18) imply that
\[
 \partial \S_y(u) = \S^{\mathrm{bd}}(v)
\]
for a.e.~$y\in\R^m$ of norm small enough. 
Chains in the same homology class have the same boundary;
therefore, for any chain~$T\in\mathscr{C}(\Omega, \, v)$, there holds
$\partial T = \S^{\mathrm{bd}}(v)$. Conversely, two chains in~$\bar{B}^{k+1}$
that have the same boundary belong to same homology class 
(relative to~$\bar{B}^{k+1}$), because the domain~$\bar{B}^{k+1}$
is contractible. As a consequence, we have
\begin{equation}\label{example2}
 \mathscr{C}(\Omega, \, v) = \left\{T\in \M_{1}(\overline{\Omega}; \, \GN)\colon 
 \partial T= \S^{\mathrm{bd}}(v)\right\} \! .
\end{equation}
In particular, mass-minimising chains in~$\mathscr{C}(\Omega, \, v)$
will be carried by a finite union of segments, connecting
the singularities of the boundary datum according to
their multiplicities. In case~$\NN=\SS^{k-1}$, 
such union of segments realises a `minimising connection', in the 
sense of Brezis, Coron and Lieb~\cite{BrezisCoronLieb}.
For~$k=2$ and~$\NN=\PR$, the condition~\eqref{example2,1}
implies that~$v$ has an \emph{even} number of non-orientable singularities;
mass-minimising chains connect the non-orientable singularities in pairs.

The characterisation~\eqref{example2} extends to 
general data~$v\in W^{1-1/k,k}(\partial B^{k+1}, \, \NN)$,
provided that we define~$\S^{\mathrm{bd}}(v)$ in a suitable way
(see~\cite[Section~3]{CO1}). It also extend to more general 
domains~$\Omega\subseteq\R^{n+k}$, \emph{so long as} the $n$-th 
homology group~$H_n(\Omega; \, \GN)$ is trivial.}
\end{example}

\begin{example}
 {\BBB If the domain has a non-trivial topology, then 
 $\mathscr{C}(\Omega, \, v)$ may contain non-trivial chains
 even if the boundary datum is smooth. For instance, take~$n=1$,
 $k=2$, $\NN=\SS^1$. Let~$\Omega\subseteq\R^3$ be a solid torus of revolution,
 defined as the image of the map~$\Psi\colon B^2\times\R\to\R^3$,
 \[
  \Psi(x, \, \theta) := 
  \left((x_1 + 2) \cos\theta, \, (x_1 + 2)\sin\theta, \, x_2 \right)
  \qquad \textrm{for } x = (x_1, \, x_2)\in B^2, \ \theta\in\R.
 \]
 We consider the smooth map~$u\colon\Omega\to\R^2$ given by
 $u(\Psi(x, \, \theta)) := x$ for~$(x, \, \theta)\in B^2\times\R$.
 The trace of~$u$ at the boudary, $v$, takes its values in~$\SS^1$
 and its restriction on each meridian curve of the torus~$\partial\Omega$
 has degree~$1$. Therefore, $\mathscr{C}(\Omega, \, v)$ is the homology class
 of~$\llbracket u^{-1}(0) \rrbracket
 \in\M_1(\overline{\Omega}; \, \Z)$,
 where~$u^{-1}(0)$ is the zero-set of~$u$ (i.e.~the
 circle~$\Psi(\{(0, \, 0)\}\times\R)$)
 with the orientation induced by~$\Psi$.
 The elements of~$\mathscr{C}(\Omega, \, v)$ can be characterised
 by means of the intersection index~$\I$. More precisely, 
 let~$D$ be the closure of~$\Psi(B^2\times\{0\})$. 
 $D$ is a $2$-disk in the plane orthogonal to~$(0, \, 1, \, 0)$;
 we give~$D$ the orientation induced by~$(0, \, 1, \, 0)$.
 By the Poincar\'e-Lefschetz duality (see 
 e.g.~\cite[Theorem~3, p.~631]{GiaquintaModicaSoucek-I}),
 for any~$T\in\M_{1}(\overline{\Omega}; \, \Z)$ we have
 \[
  T\in\mathscr{C}(\Omega, \, v) 
  \qquad \textrm{ if and only if } \qquad
  \partial T = 0 \ \textrm{ and } \ \I(T, \, \llbracket D\rrbracket) = 1.
 \]
 By a slicing argument, we deduce that the (unique) 
 mass-minimising chain~$S_{\min}$
 in~$\mathscr{C}(\Omega, \, v)$ is carried by an
 equator of~$\partial\Omega$:
 \[
  S_{\min} := \llbracket \Psi(\{(-1, \, 0)\}\times\R) \rrbracket,
 \]
 with the orientation induced by~$\Psi$. (See, e.g., \cite[Section~5.4]{pirla3}
 for a similar example, in case~$\NN=\PR$.)}
\end{example}

\section{Upper bounds}
\label{sect:limsup}

\subsection{Notations and sketch of the construction}

We say that a map $u\colon\Omega\to\R^m$ is \emph{locally piecewise affine}
if~$u$ is continuous in~$\Omega$ and, for any polyhedral set~$K\csubset\Omega$,
the restriction~$u_{|K}$ is piecewise affine.
A set~$P\subseteq\Omega$ is called \emph{locally $n$-polyhedral} if, for any
compact set~$K\subseteq\Omega$, there exists a finite union~$Q$ of 
convex, compact, $n$-dimensional polyhedra such that $P\cap K = Q\cap K$.
In a similar way, we say that a finite-mass chain $S\in\M_n(\overline{\Omega}; \, \GN)$ is 
\emph{locally polyhedral} if, for any compact set 
$K\subseteq\Omega$, there exists a polyhedral 
chain~$T$ such that $(S - T)\mres K = 0$.
If~$M$ is a polyhedral complex and~$j\geq 0$ is an integer,
we denote by~$M_j$ the $j$-skeleton of~$M$, i.e. the union of all its 
faces of dimension less than or equal to~$j$. We set~$M_{-1} := \emptyset$.

\paragraph*{Maps with nice and~$\eta$-minimal singularities.}
To construct a recovery sequence, we will work with 
$\NN$-valued maps with well-behaved singularities,
in a sense that is made precise by the definition below.
Let~$M$, $S$ be polyhedral sets in~$\R^{n+k}$ of 
dimension~$n$, $n-1$ respectively, 
and let~$u\colon\Omega\subseteq\R^{n+k}\to\R^m$.

\begin{definition}[\cite{ABO1, ABO2}]
 We say that~$u$ has a \emph{nice singularity at~$M$} if~$u$ is locally Lipschitz
 on~$\overline{\Omega}\setminus M$ and there exists a constant~$C$ such that
 \[
  \abs{\nabla u(x)} \leq C\dist^{-1}(x, \, M)
  \qquad \textrm{for a.e. } x\in\Omega\setminus M.
 \]
 We say that~$u$ has a \emph{nice singularity at~$(M, \, S)$} if~$u$ is locally
 Lipschitz on~$\overline{\Omega}\setminus(M\cup S)$ and, for any~$p>1$,
 there is a constant~$C_p$ such that
 \[
  \abs{\nabla u(x)} \leq C_p\left(\dist^{-1}(x, \, M)
  + \dist^{-p}(x, \, S)\right)
  \quad \textrm{for a.e. } x\in\Omega\setminus (M\cup S).
 \]
 We say that~$u$ has a \emph{locally nice singularity at~$M$}
 (respectively, at $(M, \, S)$)
 if, for any open subset~$W\csubset\Omega$,
 the restriction~$u_{|W}$
 has a nice singularity at~$M$ (respectively, at~$(M, \, S)$).
\end{definition}

\begin{remark}\label{remark:nice_sing}
 If~$u$ has a nice singularity at~$(M, \, S)$ then 
 $u\in W^{1,k-1}(\Omega, \R^m)$, since both~$M$ and~$S$
 have codimension strictly larger than~$k-1$ 
 (see e.g.~\cite[Lemma~8.3]{ABO2} for more details).
 In particular, if~$u\colon\Omega\to\NN$ has a nice 
 singularity at~$(M, \, S)$, then~$\S_y(u)\in\F_{n}(\Omega; \, \GN)$
 is well-defined for a.e.~$y\in B^*$. 
 Actually, $\S_{y_1}(u) = \S_{y_2}(u)$ for a.e.~$y_1$, $y_2\in B^*$
 \cite[Proposition~3]{CO1}, and we will write
 $\S(u) := \S_{y_1}(u) = \S_{y_2}(u)$. The chain~$\S(u)$ is
 supported on~$M$, and its multiplicities coincide with 
 the homotopy class of $u$ around each $n$-face of~$M$ 
 (see \cite[Lemma~18]{CO1}). 
\end{remark}

Throughout Section~\ref{sect:limsup}, we will work with 
maps with nice (or locally nice) singularities.
However, in order to obtain sharp energy estimates,
we will need to impose a further restriction
on the behaviour of our maps near the singularities.
Let~$u\colon \Omega\to\NN$ be a map with nice singularity at~$(M, \, S)$,
where~$M$, $S$ are polyhedral sets of dimension~$n$, $n-1$ respectively.
We triangulate~$M$, i.e.~we write~$M$ as a finite union of closed simplices
such that, if~$K^\prime$, $K$ are simplices with~$K\neq K^\prime$,
$K\cap K^\prime\neq\emptyset$, then~$K\cap K^\prime$ is a 
boundary face of both~$K$ and~$K^\prime$.
Let~$K\subseteq M$ be a $n$-dimensional
simplex of the triangulation, and let $K^\perp$ be 
the $k$-plane orthogonal to~$K$ through the origin.
Given positive parameters~$\delta$, $\gamma$,
we define the set
\begin{equation} \label{U-diamond}
 U(K, \, \delta, \, \gamma) := 
 \left\{x^\prime + x^{\prime\prime}\colon
 x^\prime\in K, \ x^{\prime\prime}\in K^\perp, \ 
 |x^{\prime\prime}|\leq \min\left(\delta, \,
 \gamma\dist(x^\prime, \, \partial K)\right)\right\}
\end{equation}
(see Figure~\ref{fig:diamond}).
We will identify each~$x\in U(K, \, \delta, \, \gamma)$
with a pair~$x = (x^\prime, \, x^{\prime\prime})$,
where~$x^\prime$, $x^{\prime\prime}$ are as in~\eqref{U-diamond}.
By choosing~$\delta$, $\gamma$ small enough (uniformly in~$K$),
we can make sure that the sets~$U(K, \, \delta, \, \gamma)$
have pairwise disjoint interiors.

\begin{figure}[t]
	\centering
    \begin{subfigure}{.48\textwidth}
     \includegraphics[height=.35\textheight,angle=90]{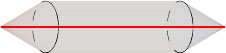}
    \end{subfigure}
    \begin{subfigure}{.48\textwidth}
     \includegraphics[height=.18\textheight]{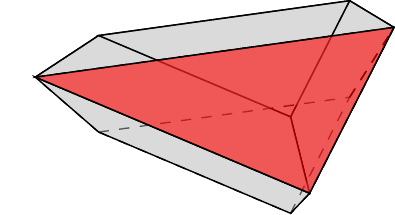}
    \end{subfigure}
	\caption{The set~$U(K, \, \delta, \, \gamma)$, in case 
	$n = 1$, $k=2$ (left) and~$n = 2$, $k=1$ (right). 
	In both cases, the polyhedron~$K$ is in red.}
	\label{fig:diamond}
\end{figure}

\begin{definition} \label{def:minimal}
 Let~$u\colon\Omega\to\NN$ be a map with nice 
 singularity at~$(M, \, S)$, and let~$\eta>0$.
 We say that~$u$ is \emph{$\eta$-minimal} 
 if there exist positive numbers~$\delta$, $\gamma$, 
 a triangulation of~$M$ and, for any~$n$-simplex~$K$ of the triangulation, 
 a Lipschitz map~$\phi_K\colon\SS^{k-1}\to\NN$
 that satisfy the following properties.
 \begin{enumerate}[label=(\roman*)]
  \item If~$K\subseteq M$, $K^\prime\subseteq M$ 
  are $n$-simplices with~$K\neq K^\prime$,
  then $U(K, \, \delta, \, \gamma)$ and~$U(K^\prime, \, \delta, \, \gamma)$
  have disjoint interiors.
  
  \item For any $n$-dimensional simplex $K\subseteq M$
  and a.e.~$x = (x^\prime, \, x^{\prime\prime})\in U(K, \, \delta, \, \gamma)$,
  we have $u(x) = \phi_K(x^{\prime\prime}/|x^{\prime\prime}|)$.
  
  \item For any $n$-dimensional simplex $K\subseteq M$
  and any map~$\zeta\in W^{1,k}(\SS^{k-1}, \, \NN)$
  that is homotopic to~$\phi_K$, we have
  \begin{equation*} 
   \int_{\SS^{k-1}} \abs{\nablaT\phi_K}^k \d\H^{k-1}
   \leq \int_{\SS^{k-1}} \abs{\nablaT\zeta}^k \d\H^{k-1} + \eta.
  \end{equation*}
 \end{enumerate}
\end{definition}
The operator~$\nablaT$ is the tangential gradient on~$\SS^{k-1}$,
i.e.~the restriction of the Euclidean gradient~$\nabla$
to the tangent plane to the sphere.

\begin{remark} \label{remark:phi-sigma}
 Thanks to the Sobolev embedding $W^{1,k}(\SS^{k-1}, \,\NN)
 \hookrightarrow C(\SS^{k-1}, \, \NN)$,
 smooth maps are dense in~$W^{1,k}(\SS^{k-1}, \, \NN)$. 
 Therefore, for any~$\eta>0$
 and any homotopy class~$\sigma\in\pi_{k-1}(\NN)$, 
 there exists a smooth map~$\phi\colon\SS^{k-1}\to\NN$
 in the homotopy class~$\sigma$ that satisfies
 \begin{equation} \label{phi-sigma}
   \int_{\SS^{k-1}} \abs{\nablaT\phi}^k \d\H^{k-1}
   \leq \int_{\SS^{k-1}} \abs{\nablaT\zeta}^k \d\H^{k-1} + \eta
 \end{equation}
 for any~$\zeta\in W^{1,k}(\SS^{k-1}, \, \NN)\cap\sigma$.
\end{remark}

\begin{remark}
 It is possible to find $C^1$-maps that satisfy
 a stronger version of~\eqref{phi-sigma}, with~$\eta=0$.
 Indeed, the compact Sobolev emebedding $W^{1,k}(\SS^{k-1}, \,\NN)
 \hookrightarrow C(\SS^{k-1}, \, \NN)$ implies that homotopy classes
 of maps~$\SS^{k-1}\to\NN$ are sequentially closed with respect
 to the weak~$W^{1,k}$-convergence.
 Then, for each homotopy class~$\sigma\in\GN$,
 there exists a map~$\phi_\sigma$ the minimises the 
 $L^{k}$-norm of the gradient in~$\sigma$.
 The map~$\phi_\sigma$ solves the $k$-harmonic map equation
 and, by Sobolev embedding, is continuous.
 Then, regularity results for $k$-harmonic maps
 (e.g.~\cite[Proposition~5.4]{diestrover-JDE}) imply
 that~$\phi_\sigma\in C^{1,\alpha}(\SS^{k-1}, \, \NN)$. However, the weaker
 condition~\eqref{phi-sigma} is enough for our purposes.
\end{remark}

\paragraph*{Construction of a recovery sequence: a sketch.}

In most of this section, we focus on the proof of Theorem~\ref{th:main}.(ii),
i.e.~we study the problem in the presence of boundary conditions;
only at the end of section, we present the
proof of Proposition~\ref{prop:main-nobd}.(ii).
As in~\cite{ABO2}, in order to define a recovery sequence,
we first construct a map~$w\colon\Omega\to\NN$
with (locally) nice singularity and prescribed singular set~$\S(w) = S$.
However, $w$ must also satisfy the boundary condition, 
$w = v$ on~$\partial\Omega$, where~$v\in W^{1-1/k,k}(\partial\Omega, \, \NN)$
is a datum. This boundary condition makes the construction
of~$w$ substantially harder.
For such a~$w$ to exists, we need a topological assumption on~$S$,
namely, that $S$ belongs to the homology class~\eqref{C_Omega_v}
determined by~$\Omega$ and~$v$.
Our approach is rather different from that of~\cite[Theorem~5.3]{ABO2}.
In~\cite{ABO2}, the authors first construct~$w$ inside~$\Omega$,
then interpolate near~$\partial\Omega$,
using the symmetries of the target~$\SS^{k-1}$,
so as to match the boundary datum.
On the contrary, we start from a map that satifies the boundary conditions
and we modify it inside~$\Omega$ so to obtain~$\S(w) = S$.
Before giving the details, we sketch the main steps of our construction.

First, we consider a locally piecewise affine extension
$u_*\in(L^\infty\cap W^{1,k})(\Omega, \, \R^m)$ of~$v$.
Since we have assumed that~$\X$ is polyhedral,
the singular set~$\S_y(u_*)$ will be locally polyhedral, for a.e.~$y$.
By projecting~$u_*$ onto~$\NN$ (using Hardt, Kinderlehrer and Lin's 
trick~\cite{HKL}, see Section~\ref{sect:extension}),
we define a map~$w_*\colon\Omega\to\NN$ such 
that~$w_* = v$ on~$\partial\Omega$,
$\S(w_*) = \S_y(u_*)$ (for a well-chosen~$y$) is locally polyhedral, 
and~$w_*$ has a locally nice singularity at~$\spt\S(w_*)$.
We cannot make sure that the singularity is nice 
up to the boundary of~$\Omega$, because the
boundary datum is not regular enough.
\begin{figure}[t]
	\centering
    \includegraphics[height=.35\textheight]{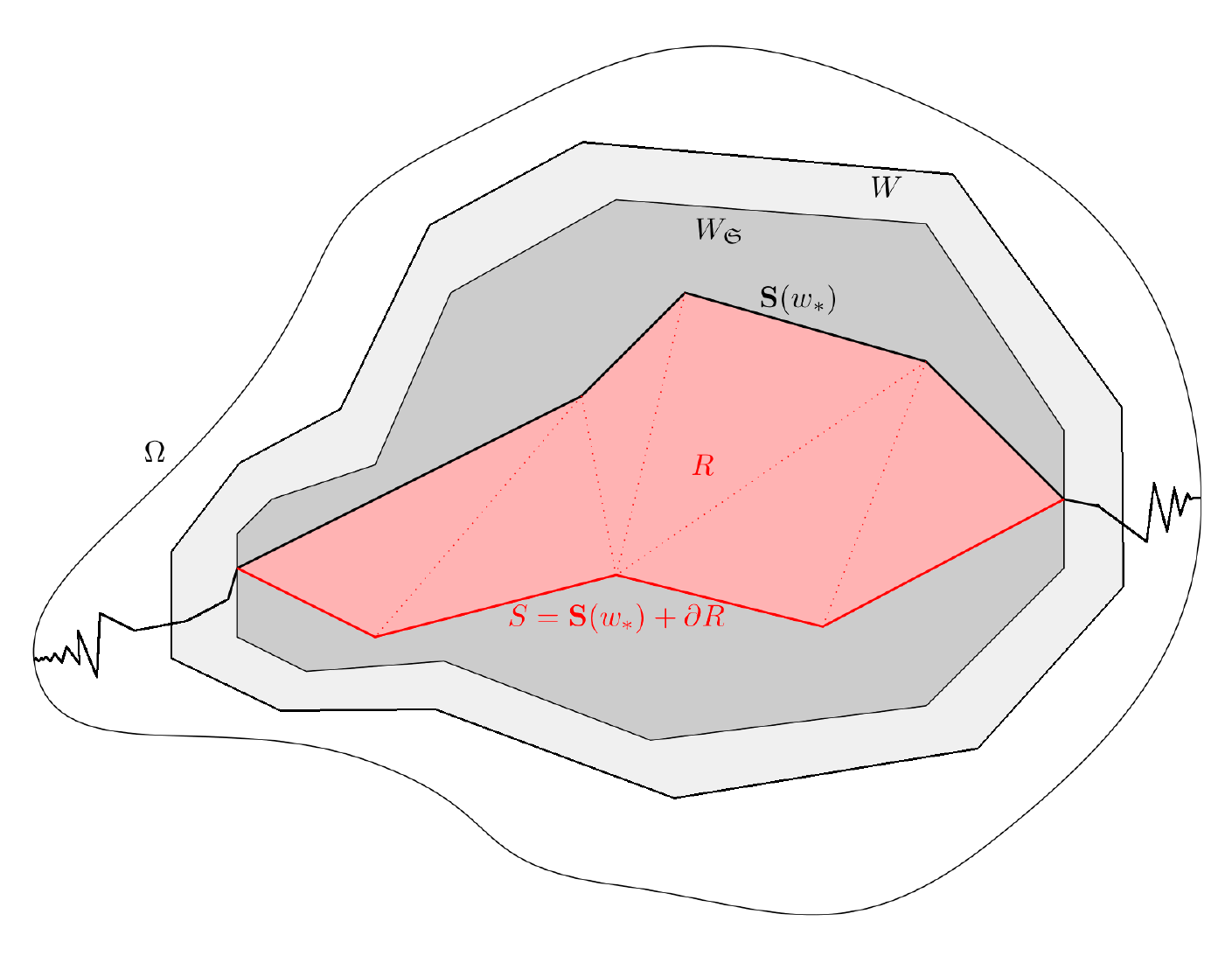}
	\caption{Sketch of the construction of a recovery sequence.
	Inside~$W_{\Sg}$, the chain~$S$ (in red) takes multiplicities
	in the set~$\Sg\subseteq\GN$. Outside~$W$, the original map~$w_*$ 
	and the modified map~$w$ coincide.}
	\label{fig:recovery}
\end{figure}
Let~$S$ be a finite-mass $n$-chain in the homology 
class~$\mathscr{C}(\Omega, \, v)$ defined by~\eqref{C_Omega_v}.
Thanks to~\eqref{S:cobord}, we know that 
$\S(w_*)= \S_y(u_*)\in\mathscr{C}(\Omega, \, v)$ and hence,
$\S(w_*)$ and~$S$ differ by a boundary. By approximation 
(see Section~\ref{sect:reduction}),
we reduce to the case
\[
 S = \S(w_*) + \partial R,
\]
where~$R$ is a polyhedral~$(n+1)$-chain with compact support in~$\Omega$.
Actually, we can make a further assumption on~$S$.
Let~$W_{\Sg}\csubset\Omega$ be an open set, with polyhedral boundary,
whose closure contains the support of~$R$ (see Figure~\ref{fig:recovery}).
Up to a density argument (Proposition~\ref{prop:approx_mult_noappendix}),
we can assume that $S\mres W_{\Sg}$
takes its multiplicities in the set~$\Sg\subseteq\GN$
defined by~\eqref{generators-intro}. Roughly speaking,
we replace each polyhedron~$K$ of~$S\mres W_{\Sg}$ with
a finite number of polyhedra, very close to each other,
whose multiplicities add up to the multiplicity of~$K$.
This is possible, because~$\Sg$ generates~$\GN$ 
by Proposition~\ref{prop:group_norm}. The 
assumption on the multiplicity of~$S\mres W_{\Sg}$ turns out
to be essential to obtain sharp energy bounds 
for our recovery sequence.

Let~$W$ be another open set, with polyhedral boundary,
such that~$W_{\Sg}\csubset W\csubset\Omega$ 
(see Figure~\ref{fig:recovery}).
In particular, $W$ contains the support of~$R$.
We aim to modify~$w_*$ inside~$W$, so to obtain a new map
$w\colon \Omega\to\NN$ with locally nice singularities
and~$\S(w) = \S(w_*) + \partial R = S$.
In other words, we need to ``move'' the singularities 
of~$w_*$ along the boundary of~$R$.
This is the key step in the construction.
We achieve this goal by a suitable generalisation
of the so-called ``insertion of dipoles'', 
Proposition~\ref{prop:dipole_simpl}
in Section~\ref{sect:dipole_simpl}.
For any $(n+1)$-po\-ly\-he\-dron~$T$ of~$R$,
we modify $w_*$ in a neighbourhood of~$T$ by
inserting an $\NN$-valued map that depends only 
on the $k-1$ coordinates in the orthogonal directions to~$T$.
To define~$w$ near~$\partial T$,
we use radial projections repeatedly, first 
onto the $n$-skeleton of~$T$, then onto 
its~$(n-1)$-skeleton, and so on. Eventually,
we obtain a map~$w\colon\Omega\to\NN$ that 
agrees with~$w_*$ out of a neighbourhood of~$\spt R$
(in particular, it matches the boundary datum),
has locally nice singularities at~$S$ and satisfies
$\S(w) = S$. By local surgery
(\cite[Lemma~9.3]{ABO2}, stated below as Lemma~\ref{lemma:ABO-minimal}),
we can also make sure that~$w_{|W}$ is~$\eta$-minimal.

The map~$w$ does not belong to the energy
space~$W^{1,k}(\Omega, \, \R^m)$, unless~$S=0$,
because it has a singularity of codimension~$k$.
Therefore, we must regularise~$w$ to construct a
recovery sequence. For~$x\in W$, we define
\begin{equation*} 
  u_\eps(x) := \min\left(\frac{\dist(x, \, \spt S)}{\eps}, \, 1\right)
  w(x).
\end{equation*}
Since~$w$ is~$\eta$-minimal in~$W$,
a fairly explicit computation allows us to estimate
the energy of~$u_\eps$ on~$W$, in terms of the area of~$\spt S$
and the maps~$\phi_K$ given by Definition~\ref{def:minimal}.
Moreover, for any simplex~$K$ of~$S\mres W_{\Sg}$,
the multiplicity~$\sigma_K$ of~$S$ at~$K$ belongs to~$\Sg$
and hence, 
\[
 \frac{1}{k} \int_{\SS^{k-1}} \abs{\nablaT\phi_K}^k\d\H^{k-1}
 \leq \abs{\sigma_K}_* + \eta,
\]
because of Definition~\ref{def:minimal} and~\eqref{generators-intro}.
Thanks to this inequality, we can indeed estimate
$E_\eps(u_\eps, \, W)$ in terms of the mass of~$S$,
up to remainder terms that can be made arbitrarily small.
However, this approach is not viable near the boundary of~$\Omega$,
because the regularity of~$w$ degenerates near~$\partial\Omega$.
Instead, we define~$u_\eps$ on~$\Omega\setminus W$ 
by adapting~\cite[Proposition 2.1]{Riviere-DenseSubsets},
see Section~\ref{sect:extension}. The two pieces 
--- inside and outside~$W$ --- 
are glued together by linear interpolation.

\subsection{Insertion of dipoles along a simplex}
\label{sect:dipole_simpl}

Our next result, Proposition~\ref{prop:dipole_simpl}, is the
main building block in the construction of the recovery sequence.

\begin{prop} \label{prop:dipole_simpl}
 Let~$D\subseteq\R^{n+k}$ be a bounded domain.
 Let~$\Sigma\subseteq D$ be a polyhedral set of dimension~$n$,
 and~$u\in W^{1,k-1}(D, \, \NN)$ a map with nice singularity at~$\Sigma$.
 Let~$T\csubset D$ be an {\BBB oriented} simplex of
 dimension~$n+1$ and~$\sigma\in\GN$.
 Then, there exists a map~$\tilde{u}\in W^{1, k-1}(D, \, \NN)$,
 with nice singularity at a polyhedral set of dimension~$n$, such that
 $\tilde{u} = u$ in a neighbourhood of~$\partial D$
 and $\S(\tilde{u}) = \S(u) + \sigma\partial\llbracket T \rrbracket$.
\end{prop}

Perhaps it is worth commenting on the assumptions 
of Proposition~\ref{prop:dipole_simpl}. In terms of regularity of $\NN$,
we do not need to work with smooth manifolds: a compact, connected 
Lipschitz neighbourhood retract would do. The assumption that
$\NN$ is $(k-2)$-connected could also be relaxed. 
$(k-2)$-connectedness is used in~\cite{PakzadRiviere, CO1}
to construct~$\S(u)$ for arbitrary $u\in W^{1,k-1}(\Omega, \, \NN)$;
however, if~$u$ has nice singularities and~$\GN$ is Abelian,
then~$\S(u)$ can be defined in a straightforward way. 
On the other hand, we must assume that $\NN$ is $(k-1)$-free
(that is, the fundamental group of~$\NN$ acts
trivially on~$\GN$). Should~$\NN$ not be $(k-1)$-free, 
we could not identify free homotopy classes of maps~$\SS^{k-1}\to\NN$
with elements of~$\GN$. In this case, the product of
free homotopy classes~$\SS^{k-1}\to\NN$ is multi-valued 
and hence, the equality
$\S(\tilde{u}) = \S(u) + \sigma\partial\llbracket T\rrbracket$ may fail.

\begin{figure}
 \begin{minipage}{.49\textwidth}
  \centering
  \includegraphics[width=\linewidth]{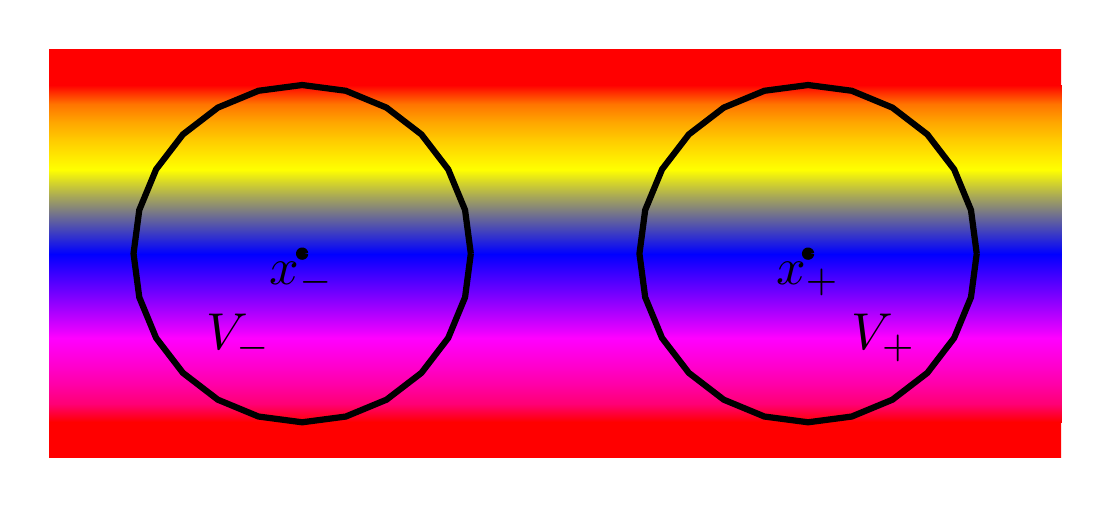}  
  (a)
 \end{minipage}
 \begin{minipage}{.5\textwidth}
  \centering
  \includegraphics[width=\linewidth]{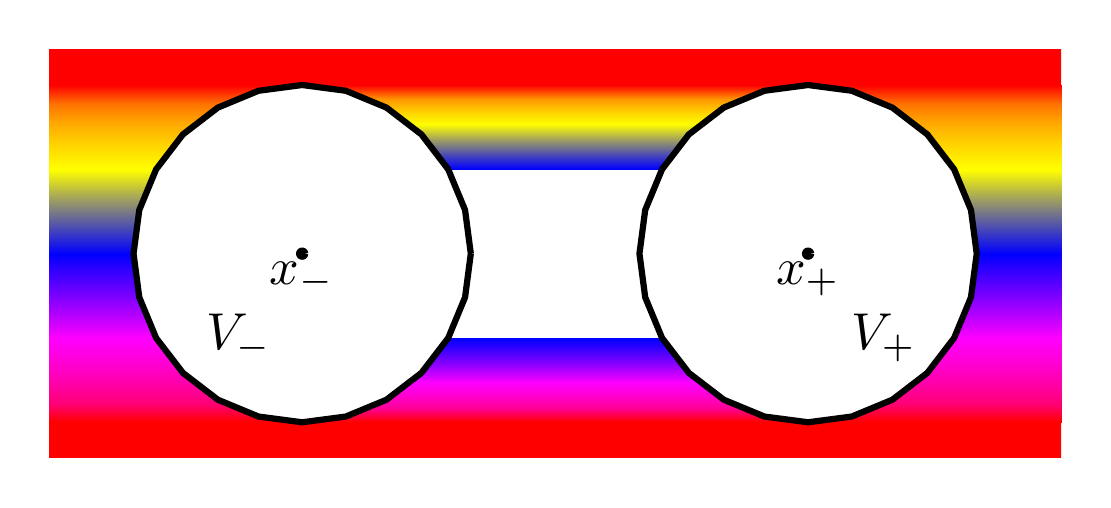} 
  (b)
 \end{minipage}
 
 \begin{minipage}{.49\textwidth}
  \centering
  \includegraphics[width=\linewidth]{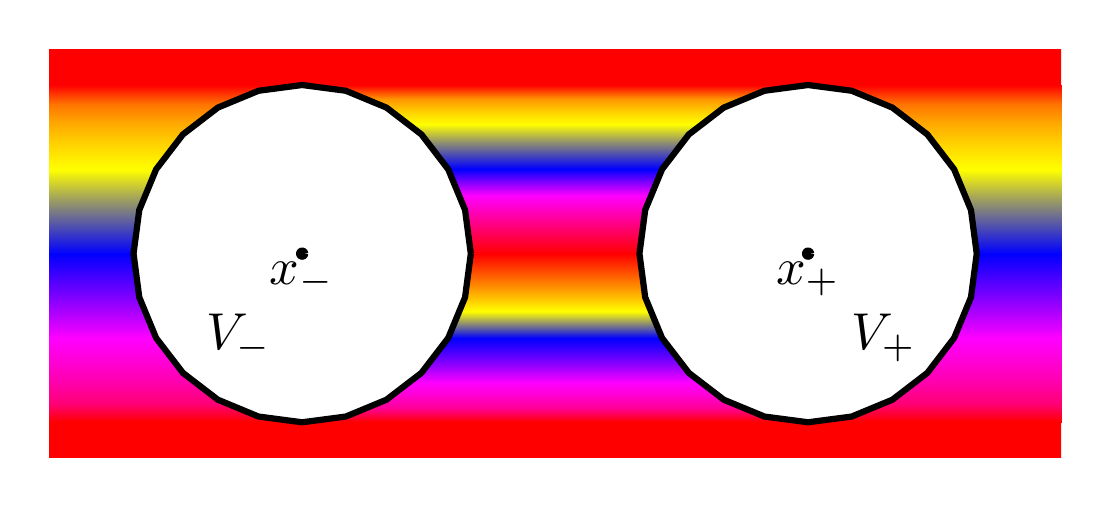}  
  (c)
 \end{minipage}
 \begin{minipage}{.5\textwidth}
  \centering
  \includegraphics[width=\linewidth]{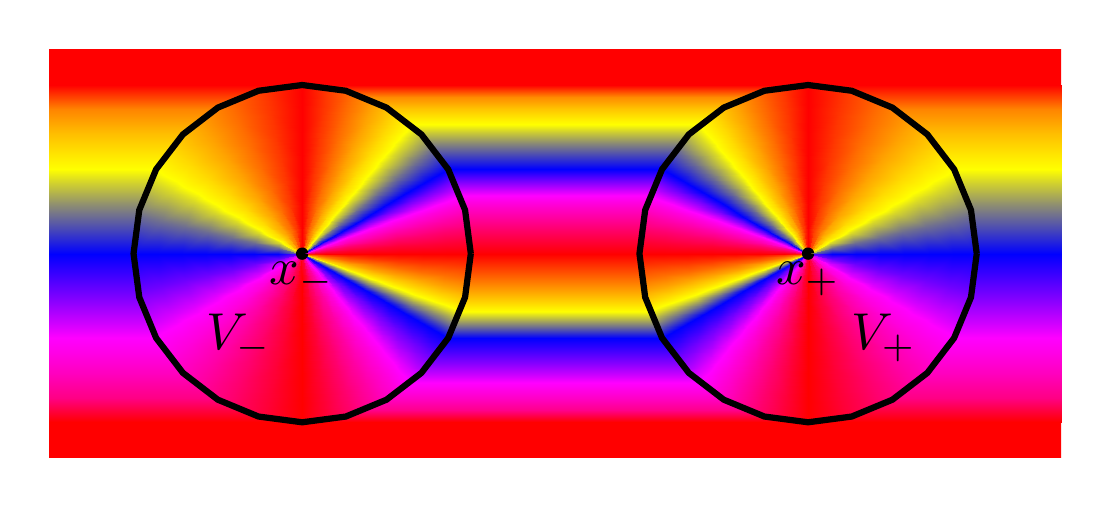}  
  (d)
 \end{minipage}
 
 \caption{{\BBB Idea of the proof of Proposition~\ref{prop:dipole_simpl}:
 an example with~$k=2$, $n=0$ and~$\NN=\SS^1$. The initial map~$u$
 is plotted in~(a); the values of~$u$ are represented by the 
 colour code. We aim to insert singularities of degrees~$1$, $-1$ 
 at the points~$x_+$, $x_-$. First, we reparametrise~$u$,
 creating a `slit' along the segment of endpoints~$x_+$ and~$x_-$ (b).
 Then, we fill the slit by inserting a map that winds around the 
 circle exactly once, as we move in the direction orthogonal
 to the segment of endpoints~$x_+$, $x_-$ (c). Finally, we define~$\tilde{u}$
 in the disks~$V_+$, $V_-$ in such a way that~$\tilde{u}$ 
 is homogeneous inside each disk~(d).
 The new map~$\tilde{u}$ behaves as required. 
 For instance, there are exactly three yellow points 
 on~$\partial V_+$; as we move anticlockwise around~$\partial V_+$,
 two of them carry the orientation `from red to blue' 
 and the other one carries the opposite orientation `from blue to red'.
 If we orient the target~$\SS^1$ `from red to yellow to blue',
 then the degree of~$\tilde{u}$ on~$\partial V_+$ is~$1$.}}
 \label{fig:colors}
\end{figure}

The proof of Proposition~\ref{prop:dipole_simpl} (see Figure~\ref{fig:colors})
is based on a construction known as ``insertion of dipoles''.
{\BBB Several variants of this construction are available in the literature
(see e.g.~\cite{BrezisCoronLieb, Bethuel1990, BethuelbrezisCoron,
GiaquintaModicaSoucek-I, PakzadRiviere}), but all 
of them rely of the following fact: a map~$B^{k-1}\to\NN$ that takes a constant
value on~$\partial B^{k-1}$ may be identified with a map~$\SS^{k-1}\to\NN$,
by collapsing the boundary of the disk to a point. As a consequence,
if a continuous map~$\phi\colon B^{k-1}\to\NN$ is constant on~$\partial B^{k-1}$,
then we may define the homotopy class of~$\phi$ as an element of~$\pi_{k-1}(\NN)$.
(In principle, we should distinguish between free or based homotopy, according 
to whether the boundary value of~$\phi$ is allowed to vary during
the homotopy or not; however, the assumption~\eqref{hp:N} guarantees that these
two notions are equivalent.)}

\begin{lemma} \label{lemma:ext1}
 {\BBB Let~$K$ be a convex polyhedron,
 let~$h\colon K\to\NN$ be a Lipschitz map, and let~$\sigma\in\GN$.
 Then, there exists a Lipschitz map
 $u\colon K\times B^{k-1}\to\NN$ such that
 \begin{equation} \label{ext11}
  u(x^\prime, \, x^{\prime\prime}) = h(x^\prime)
  \qquad \textrm{for any } 
  (x^\prime, \, x^{\prime\prime})\in K\times \partial B^{k-1}
 \end{equation}
 and, for any~$\sigma\in\GN$, the homotopy class 
 of~$u(x^\prime, \, \cdot)$ is~$\sigma$.}
\end{lemma}
{\BBB The proof of Lemma~\ref{lemma:ext1} is completely 
standard, but we provide it for the sake of convenience.}
\begin{proof}[Proof of Lemma~\ref{lemma:ext1}]
 We choose a point~$x^\prime_0\in K$ and consider the map
 $\psi\colon [0, \, 1]\times K\to K$ 
 as~$\psi(t, \, x^\prime) := t x^\prime + (1 - t)x^\prime_0$.
 {\BBB We define $u\colon K\times (B^{k-1}\setminus B^{k-1}_{1/2})\to\NN$ as
 \[
   u(x^\prime, \, x^{\prime\prime}) :=
   (h\circ\psi)(2|x^{\prime\prime}| - 1, \, x^\prime) 
   \qquad \textrm{for } x^\prime\in K, \
   1/2 \leq \abs{x^{\prime\prime}} \leq 1.
 \]
 The map~$u$ is Lipschitz and satisfies~\eqref{ext11}; moreover,
 for~$\abs{x^{\prime\prime}} = 1/2$ we have 
 $u(x^\prime, \, x^{\prime\prime}) = h(x^\prime_0)$.}
 Now, we take a smooth map
 $\phi\colon B^{k-1}\to\NN$ that is constant
 on~$\partial B^{k-1}$ --- say, $\phi = z_0\in\NN$ on~$\partial B^{k-1}$
 --- and has homotopy class~$\sigma$.
 Let~$\zeta\colon [0, \, 1]\to\NN$ be a 
 Lipschitz curve with $\zeta(0) = z_0$,
 $\zeta(1) = h(x^\prime_0)$.
 We define~$u\colon K\times B^{k-1}_{1/2}\to\NN$ as
 \[
  u(x^\prime, \, x^{\prime\prime}) :=
   \begin{cases}
    \zeta\left(4\abs{x^{\prime\prime}} - 1\right) 
    & \textrm{if } 1/4 \leq \abs{x^{\prime\prime}} < 1/2 \\
    \phi(4x^{\prime\prime}) 
    & \textrm{if } \abs{x^{\prime\prime}} < 1/4.
   \end{cases}
 \]
 {\BBB For any~$x^\prime\in K$, the map~$u(x^\prime, \, \cdot)$ 
 is (freely) homotopic to~$\sigma$, via a reparametrisation and
 a change of base-point. Therefore, the homotopy class 
 of~$u(x^\prime, \, \cdot)$ is~$\sigma$.}
\end{proof}

\begin{proof}[Proof of Proposition~\ref{prop:dipole_simpl}]
 We triangulate~$\Sigma\cup T$, that is, we write~$\Sigma\cup T$
 as a finite union of closed simplices in such a way that, 
 for any simplices~$K$, $K^\prime$ with~$K\neq K^\prime$, 
 $K\cap K^\prime$ is either empty or a boundary face
 of both~$K$ and~$K^\prime$.
 We denote by~$T_n$ the $n$-skeleton of this triangulation 
 (i.e., the union of all simplices of dimension~$n$ or less).
 We will construct a sequence of maps
 $u^{n+1}$, $u^n$, \ldots, $u^1$, $u^0$ 
 by modifying  the given map~$u$ first along
 the simplices of dimension~$n+1$ that are contained in~$T$, then along
 those of dimension~$n$, and so on. In order to do so,
 we first need to construct a suitable covering of~$T$.
 
 \begin{figure}[t]
	\centering
    \begin{subfigure}{.46\textwidth}
     \includegraphics[height=.35\textheight,angle=90]{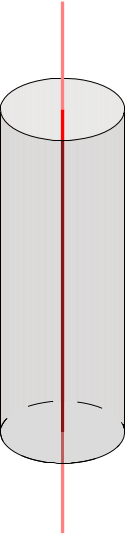}
    \end{subfigure}
    \begin{subfigure}{.53\textwidth}
     \includegraphics[height=.19\textheight]{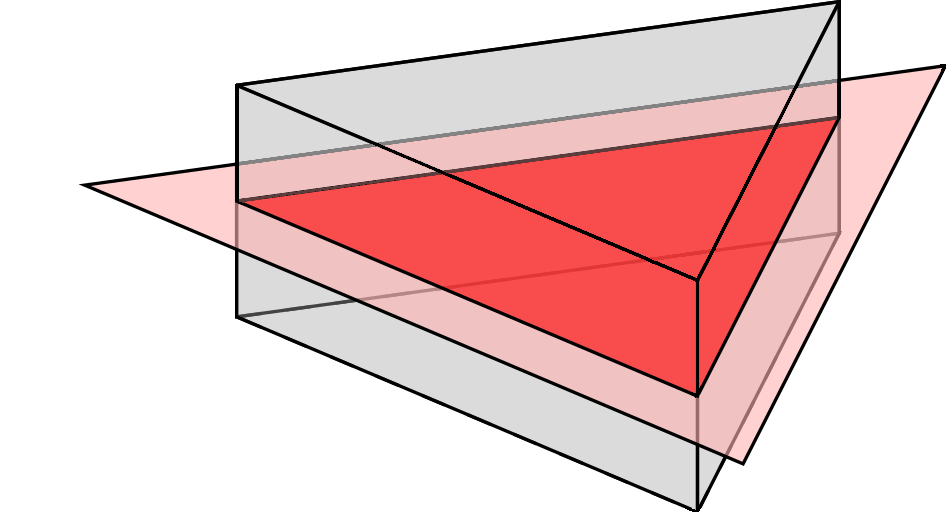}
    \end{subfigure}
	\caption{The set~$V(K, \, \delta, \, \gamma)$, in case 
	$n = 1$, $k=2$ (left) and~$n = 2$, $k=1$ (right). 
	In both cases, the polyhedron~$K$ is in pink and~$\tilde{K}$ is in red.}
	\label{fig:truncated}
 \end{figure}
 
 \begin{figure}[tb]
		\centering
		\includegraphics[height=.4\textheight]{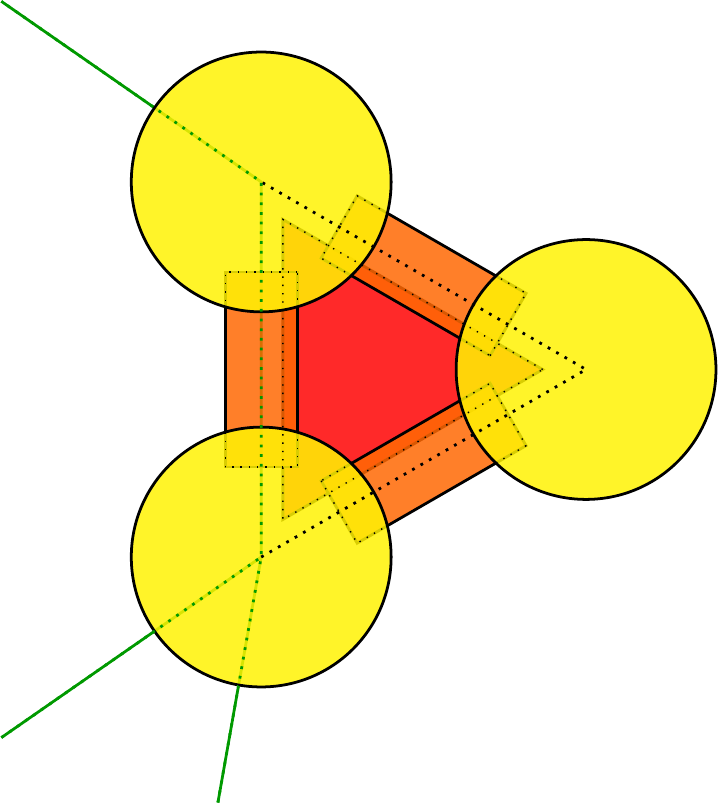}
	\caption{The covering of~$T$, in case~$n=1$ and~$k=2$
	(view from the top). The set~$\Sigma$ is in green.}
	\label{fig:covering}
 \end{figure}
 
 \setcounter{step}{0}
 \begin{step}[Construction of a covering of~$T$] \label{step:covering}
  Let~$K\subseteq T$ be a simplex of dimension~$j>0$.
  Let~$K^\perp$ be the orthogonal $(n+k-j)$-plane to~$K$ 
  through the origin. We fix positive
  numbers $\delta_K$, $\gamma_K$ and define
  \begin{gather}
   \tilde{K} := \left\{x^\prime\in K\colon 
   \dist(x^\prime, \, \partial K) > \gamma_K\right\} \!, \label{tildeK} \\
   V_K := \left\{x^\prime + x^{\prime\prime}
   \colon x^\prime\in\tilde{K}, \
   x^{\prime\prime}\in K^\perp, \
   |x^{\prime\prime}| < \delta_K \right\} \!,  \label{V_K} \\
   \Gamma_K := \left\{x^\prime + x^{\prime\prime}
   \colon x^\prime\in\tilde{K}, \
   x^{\prime\prime}\in K^\perp, \
   |x^{\prime\prime}| = \delta_K \right\} \label{Gamma_K}
  \end{gather} 
  (see Figure~\ref{fig:truncated}). If~$K$ is a~$0$-dimensional simplex,
  i.e.~a point, we define $V_K := B^{n+k}(K, \, \delta_K)$
  and~$\Gamma_K:= \partial V_K$.
  By choosing~$\delta_K$, $\gamma_K$ in a suitable way, we can make
  sure that the following properties are satisfied:
  \begin{enumerate}[label=(\alph*)]
   \item $V_K\csubset D$ for any simplex~$K\subseteq T$.
   \item For any~$j$-dimensional simplex~$K\subseteq T$, we have
   \[
    \partial V_K\setminus\Gamma_K \subseteq 
    \bigcup_{K^\prime\subseteq T\colon \dim K^\prime < j} V_{K^\prime}
   \]
   (in case~$j=0$, both sides of the inclusion are empty).
   \item For any simplices $K\subseteq T$, 
   $K^\prime\subseteq T$ with~$K\neq K^\prime$, 
   $\dim K = \dim K^\prime$, we have
   $\overline{V_K} \cap \overline{V_{K^\prime}} = \emptyset$.
   \item For any simplices~$K\subseteq T$, $K^\prime\subseteq\Sigma\cup T$
   with~$K\not\subseteq K^\prime$, 
   we have~$\overline{V_K}\cap K^\prime = \emptyset$.
   \item {\BBB No simplex~$K\subseteq T$ is entirely contained in
   $\cup\{\overline{V_{K^\prime}}\colon 
    \dim K^\prime < \dim K\}$.}
  \end{enumerate}
  Property~(b) implies that the~$V_K$'s do cover~$T$. To construct
  a covering that satisfies~(a)--(e), we first 
  cover the $0$-skeleton of~$T$ by pairwise
  disjoint balls that are compactly contained in~$D$.
  Then, we cover each~$1$-dimensional simplex 
  in~$T$ by a ``thin cylinder'', whose bases are 
  contained in the balls we have chosen before. Next, we cover
  each~$2$-dimensional simplex by a ``thin shell'', and so on,
  as illustrated in Figure~\ref{fig:covering}.
  At each step, we can make sure that the properties (a)--(e)
  are satisfied, because the simplices have pairwise
  disjoint interiors and only intersect along their boundaries.
  As a consequence of~(d), for any simplex~$K\subseteq T$ there holds
  \begin{equation} \label{simpl}
   \begin{cases}
    \overline{V_K} \cap (\Sigma\cup T_n) = \emptyset 
      & \textrm{if } \dim K = n+1 \\
    \Gamma_K \cap (\Sigma\cup T_n) = \emptyset 
      & \textrm{if } \dim K = n.
   \end{cases}
  \end{equation}
  For any integer~$j\in\{0, \, 1, \, \ldots, \, n+1\}$, we define
  \begin{equation*} 
   V^{=j} := \bigcup_{K\subseteq T\colon \dim K = j} V_K,
   \qquad V^{< j} := \bigcup_{i=0}^{j-1} V^{=i},
   \qquad V^{\geq j} := \bigcup_{i=j}^{n+1} V^{=i}
  \end{equation*}
  and~$V^{<0} := \emptyset$.
 \end{step}

 \begin{step}[Construction of~$u^{n+1}$] \label{step:un+1}
  Let~$K\subseteq T$ be a $(n+1)$-simplex of the triangulation,
  {\BBB with the orientation induced by~$T$.} 
  We identify~$V_K$ with~$\tilde{K}\times B^{k-1}(0, \, \delta_K)$,
  where~$\tilde{K}$ is given by~\eqref{tildeK}.
  {\BBB We construct a Lipschitz map~$u^{n+1}_K\colon V_K\to\NN$
  as follows. First, we let
  \begin{equation} \label{hole}
   u^{n+1}_K(x^\prime, \, x^{\prime\prime})
   := u\left(x^\prime, \, 2x^{\prime\prime}
    - \frac{\delta_K x^{\prime\prime}}{|x^{\prime\prime}|}\right)
   \qquad \textrm{for } x^\prime\in\tilde{K}, \
   \delta_K/2 \leq |x^{\prime\prime}| \leq \delta_K.
  \end{equation}
  Thus, $u^{n+1}_K = u$ on~$\Gamma_K$, while
  $u^{n+1}_K(x^\prime, \, x^{\prime\prime}) = u(x^\prime, \, 0)$
  for~$|x^{\prime\prime}| = \delta_K/2$. Since the trace of~$u^{n+1}_K$
  on~$\tilde{K}\times\partial B^{k-1}(0, \, \delta_K/2)$
  only depends on the variable~$x^\prime$, we may apply Lemma~\ref{lemma:ext1}
  and define~$u^{n+1}_K$ in~$\tilde{K}\times B^{k-1}(0, \, \delta_K/2)$
  in such a way that, for any~$x^\prime\in\tilde{K}$,
  \begin{equation} \label{hc_hole}
   \textrm{the homotopy class of }
   u^{n+1}_K(x^\prime, \, \cdot)_{|B^{k-1}(0, \, \delta_K/2)}
   \ \textrm{ is } (-1)^{n+1}\sigma.
  \end{equation}
  The sign~$(-1)^{n+1}$ will be useful to compensate for
  orientation effects, later on in the proof.}
  
  We define a map
  \[
   u^{n+1}\colon \left(D\setminus V^{< n+1}\right)\cup V^{=n+1}\to\NN
  \]
  as follows: $u^{n+1}(x):= u^{n+1}_K(x)$ if~$x\in V_K$ for some
  $(n+1)$-simplex~$K$, and~$u^{n+1}(x) := u(x)$ otherwise.
  This definition is consistent. Indeed, the sets~$\overline{V_K}$
  are pairwise disjoint, due to~(c). Moreover, if a point~$x$
  belongs both to~$\overline{V_K}$ and to~$D\setminus V^{< n+1}$,
  then~$x\in\Gamma_V$ because of~(b), so~$u^{n+1}_K(x) = u(x)$ by~(i).
  Therefore, the map~$u^{n+1}$ is well-defined and locally Lipschitz 
  out of~$\Sigma$, with nice singularity at~$\Sigma$.
 \end{step}
 
 \begin{step}[Construction of~$u^{n}$] \label{step:un}
  Let~$K\subseteq T$ be a $n$-simplex.
  We identify $V_K$ with~$\tilde{K}\times B^k(0, \, \delta_K)$.
  The map~$u^{n+1}$ is Lipschitz continuous
  on~$\Gamma_K$, due to~\eqref{simpl}.
  Let~$\sigma_K\in\GN$ be the homotopy class of~$u^{n+1}$
  on an arbitrary slice of~$\Gamma_K$,
  of the form~$\{x^\prime\}\times\partial B^k(0, \, \delta_K)$. 
  If~$\sigma_K = 0$ then, by adapting the arguments of
  Lemma~\ref{lemma:ext1}, we can construct a Lipschitz 
  continuous map $u^n_K\colon V_K\to\NN$ such that
  $u^n_K = u^{n+1}$ on~$\Gamma_K$. 
  If~$\sigma_K\neq 0$, we define~$u^n_K\colon V_K\to\NN$ as
  \[
   u^n_K(x^\prime, \, x^{\prime\prime}) := 
   u^{n+1}\left(x^\prime, \, \frac{\delta_K x^{\prime\prime}}{\abs{x^{\prime\prime}}}\right)
   \qquad \textrm{for } (x^\prime, \, x^{\prime\prime})\in \tilde{K}\times B^k(0, \, \delta_K).
  \]
  In both cases, by a straightforward 
  computation, we obtain
  \begin{equation} \label{nicesing-n}
   \abs{\nabla u^n_K(x^\prime, \, x^{\prime\prime})} \lesssim |x^{\prime\prime}|^{-1}
   \qquad \textrm{for a.e.~}(x^\prime, \, x^{\prime\prime})
   \in \tilde{K}\times B^{k}(0, \, \delta_K),
  \end{equation}
  where the proportionality constant at the 
  right-hand side depends on~$\delta_K$ and~$u^{n+1}$.
  We define
  \[
   u^{n}\colon \left(D\setminus V^{<n}\right)\cup V^{\geq n}\to\NN
  \]
  as follows: $u^{n}(x):= u^{n}_K(x)$ if~$x\in V_K$ for some
  $n$-simplex~$K$, and~$u^{n}(x) := u^{n+1}(x)$ otherwise.
  Thanks to~(b), (c) and~\eqref{nicesing-n}, 
  we can argue as in Step~\ref{step:un+1} and check 
  that~$u^{n}$ is locally Lipschitz out of~$\Sigma\cup T_n$,
  with nice singularity at~$\Sigma\cup T_n$.
 \end{step}
 
 \begin{step}[Construction of~$u^j$ for~$j < n$]
  We proceed by induction. Let~$j\in\{0, \, 1, \, \ldots, n-1\}$.
  Suppose we have constructed a map
  \[
   u^{j+1}\colon \left(D\setminus V^{<j+1}\right)\cup V^{\geq j+1}\to\NN
  \]
  that is locally Lipschitz out of~$\Sigma\cup T_n$ and
  has a nice singularity at~$\Sigma\cup T_n$. 
  Let~$K\subseteq T$ be a $j$-simplex. By identifying
  $V_K$ with~$\tilde{K}\times B^{n+k-j}(0, \, \delta_K)$,
  we define $u^j_K\colon V_K\to\NN$,
  \begin{equation*} 
   u^j_K(x^\prime, \, x^{\prime\prime}) := 
   u^{j+1}\left(x^\prime, \, \frac{\delta_K x^{\prime\prime}}{\abs{x^{\prime\prime}}}\right)
   \qquad \textrm{for } (x^\prime, \, x^{\prime\prime})\in \tilde{K}\times B^{n+k-j}(0, \, \delta_K).
  \end{equation*}
  The map~$u^j_K$ is locally Lipschitz out of the set
  \[
   A := \left\{(x^\prime, \, x^{\prime\prime})\in\tilde{K}\times B^{n+k-j}(0, \, \delta_K)
   \colon \left(x^\prime, \, \frac{\delta_K x^{\prime\prime}}{\abs{x^{\prime\prime}}}\right)
   \in \Sigma\cup T_n\right\} \!.
  \]
  By Property~(d), the only simplices of~$\Sigma\cup T_n$ 
  that intersect~$\overline{V}_K$
  are those that contain~$K$. Therefore, if $H_1$, $H_2$, \ldots, $H_p$
  denote the $n$-dimensional (closed) simplices of~$\Sigma\cup T_n$ that 
  contain~$K$, then
  \begin{equation} \label{Hi1}
   (\Sigma\cup T_n)\cap\overline{V_K} 
   = \bigcup_{i=1}^p \left(H_i\cap\overline{V_K}\right)
  \end{equation}
  Moreover, Property~(d) and the convexity of~$H_i$ imply that
  \begin{equation} \label{Hi2}
   H_i\cap\overline{V_K} = \tilde{K}\times    
   \left(\tilde{H}_i\cap \bar{B}^{n+k-j}(0, \, \delta_K)\right) \!,
  \end{equation}
  where~$\tilde{H}_i\subseteq\R^{n+k-j}$ is a cone
  (i.e.,~$\lambda x\in\tilde{H}_i$ for any~$x\in\tilde{H}_i$ and any~$\lambda\geq 0$). 
  As a consequence, 
  \begin{equation*} 
   \begin{split}
    A 
    &\stackrel{\eqref{Hi1}, \, \eqref{Hi2}}{=} \bigcup_{i=1}^p 
    \left(\tilde{K}\times (\tilde{H}_i\cap B^{n+k-j}(0, \, \delta_K))\right) 
    \stackrel{\eqref{Hi2}}{=} \bigcup_{i=1}^p 
    \left(H_i\cap V_K\right) 
    \stackrel{\eqref{Hi1}}{=} (\Sigma\cup T_n)\cap V_K,
   \end{split}
  \end{equation*}
  that is, $u^j_K$ is locally Lipschitz out of~$\Sigma\cup T_n$.
  We claim that
  \begin{equation} \label{nicesing-j}
   |\nabla u^j_K(x)| \lesssim \dist^{-1}(x, \, \Sigma\cup T_n)
   \qquad \textrm{for a.e.~}x\in V_K,
  \end{equation}
  where the proportionality constant at the 
  right-hand side may depend on~$\delta_K$.
  Given~$x = (x^\prime, \, x^{\prime\prime})\in V_K$,
  let~$y(x) := (x^\prime, \, \delta_K x^{\prime\prime}/\abs{x^{\prime\prime}})$.
  By the induction hypothesis, $u^{j+1}$ 
  has a nice singularity at~$\Sigma\cup T_n$.
  Therefore, an explicit computation gives
  \begin{equation} \label{nicesing-j1}
   |\nabla u^j_K(x)| 
   \lesssim |x^{\prime\prime}|^{-1} 
   \dist^{-1}\left(y(x), \, \Sigma\cup T_n\right)
  \end{equation}
  for a.e.~$x\in V_K$. By~\eqref{Hi1} and~\eqref{Hi2}, the set~$\Sigma\cup T_n$
  agrees with~$\tilde{K}\times \cup_i\tilde{H}_i$ in~$\overline{V_K}$,
  and~$\cup_i\tilde{H}_i$ is a cone. Then, by a geometric argument
  (see Figure~\ref{fig:distance}), we have
  \begin{equation} \label{nicesing-j2}
   \frac{\abs{x^{\prime\prime}}}{\dist(x, \, \Sigma\cup T_n)} 
   = \frac{\delta_K}{\dist\left(y(x), \, \Sigma\cup T_n\right)}
  \end{equation}
  By combining~\eqref{nicesing-j1} and~\eqref{nicesing-j2},
  \eqref{nicesing-j} follows.
  Finally, we define
  \[
   u^{j}\colon \left(D\setminus V^{<j}\right)\cup V^{\geq j}\to\NN
  \]
  as follows: $u^{j}(x):= u^{j}_K(x)$ if~$x\in V_K$ for some
  $j$-simplex~$K\subseteq T$, and~$u^{j}(x) := u^{j+1}(x)$ otherwise.
  Thanks to~(b), (c) and~\eqref{nicesing-j}, 
  the map~$u^{j}$ is well-defined, locally Lipschitz out of~$\Sigma\cup T_n$
  and has a nice singularity at~$\Sigma\cup T_n$.
 \end{step}
 
 \begin{figure}[t]
		\centering
		\includegraphics[height=.3\textheight]{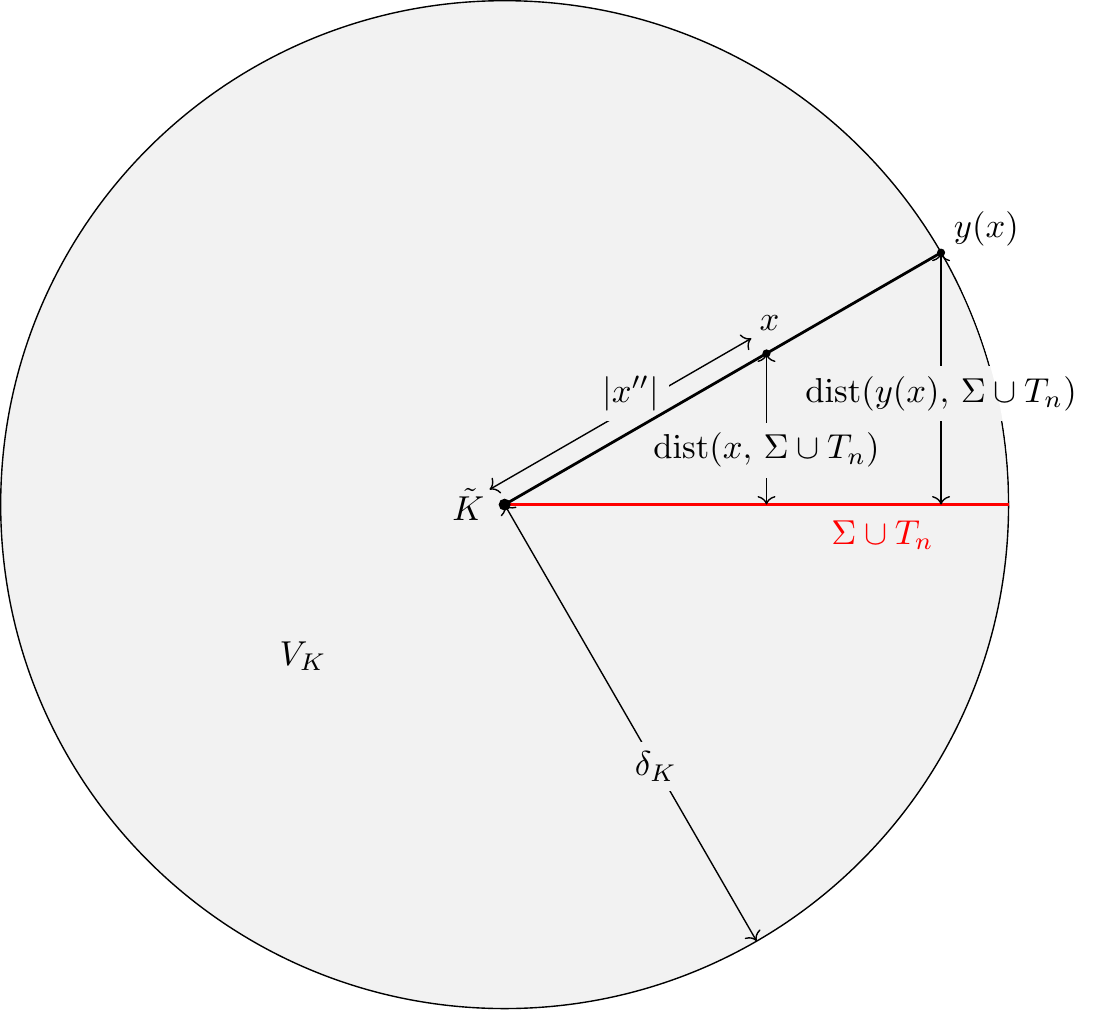}
	\caption{Proof of~\eqref{nicesing-j2}. The picture represents a slice of~$V_K$,
	of the form~$\{x^\prime\}\times B^{n+k-j}(0, \, \delta_K)$.}
	\label{fig:distance}
 \end{figure}
 
 \begin{step}[Conclusion]
  By induction, we have constructed a sequence of 
  maps~$u^{n+1}$, $u^{n}$, \ldots, $u^1$, $u^0$.
  Let~$\tilde{u}:= u^0\colon D\to\NN$. 
  By construction, the map~$\tilde{u}$ has a nice singularity at~$\Sigma\cup T_n$
  and agrees with~$u$ out of~$V^{< n+1}\cup V^{=n+1}$. In particular,
  $\tilde{u} = u$ in a neighbourhood of~$\partial D$, because of~(a).
  
  {\BBB It only remains to compute~$\S(\tilde{u})$.
  Let~$K$ be an $n$-simplex of~$T$. By Property~(e), $K$ is not 
  entirely contained in~$\overline{V^{<n}}$; 
  we take a point~$x\in K\setminus\overline{V^{<n}}$.
  Let~$K^\perp$ be the orthogonal $k$-plane to~$K$ at~$x$,
  and let~$F := \overline{V_K}\cap K^\perp$. By Property~(d),
  the only $(n+1)$-simplices that intersect~$F$
  are those that contain~$K$; we call them~$H_1$, \ldots, $H_p$. 
  We consider the restriction of~$\tilde{u}$ to the~$(k-1)$-sphere~$\partial F$.
  By construction (see~\eqref{hole} and~\eqref{hc_hole} in Step~\ref{step:un+1}), 
  $\tilde{u}_{|\partial F}$ consists (up to homotopy)
  of a reparametrisation of~$u_{|\partial F}$,
  with the insertion of `bubbles' around the points~$\partial F \cap H_i$.
  Each bubble carries the homotopy class~$\sigma$
  or~$-\sigma$, depending on the orientation of~$H_i$
  (which, we recall, is the one induced by~$T$).
  The net topological contribution of all the bubbles 
  may vanish or not, depending on whether the point~$x$ belongs
  to the boundary of~$T$ or not. As a result, we have}
  \[
   \begin{split}
    &(\textrm{homotopy class of } \tilde{u}_{|\partial F}\colon\partial F\simeq\SS^{k-1} \to\NN)\\
    &\qquad\qquad 
    = (\textrm{homotopy class of } u_{|\partial F}\colon\partial F\simeq\SS^{k-1} \to\NN)
    + \sigma(\textrm{multiplicity of } \partial\llbracket T \rrbracket \textrm{ at } x).
   \end{split} 
  \]
  {\BBB The sign of the second term 
  in the right-hand side depends on the choice of the sign
  we made in Equation~\eqref{hc_hole}
  (see, for instance, Property~(iv) in Lemma~8 of~\cite{CO1}).}
  Then, by Remark~\ref{remark:nice_sing}, 
  $\S(\tilde{u}) = \S(u) + \sigma\partial\llbracket T\rrbracket$.
  \qedhere
 \end{step}
\end{proof}

\subsection{Projection of a \texorpdfstring{$W^{1,k}$}{}-map onto~\texorpdfstring{$\NN$}{N}}
\label{sect:extension}

Before we pass to the construction of a recovery sequence, 
we gather some useful results, based on earlier work by
Hardt, Kinderlehrer and Lin
\cite[Lemma~2.3]{HKL}, \cite{HardtLin-Minimizing}, 
and Rivi\`ere \cite[Proposition 2.1]{Riviere-DenseSubsets};
see also~\cite[Proposition~6.4]{ABO2}
for similar statements in case~$\NN=\SS^{k-1}$.

For any~$y\in\R^m$, we consider the map~$\tilde{\RR}_y \colon z\mapsto\RR(z - y)$
which is well defined for~$z\in\R^m\setminus(\X+y)$.
This is not a retraction onto~$\NN$, in general, because
it does not restrict to the identity on~$\NN$.
However, for sufficiently small~$\abs{y}$ --- 
say, $y\in B^m_{\sigma}$ with~$\sigma>0$ small enough ---
the restriction~$\tilde{\RR}_{y|\NN}$ is a small perturbation of the identity
and, in particular, it is a diffeomorphism.
For~$y\in B^m_\sigma$ and~$z\in\R^m\setminus(\X+y)$, let us define
\begin{equation}\label{RR_y}
 \RR_y(z) := \left(\left(\tilde{\RR}_{y|\NN}\right)^{-1} \circ\RR\right)(z - y).
\end{equation}
This map is indeed a smooth retraction of~$\R^m\setminus(\X+y)$ onto~$\NN$.
We also define a function~$\psi\colon\R^m\to\R$ by
\begin{equation} \label{psi}
 \psi(z) := \min\left\{\frac{\dist(z, \, \X)}{\dist(\NN, \, \X)}, \, 1\right\}
 \qquad \textrm{for } z\in\R^m.
\end{equation}
The function~$\psi$ is Lipschitz and $\psi=1$ on~$\NN$.
By Proposition~\ref{prop:X} and~\eqref{psi}, we have
\begin{equation} \label{nabla_RRy}
 \abs{\nabla\RR_y(z)}\lesssim \frac{1}{\dist(z-y, \, \X)}
 \lesssim \frac{1}{\psi(z-y)}
\end{equation}
for any $y\in B^m_\sigma$ and $z\in\R^m\setminus(\X+y)$.
The proportionality constants here depend on~$\sigma$,
but~$\sigma = \sigma(\NN, \, \X, \, \RR)$ is fixed once and for all.
Finally, let~$\xi_\eps(t) := \min(t/\eps, \, 1)$ for~$t\geq 0$.

\begin{lemma} \label{lemma:extension}
 Let~$\Lambda$ be a positive number, and
 let~$u\in (L^\infty\cap W^{1,k})(\Omega, \, \R^m)$ be such that
 $\|u\|_{L^\infty(\Omega)}\leq\Lambda$.
 For~$y\in B^m_\sigma$, $\eps>0$ and~$x\in\Omega$, define
 \[
  w_y(x) := (\RR_y\circ u)(x), \qquad
  w_{\eps, y}(x) := (\xi_\eps\circ\psi)(u(x) - y) \, w_y(x).
 \]
 Then, the following properties hold.
 \begin{enumerate}[label=(\roman*), ref=\roman*]
  \item \label{ext:W1,k-1} For a.e.~$y\in B^m_\sigma$, 
  $w_y\in W^{1, k-1}(\Omega, \, \NN)$ and~$\S(w_y) = \S_y(u)$.
  
  \item \label{ext:W1,k} For a.e.~$y\in B^m_\sigma$ and sufficiently small~$\eps$,
  $w_{\eps, y}\in (L^\infty\cap W^{1,k})(\Omega, \, \R^m)$ and 
  $\|w_{\eps,y}\|_{L^\infty(\Omega)}\leq\max\{\abs{z}\colon z\in\NN\}$.
  
  \item \label{ext:energy} For any open set~$D\subseteq\Omega$, there holds
  \begin{equation*}
   \begin{split}
    &\int_{B^m_\sigma} \left(E_\eps(w_{\eps,y}, \, D) 
     + \eps^{-k} \L^{n+k}\{x\in D\colon w_{\eps,y}(x)\neq w_y(x)\}\right)\d y \\
    &\qquad\qquad\qquad \leq C_\Lambda \left(
     \abs{\log\eps}\norm{\nabla u}^k_{L^k(D)} + \L^{n+k}(D) \right) \!,
   \end{split}
  \end{equation*}
  where~$C_\Lambda$ is a positive constant that only depends 
  on~$\NN$, $k$, $\X$ and~$\Lambda$.
  
  \item \label{ext:conv} For a.e.~$y\in B^m_\sigma$
  there exists a (non-relabelled) \emph{subsequence} $\eps\to 0$
  such that $w_{\eps,y}\to w_y$ strongly in~$W^{1, k-1}(\Omega, \, \R^m)$.
 \end{enumerate}
\end{lemma}

\begin{remark} \label{rk:extension}
 Statement~\eqref{ext:energy} of Lemma~\ref{lemma:extension} 
 implies, via an averaging argument, that
 \[
  \inf\left\{ E_\eps(u)\colon u\in W^{1,k}_v(\Omega, \, \R^m)\right\} 
  \lesssim \abs{\log\eps}
 \]
 for any~$v\in W^{1-1/k,k}(\partial\Omega, \, \NN)$
 and any~$\eps>0$.
\end{remark}

\begin{proof}[Proof of Lemma~\ref{lemma:extension}]
 Throughout the proof, we denote by
 $C_\Lambda$ a generic positive constant that only depends on
 $\NN$, $k$, $\X$ and~$\Lambda$ (and may change from one 
 occurence to the other).

 \setcounter{step}{0}
 \begin{step}[Proof of~\eqref{ext:W1,k-1}]
  For a.e.~$y$, we have~$\RR\circ(u - y)\in W^{1,k-1}(\Omega, \, \NN)$
  (see e.g.~\cite[Lemma~14]{CO1} for a proof of this claim).
  Moreover, by~\cite[Lemma~17]{CO1} we know that
  \begin{equation*} 
   \S_{y^\prime}(\RR\circ(u - y)) = \S_y(u) \qquad 
   \textrm{for a.e. } y, \, y^\prime\in B^m_\sigma.
  \end{equation*} 
  Now, $w_y$ is obtained from~$\RR\circ(u-y)$ by composition with a map, 
  $(\tilde{\RR}_{y|\NN})^{-1}$, which is homotopic to the identity on~$\NN$.
  Therefore, from the identity above we obtain
  \begin{equation} \label{ext8}
   \S_{y^\prime}(w_y) = \S_{y}(u) 
   \qquad \textrm{for a.e. } y, \, y^\prime\in B^*.
  \end{equation}
  This can be first checked when~$u$ is smooth, using 
  \cite[Lemma~18]{CO1}, and remains true for a general~$u$
  by a density argument, using the continuity of~$\S$ and
  e.g.~\cite[Lemma~14]{CO1}.
 \end{step}
 
 \begin{step}[Proof of~\eqref{ext:W1,k}, \eqref{ext:energy}]
  It is immediate to see that 
  $\|w_{\eps,y}\|_{L^\infty(\Omega)}\leq\max\{\abs{z}\colon z\in\NN\}$.
  By~(i), $w_{\eps, y}\in(L^\infty\cap W^{1,k-1})(\Omega, \, \R^m)$ for a.e.~$y$,
  and by the chain rule, we have the pointwise bound
  \[
   \abs{\nabla w_{\eps,y}(x)} \leq C_\Lambda\left( 
   (\xi_\eps^\prime\circ\psi)(u(x) - y) \abs{\nabla u(x)} +
   (\xi_\eps\circ\psi)(u(x) - y) \abs{\nabla w_y(x)} \right)
  \]
  for a.e.~$x\in\Omega$. Thanks to~\eqref{nabla_RRy}, we deduce that
  \begin{equation} \label{ext4}
   \begin{split}
    \abs{\nabla w_{\eps,y}(x)} &\leq C_\Lambda\left( 
    (\xi_\eps^\prime\circ\psi)(u(x) - y) +
    \frac{(\xi_\eps\circ\psi)(u(x) - y)}{\psi(u(x) - y)}
    \right) \abs{\nabla u(x)} \\
    &\leq C_\Lambda\left(
    \frac{\one_{\{\psi(u(x) - y)\leq\eps\}}}{\eps} +
    \frac{\one_{\{\psi(u(x) - y)\geq\eps\}}}{\psi(u(x) - y)}
    \right) \abs{\nabla u(x)} 
   \end{split}
  \end{equation}
  (where, as usual, $\one_A$ denotes the characteristic function of a set~$A$).
  On the other hand, the $L^\infty$-norm of~$w_{\eps,y}$
  is uniformly bounded in terms of~$\NN$ only, and hence there holds
  \begin{equation} \label{ext5}
   f(w_{\eps,y}) \lesssim \one_{\{w_{\eps,y}\neq w_y\}}
   = \one_{\{\psi(u - y)<\eps\}} .
  \end{equation}
  Together, \eqref{ext4} and~\eqref{ext5} imply that
  \[
   \begin{split}
    &E_\eps(w_{\eps,y}, \, D)
    + \eps^{-k}\L^{n+k}\{x\in D\colon w_{\eps,y}(x)\neq w_y(x)\} 
    \leq C_\Lambda \int_\Omega \frac{\one_{\{\psi(u(x) - y)\leq\eps\}}}{\eps^k} \, \d x \\
    &\qquad\qquad + C_\Lambda \int_D \left(
    \frac{\one_{\{\psi(u(x) - y)\leq\eps\}}}{\eps^k} +
    \frac{\one_{\{\psi(u(x) - y)\geq\eps\}}}{\psi(u(x) - y)^k}
    \right) \abs{\nabla u(x)}^k \d x \\
   \end{split}
  \]
  We integrate the previous inequality for~$y\in B^m_\sigma$,
  apply Fubini theorem and make the change 
  of variable~$z = u(x) - y$:
  \[
   \begin{split}
    &\int_{B^m_\sigma} \left(E_\eps(w_{\eps,y}, \, D)
    + \eps^{-k}\L^{n+k}\{x\in D\colon w_{\eps,y}(x)\neq w_y(x)\}\right) \d y \\
    &\qquad \leq C_\Lambda 
    \int_D \int_{B^m_{\sigma + \Lambda}} 
    \left\{\left(\frac{\one_{\{\psi(z)\leq\eps\}}}{\eps^k} +
    \frac{\one_{\{\psi(z)\geq\eps\}}}{\psi(z)^k}\right) 
    \abs{\nabla u(x)}^k + \frac{\one_{\{\psi(z)\leq\eps\}}}{\eps^k}
    \right\} \d z \, \d x .
   \end{split}
  \]
  Since~$\X$ is a finite union of
  simplices of codimension~$k$ or higher,
  for~$\eps$ sufficiently small there holds
  \[
   \int_{B^m_{\sigma + \Lambda}}
   \one_{\{\psi(z)\leq\eps\}} \, \d z \leq C_\Lambda \eps^k, \qquad
   \int_{B^m_{\sigma + \Lambda}}
   \frac{\one_{\{\psi(z)\geq\eps\}}}{\psi(z)^k} \, \d z 
   \leq C_\Lambda \abs{\log\eps}
  \]
  (see e.g.~\cite[Lemma~8.3]{ABO2}).
  As a consequence, we obtain~\eqref{ext:energy}.
 \end{step}
 
 \begin{step}[Proof of~\eqref{ext:conv}]
  For a.e.~$y\in B^m_\sigma$, the set $\{\psi(u - y) = 0\} = (u - y)^{-1}(\X)$
  has Lebesgue measure equal to zero (see e.g. \cite[proof of Lemma~14]{CO1}).
  Then, since~$\xi_\eps\to 1$ pointwise on~$(0, \, +\infty)$ as~$\eps\to 0$,
  we have~$w_{\eps,y}\to w_y$ a.e. as~$\eps\to 0$, for a.e.~$y$.
  Using the chain rule,~\eqref{nabla_RRy} and~\eqref{ext4}, 
  we obtain that
  \[
   \abs{\nabla w_{\eps, y}(x) - \nabla w_y(x)}
   \leq C_\Lambda \left(\frac{1}{\eps} + 
   \frac{1}{\psi(u(x) - y)} \right) 
   \one_{\{\psi(u(x) - y)\leq\eps\}}\abs{\nabla u(x)} \! .
  \]
  for a.e.~$x\in\Omega$.
  We raise both sides of this inequality to the $(k-1)$-th power,
  integrate over~$(x, \, y)\in\Omega\times B^m_\sigma$, apply
  Fubini theorem and make the change of variable~$z = u(x) - y$:
  \[
   \begin{split}
    \int_{B^*} \norm{\nabla w_{\eps, y} - 
    \nabla w_y}^{k-1}_{L^{k-1}(\Omega)} \d y \leq C_\Lambda 
    \int_\Omega \int_{B^m_{\sigma + \Lambda}} 
    \left(\frac{1}{\eps^{k-1}} + \frac{1}{\psi(z)^{k-1}}\right)
    \one_{\{\psi(z)\leq\eps\}} \abs{\nabla u(x)}^{k-1} 
    \d z \, \d x .
   \end{split}
  \]
  We apply~\cite[Lemma~8.3]{ABO2} to estimate
  the integral with respect to~$z$:
  since~$\X$ has codimension~$k$, we obtain
  \[
   \begin{split}
    \int_{B^m_{\sigma + \Lambda}} 
    \left(\frac{1}{\eps^{k-1}} + \frac{1}{\psi(z)^{k-1}}\right)
    \one_{\{\psi(z)\leq\eps\}} \d z \leq C_\Lambda\eps,
   \end{split}
  \]
  so
  \begin{equation*} 
   \int_{B^m_\sigma} \norm{\nabla w_{\eps, y} - 
     \nabla w_y}_{L^{k-1}(\Omega)}^{k-1} \d y
   \leq C_\Lambda \eps \norm{\nabla u}_{L^{k-1}(\Omega)}^{k-1} \!.
  \end{equation*}
  By Fatou lemma, we deduce
  \begin{equation*} 
   \int_{B^m_\sigma} \liminf_{\eps\to 0} \norm{\nabla w_{\eps, y} - 
     \nabla w_y}_{L^{k-1}(\Omega)}^{k-1} \d y = 0 ,
  \end{equation*}
  so~\eqref{ext:conv} follows. \qedhere
 \end{step}
\end{proof}

\subsection{Construction of a recovery sequence}
\label{sect:recovery}

\subsubsection{Construction of an \texorpdfstring{$\NN$}{N}-valued map with
nice singularity at a locally polyhedral set}
\label{sect:w*}

In this section, we give the construction
of a recovery sequence. We first construct a map
$\Omega\to\NN$ that matches the Dirichlet boundary datum and
has nice singularities along a locally polyhedral set. 

\begin{lemma} \label{lemma:u*}
 Any boundary datum~$v\in W^{1-1/k, k}(\partial\Omega, \, \NN)$
 can be extended to a map $u^*\in (L^\infty\cap W^{1,k}_v)(\Omega, \, \R^m)$ 
 that satisfies the following properties, for a.e. $y\in\R^m$:
 \begin{enumerate}[label=(\alph*), ref=\alph*]
  \item $\M(\S_y(u_*))<+\infty$ and~$\S_{y}(u_*)\mres\partial\Omega = 0$;
  \item the chain $\S_{y}(u_*)$ is locally polyhedral;
  \item the chain $\S_{y}(u_*)$ takes its multiplicities in a
  finite subset of~$\GN$, which depends only on~$\NN$, $\RR$, $\X$;  
  \item there exists a locally $(n-1)$-polyhedral set~$P_y$ such that
  $\RR\circ (u_*-y)$ has a locally nice
  singularity at~$\spt\S_{y}(u_*)\cup P_{y}$.
 \end{enumerate}
\end{lemma}

The proof of Lemma~\ref{lemma:u*} relies on the following fact.

\begin{lemma} \label{lemma:loc_piecewise_affine}
 Any boundary datum~$v\in W^{1-1/k, k}(\partial\Omega, \, \R^m)$
 has a locally piecewise affine extension $u_*\in (L^\infty\cap W^{1,k}_v)(\Omega, \, \R^m)$.
\end{lemma}
We give a proof of Lemma~\ref{lemma:loc_piecewise_affine},
for the convenience of the reader only.
\begin{proof}[Proof of Lemma~\ref{lemma:loc_piecewise_affine}]
 Arguing component-wise, we reduce to the case~$m=1$.
 Let~$u\in W^{1,k}_v(\Omega)$ be an extension of~$v$. By a truncation argument, 
 we can make sure that $v\in L^\infty(\Omega)$.
 Let~$\Gamma_1 := \{x\in\Omega\colon \dist(x, \, \partial\Omega) > 1/2\}$ and,
 for any integer~$j\geq 2$, let $\Gamma_j := \{x\in\Omega\colon (j+1)^{-1} 
 < \dist(x, \, \partial\Omega) < (j - 1)^{-1}\}$.
 Using a partition of unity, 
 we construct a sequence of smooth functions
 $\varphi_j \in C^\infty_{\mathrm{c}}(\Gamma_j)$ such that
 $\sum_{j\geq 1}\varphi_j = 1$. Thanks to e.g.
 \cite[Theorem~1]{VanSchaftingen-PiecewiseAffine},
 for any~$j$ there exists a triangulation~$\mathcal{T}_j$ of~$\R^{n+k}$ such that 
 the piecewise affine interpolant~$u_j$ of~$\varphi_j u$ along~$\mathcal{T}_j$ 
 is well-defined (that is, all the vertices of~$\mathcal{T}_j$ are Lebesgue points
 of~$\varphi_j u$) and there holds
 \begin{equation} \label{interpolant}
  \norm{\nabla u_j - \nabla(\varphi_j u)}_{L^k(\R^{n+k})} \leq 2^{-j}.
 \end{equation}
 Moreover, the proof of~\cite[Theorem~1]{VanSchaftingen-PiecewiseAffine}
 shows that for any~$r>0$, we can choose~$\mathcal{T}_j$ 
 such that all the simplices of~$\mathcal{T}_j$
 have diameter~$\leq r$. In particular, we can make sure that 
 $u_j$ is still supported in~$\Gamma_j$. 
 Now, we define~$u_* := \sum_{j\geq 1} u_j$. Since the support of~$u_j$ 
 intersects the support of~$u_i$ only for finitely many~$i$,
 the function~$u_*$ is locally piecewise affine. Moreover, $u_*\in L^\infty(\Omega)$
 because, by construction, $\norm{u_j}_{L^\infty(\Omega)}\leq \norm{u}_{L^\infty(\Omega)}$
 for any~$j$, and $u\in W^{1, k}(\Omega)$ due to~\eqref{interpolant}.
 Finally, for any~$N\geq 1$
 the function $\sum_{j=1}^N (u_j - \varphi_j u)$ is compactly supported in~$\Omega$,
 and hence $\sum_{j=1}^N (u_j - \varphi_j u)\in W^{1, k}_0(\Omega)$. Passing to the limit
 as~$N\to +\infty$, we conclude that $u_* - u \in W^{1,k}_0(\Omega)$, and the lemma follows.
\end{proof}

\begin{proof}[Proof of Lemma~\ref{lemma:u*}]
 Let~$u_*$ be the locally piecewise extension of~$v$
 given by Lemma~\ref{lemma:loc_piecewise_affine}. 
 Statement~(a) follows from~\eqref{S:mass} in Proposition~\ref{prop:S}, because~$u_*\in W^{1,k}(\Omega, \, \R^m)$.
 Let~$K\subseteq\Omega$ be a (closed) $(n+k)$-simplex such that~$u_{*|K}$ is affine.
 Since we have assumed that~$\X$ is polyhedral, for any~$y\in\R^m$
 the inverse image $(u_* - y)^{-1}(\X)\cap K$ is polyhedral too.
 Take~$y\in\R^m$ such that $(u_* - y)_{|K}$ is transverse to each cell of~$\X$.
 By \cite[Corollary~1]{CO1}, we have  $\S_y(u_*)\mres K = \S_y(u_{*|K})$
 and by definition (see \cite[Section~3.2]{CO1} and Section~\ref{appendix:S} below),
 the latter is a polyhedral chain
 supported on~$(u_* - y)^{-1}(\X)\cap K$.
 Thus, $\S_y(u_*)$ is locally polyhedral.
 Moreover, $\S_y(u_*)$ take its multiplicities in the set
 \[
  \left\{\pm(\textrm{homotopy class of } \RR \textrm{ around } H)\colon
  H \textrm{ is a } (m-k)\textrm{-polyhedron of } \X\right\},
 \]
 which is a finite subset of~$\GN$, because~$\X$ is a
 finite union of polyhedra.
 Finally, let us prove Statement~(d). 
 Take an open set~$W\csubset\Omega$, and take~$y\in\R^m$ 
 such that $u_{*|W}$ is transverse to each cell of~$\X$.
 Let~$K$ be a $(n+k)$-simplex such that $K\cap W\neq\emptyset$
 and $u_{*|K}$ is affine. By transversality, we see that
 \[
  \dist(u_*(x) - y, \, \X) \geq C_{K, y} \dist(x, (u_* - y)^{-1}(\X))
  \qquad \textrm{for any } x\in K,
 \]
 where~$C_K > 0$ is a constant that depends on the
 (constant) gradient of~$u_*$ on~$K$ and on~$y$. Since~$W$ is covered by 
 finitely many simplices, we have
 \[
  \dist(u_*(x) - y, \, \X) \geq C_{ W, y} \dist(x, (u_* - y)^{-1}(\X))
  \qquad \textrm{for any } x\in  W,
 \]
 where~$C_{ W, y}:=\min_{K\colon K\cap W\neq\emptyset} C_{K,y}>0$.
 Then, by applying the chain rule and Proposition~\ref{prop:X},
 we conclude that $\RR\circ(u_* - y)_{|W}$ has a nice singularity
 at~$(\spt\S_y(u_*)\cup P_y)\cap W$, where~$P_y := (u_* - y)^{-1}(\X_{m-k-1})$.
\end{proof}

\subsubsection{Reduction of the problem}
\label{sect:reduction}

Throughout the rest of Section~\ref{sect:limsup}, 
we fix the boundary datum~$v\in W^{1-1/k, k}(\partial\Omega, \, \NN)$
and let~$u_*$ be the map given by Lemma~\ref{lemma:u*}.
We also fix $y^*\in\R^m$, with $\abs{y^*}$ sufficiently small,
in such a way that Statements (a)--(d) in Lemma~\ref{lemma:u*}
are satisfied. Let~$w_* := \RR_{y_*}\circ u_*$,
where~$\RR_{y_*}$ is defined by~\eqref{RR_y}.
By Lemma~\ref{lemma:u*}, the map $w_*$ has a locally nice 
singularity at~$\spt\S_{y_*}(u_*)\cup P_{y_*}$,
where $P_{y_*}$ is a locally polyhedral set of dimension~$n-1$.
By Lemma~\ref{lemma:extension}, we can choose~$y_*$
so to have $w_*\in W^{1, k-1}(\Omega, \, \NN)$ and~$\S(w_*) = \S_{y_*}(u_*)$
as well.

\begin{remark} \label{remark:S(w*)}
 For a generic map $w\in W^{1, k-1}(\Omega, \, \R^m)$,
 $\S(w)$ is only well-defined as a \emph{relative} flat chain,
 $\S(w)\in\F_n(\Omega; \, \GN)$ 
 (see~\cite[Section~3]{CO1}). However, $\S_{y_*}(u_*)$
 is well-defined as an element of~$\F_n(\R^{n+k}; \, \GN)$,
 because $u_*\in W^{1,k}(\Omega, \, \R^m)$
 (see Proposition~\ref{prop:S}). With a slight abuse of notation,
 we will regard
 $\S(w_*)$ as an element of~$\F_n(\R^{n+k}; \, \GN)$, too.
\end{remark}

Let~$S$ be a finite-mass $n$-chain, supported in~$\overline{\Omega}$,
that is cobordant to~$\S(w_*)$. By definition of~$\mathscr{C}(\Omega, \, v)$,
Equation~\eqref{C_Omega_v}, $S$ and~$\S(w_*)$ differ by a boundary.
By an approximation argument, we will reduce to the case~$S$ has a special form.

\begin{prop} \label{prop:approx-noappendix}
 Let~$S\in\mathscr{C}(\Omega, \, v)$ be a finite mass chain.
 Then, there exists a sequence of polyhedral $(n+1)$-chains~$R_j$,
 with compact support in~$\Omega$, such that 
 $\S(w_*) + \partial R_j \to S$ (with respect to the~$\F$-norm)
 and $\M(\S(w_*) + \partial R_j)\to \M(S)$ as~$j\to+\infty$.
\end{prop}

The proof of Proposition~\ref{prop:approx-noappendix} 
is left to Appendix~\ref{sect:approximation}. Thanks to 
Proposition~\ref{prop:approx-noappendix}, and a diagonal argument,
we can assume with no loss of generality that~$S$ has the form
\begin{equation} \label{specialS1}
 S = \S(w_*) + \partial R,
\end{equation}
where~$R$ is a polyhedral $(n+1)$-chain, 
compactly supported in~$\Omega$.
There is one further assumption we can make.
Let~$W_{\Sg}\csubset\Omega$ be an open set,
with polyhedral boundary, such that~$\partial W_{\Sg}$ is transverse to~$\spt S$
(more precisely, there exist triangulations of~$\partial W_{\Sg}$
and~$\spt S$ such that any simplex of the triangulation of~$\partial W_{\Sg}$
is transverse to any simplex of the triangulation of~$\spt S$) and
\begin{equation} \label{specialS2}
 \spt R \subseteq \overline{W_{\Sg}}, \qquad S\mres\partial W_{\Sg} = 0.
\end{equation}
The condition $S\mres\partial W_{\Sg} = 0$ is satisfied 
because, by transversality,
$\spt S\cap\partial W_{\Sg}$ has dimension~$(n-1)$ or less and hence,
it cannot support a non-trivial polyhedral $n$-chain.

\begin{prop} \label{prop:approx_mult_noappendix} 
 There exists a sequence of polyhedral 
 $(n+1)$-chains~$R_j$, supported in $\overline{W_{\Sg}}$,
 such that the following hold:
 \begin{enumerate}[label=(\roman*)]
  \item $S+\partial R_j \to S$, with respect 
  to the~$\F$-norm, as~$j\to+\infty$;
  \item $\M(S+\partial R_j)\to\M(S)$ as~$j\to+\infty$;
  \item for any~$j$, $(S + \partial R_j)\mres\partial W_{\Sg} = 0$;
  \item for any~$j$, the chain $(S + \partial R_j)\mres W_{\Sg}$ takes multiplicities 
  in the set~$\Sg\subseteq\GN$ defined by~\eqref{generators-intro}.
 \end{enumerate}
\end{prop}

The proof of Proposition~\ref{prop:approx_mult_noappendix}
will be given in Appendix~\ref{sect:approximation}.
Thanks to Proposition~\ref{prop:approx_mult_noappendix}, 
it is not restrictive to assume that
\begin{equation} \label{specialS3}
 S\mres W_{\Sg} \textrm{ takes its multiplicities in } \Sg
\end{equation}
in addition to~\eqref{specialS1}, \eqref{specialS2}.
Indeed, if~\eqref{specialS2} does not hold,
we replace~$S$ with with a chain of the form~$S + \partial R_j$
as given by Proposition~\ref{prop:approx_mult_noappendix},
we replace~$R$ with~$R+R_j$, then we use a diagonal 
argument to pass to the limit as~$j\to+\infty$.

\subsubsection{Construction of an \texorpdfstring{$\NN$}{N}-valued map with
prescribed singular set}
\label{sect:dipole}

Our next task is to construct a map~$w\colon\Omega\to\NN$,
with locally nice singularities, in such a way that~$\S(w) = S$.
To do so, we fix an open set~$W\csubset\Omega$ such that
$W_{\Sg}\csubset W$ and~$\partial W$ is transverse to~$\spt S$
(i.e., there exist triangulations of~$\partial W_{\Sg}$
and~$\spt S$ such that any simplex of the triangulation of~$\partial W_{\Sg}$
is transverse to any simplex of the triangulation of~$\spt S$).
We also fix a small parameter~$\eta>0$.

\begin{lemma} \label{lemma:dipole}
 For any~$W$ as above and any~$\eta>0$,
 there exists a map~$w\in W^{1,k-1}(\Omega, \, \NN)$
 that satisfies the following properties:
 \begin{enumerate}[label=(\roman*)]
  \item $w = w_*$ a.e. in~$\Omega\setminus W$;
  \item $w$ has a locally nice singularity at $(\spt S, \, Q_*)$,
  where~$Q_*\supseteq(\spt S)_{n-1}$ 
  is a locally~$(n-1)$-polyhedral set;
  \item $\S(w) = S$;
  \item $w_{|W}$ is~$\eta$-minimal.
 \end{enumerate}
\end{lemma}

Lemma~\ref{lemma:dipole} follows from 
Proposition~\ref{prop:dipole_simpl},
combined with the following lemma from~\cite{ABO2}:

\begin{lemma}[Lemma~9.3, \cite{ABO2}] \label{lemma:ABO-minimal}
 Let $K\subseteq\R^{n+k}$ be a $n$-simplex,
 and let $\delta$, $\gamma$ be positive parameters. 
 Let $u\colon U(K, \, \delta, \, \gamma)\to\NN$ be a map
 with nice singularity at $K$, and let~$\sigma\in\GN$
 the homotopy class of~$u$ around~$K$.
 Let~$\phi\colon\SS^{k-1}\to\NN$
 be a Lipschitz map in the homotopy class~$\sigma$.
 Then, there exists a map $\tilde{u}\colon U(K, \, \delta, \, \gamma)\to\NN$ 
 that satisfies the following properties:
 \begin{enumerate}[label=(\roman*)]
  \item $\tilde{u} = u$ on $\partial U(K, \, \delta, \, \gamma)$;
  \item $\tilde{u}$ has a nice singularity at $(K, \, \partial K)$;
  \item $\S(\tilde{u}) = \S(u)$;
  \item $\tilde{u}(x) = \phi(x^{\prime\prime}/|x^{\prime\prime}|)$ 
  for any $x = (x^\prime, \, x^{\prime\prime})\in
  U(K, \, \delta/4, \, \gamma/4)$.
 \end{enumerate}
\end{lemma}

In \cite{ABO2}, this result is proved in the 
particular case $\NN=\SS^{k-1}$. However, the same proof 
applies to a general target $\NN$: the map~$\tilde{u}$
is constructed by a suitable reparametrisation of the 
domain~$U(K, \, \delta, \, \gamma)$, and 
the arguments do not rely on properties of the target~$\NN$
other than (Lipschitz) path-connectedness. Property~(iii)
follows from Remark~\ref{remark:nice_sing} and~(ii), (iv).

\begin{proof}[Proof of Lemma~\ref{lemma:dipole}]
 By~\eqref{specialS2}, we 
 have~$\spt R \subseteq \overline{W}_{\Sg}\subseteq W$.
 By triangulating, we can write~$R$ in the form
 \[
  R = \sum_{i=1}^q \sigma_i \llbracket T_i\rrbracket,
 \]
 where the coefficients~$\sigma_i$ belong to~$\GN$ 
 and each~$T_i\csubset W$ is a convex $(n+1)$-simplex.
 We apply Proposition~\ref{prop:dipole_simpl},
 so to modify~$w_*$ in a neighbourhood of~$T_1$. We obtain
 a new map~$w_1\in W^{1,k}(\Omega, \, \NN)$ that
 has a locally nice singularity at~$\spt\S(w_*)\cup (T_1)_n\cup P_{y_*}$
 (with~$(T_1)_n$ is the $n$-skeleton of a suitable triangulation of~$T$),
 satisfies $w_1 = w_*$ on~$\Omega\setminus W$ 
 and~$\S(w_1) = \S(w_*) + \sigma_1\,\partial\llbracket T_1\rrbracket$.
 Now, we use Proposition~\ref{prop:dipole_simpl}
 to modify~$w_1$ in a neighbourhood of~$T_2$, and so on.
 By applying iteratively Proposition~\ref{prop:dipole_simpl},
 we construct a sequence of maps~$w_1$, $w_2$, \ldots, $w_q$.
 The map~$w_q$ has a locally nice singularity at
 $\spt\S(w_*) \cup (\spt R)_n \cup P_{y_*}$,
 satisfies~$w_q= w_*$ on~$\Omega\setminus W$ 
 and~$\S(w_q) = \S(w_*) + \partial R = S$.
 
 To complete the proof, it only remains to modify~$w_q$
 so as to satisfy~(iv).
 Since~$W\csubset\Omega$ has polyhedral boundary,
 the restriction~$S\mres W$ is a polyhedral chain.
 Let~$K$ be a $n$-face of~$\spt\S(w_*) \cup (\spt R)_n$. 
 The interior of~$K$ is contained in~$W$ and hence, 
 for sufficiently small parameters~$\delta>0$,
 $\gamma>0$, the interior of~$U(K, \, \delta, \, \gamma)$ 
 is contained in~$W$. Let~$\sigma_K\in\GN$ be
 the homotopy class of~$w_q$ around~$K$. By Remark~\ref{remark:phi-sigma},
 there exists a smooth map~$\phi_K\colon\SS^{k-1}\to\NN$
 that satisfies
 \begin{equation*}
   \int_{\SS^{k-1}} \abs{\nablaT\phi_K}^k \d\H^{k-1}
   \leq \int_{\SS^{k-1}} \abs{\nablaT\psi}^k \d\H^{k-1} + \eta
 \end{equation*}
 for any~$\psi\in W^{1,k}(\SS^{k-1}, \, \NN)\cap\sigma_K$.
 If~$\sigma_K = 0$, we choose~$\phi_K$ to be constant.
 We apply Lemma~\ref{lemma:ABO-minimal} to~$u = w_q$ and~$\phi = \phi_K$.
 By doing so for each~$K$, we obtain a 
 map~$w\colon\Omega\to\NN$ that agrees
 with~$w_*$ on~$\Omega\setminus\overline{W}$
 and is~$\eta$-minimal on~$W$.
 By Remark~\ref{remark:nice_sing}, $\S(w) = \S(w_q) = S$.
 Moreover, since~$\phi_K$ is constant if~$\sigma_K = 0$,
 $w$ has a locally nice singularity at~$(\spt S, \, Q_*)$
 where $Q_* := (\spt\S(w_*))_{n-1}\cup(\spt R)_{n-1}\cup P_{y_*}$.
 Therefore, $w$ has all the desired properties.
\end{proof}

\subsubsection{\texorpdfstring{$\eps$}{epsilon}-regularisation}

The map~$w\colon\Omega\to\NN$ given by Lemma~\ref{lemma:dipole}
has a singularity of codimension~$k$ at~$\spt S$, so
$w\notin W^{1,k}(\Omega, \, \NN)$ unless~$S = 0$.
Therefore, in order to define a recovery sequence, we need to 
regularise~$w$ around~$\spt S$. We do so by defining the maps
\begin{equation} \label{w_eps_def}
 w_\eps(x) := \min\left\{\frac{\dist(x, \, \spt S)}{\eps}, \, 1 \right\} w(x)
 \qquad \textrm{for any } x\in\Omega.
\end{equation}

\begin{lemma} \label{lemma:w_eps_int}
 For sufficiently small~$\eps$, the map~$w_\eps$ defined 
 by~\eqref{w_eps} belongs to~$(L^\infty\cap W^{1,k}_{\mathrm{loc}})(\Omega, \, \R^m)$.
 Moreover, the following properties holds.
 \begin{enumerate}[label=(\roman*), ref=\roman*]
  \item \label{w_eps_norm} 
  $w_\eps\to w$ strongly in~$W^{1,k-1}_{\mathrm{loc}}(\Omega)$ as~$\eps\to 0$.
  
  \item \label{w_eps} For any open set~$D\csubset\Omega$
  with polyhedral boundary, there holds
  \[
    \limsup_{\eps\to 0} \frac{E_\eps(w_\eps, \, D)}{\abs{\log\eps}}
    \leq C_{w, D} \, \M(S\mres\overline{D}),
  \]
  where the constant~$C_{w, D}$ depends
  on the map~$w$ and on~$\dist(D, \, \partial\Omega)$.
  
  \item \label{w_eps_sharp} We have
  \[
   \limsup_{\eps\to 0} \frac{E_\eps(w_\eps, \, W)}{\abs{\log\eps}}
   \leq (1 + C\eta)\,\M(S\mres\overline{W}) 
   + C\,\M(\overline{W}\setminus W_{\Sg}), 
  \]
  where~$C$ is a constant that depends only on~$\NN$, $\X$, $\RR$ and~$k$.
 \end{enumerate}
\end{lemma}
\begin{proof}
 Let~$Z_\eps:= \{x\in\R^{n+k}\colon \dist(x, \, \spt S)< \eps\}$,
 and let~$\zeta_\eps$ be the characteristic function of~$Z_\eps$
 (i.e.~$\zeta_\eps:= 1$ on~$Z_\eps$, $\zeta_\eps:= 0$ elsewhere).
 
 \setcounter{step}{0}
 \begin{step}[Proof of~(i)]
  Let~$D\csubset\Omega$ be an open set. We choose a number~$p$, 
  with~$1 < p < (k+1)/(k-1)$.
  Since~$w$ has a locally nice singularity at~$(\spt S, \, Q_*)$,
  at a.e.~point of~$D$ we have
  \begin{equation} \label{w_eps_0}
   \begin{split}
    \abs{\nabla w_\eps} 
    &\lesssim \left(\frac{\dist(\cdot, \, \spt S)\zeta_\eps}{\eps} 
    + 1 - \zeta_\eps\right)\abs{\nabla w} 
    + \frac{\zeta_\eps}{\eps} \\
    &\leq C_{w,D} \left(\left(\frac{\dist(\cdot, \, \spt S)\zeta_\eps}{\eps} 
    + 1 - \zeta_\eps\right)\left(\dist^{-1}(\cdot, \, \spt S) 
    + \dist^{-p}(\cdot, \, Q_*)\right)
    + \frac{\zeta_\eps}{\eps}\right) \\
    &\leq  C_{w,D} \left(
    (1 - \zeta_\eps)\dist^{-1}(\cdot, \, \spt S) 
    + \dist^{-p}(\cdot, \, Q_*)
    + \frac{\zeta_\eps}{\eps}\right) \!,
   \end{split}
  \end{equation}
  where~$C_{w,D}$ is a constant that depends on~$w$,
  $\dist(D, \, \partial\Omega)$ and~$p$, but not on~$\eps$.
  Therefore,
  \begin{gather*}
    \int_{D\cap Z_\eps}\abs{\nabla w_\eps}^{k-1}
    \leq C_{w,D}\left(\int_{D\cap Z_\eps}\dist^{p-kp}(x, \, Q_*)\,\d x
    + \frac{\L^{n+k}(D\cap Z_\eps)}{\eps^{k-1}}\right) 
  \end{gather*}
  By our choice of~$p$, we have $p-kp > -(k+1)$.
  Since~$Q_*$ has codimension~$k+1$,
  \cite[Lemma~8.3]{ABO2} implies that the function
  $\dist^{p-kp}(\cdot, \, Q_*)$ is integrable and that
  \begin{equation} \label{Psi_eps}
   \L^{n+k}(Z_\eps)\lesssim\eps^k.   
  \end{equation}
  As a consequence, we have
  \[
   \int_{D} \abs{\nabla (w - w_\eps)}^{k-1} 
   \lesssim \int_{D\cap Z_\eps}\abs{\nabla w}^{k-1}
   +\int_{D\cap Z_\eps}\abs{\nabla w_\eps}^{k-1} \to 0
  \]
  as~$\eps\to 0$, and~(i) follows.
 \end{step}
 
 \begin{step}[Proof of~(ii)]
  Let~$D\csubset\Omega$ and~$1 < p < 1 + 1/k$.
  From~\eqref{w_eps_0}, we deduce
  \begin{gather*}
   \begin{split}
    E_\eps(w_\eps, \, D)
    \leq C_{w,D}\left(\int_{D\setminus Z_\eps}\dist^{-k}(x, \, \spt S)\,\d x
    + \int_{D}\dist^{-kp}(x, \, Q_*)\,\d x
    + \frac{\L^{n+k}(D\cap Z_\eps)}{\eps^k}\right) 
   \end{split}
  \end{gather*}
  The second and third term at the right-hand side 
  are uniformly bounded with respect to~$\eps\to 0$,
  due to \cite[Lemma~8.3]{ABO2} and~\eqref{Psi_eps}.
  Since~$\spt S \cap D$ is contained in a finite union of
  polyhedra of codimension~$k$ or higher and~$D$ has polyhedral boundary,
  a computation based on Fubini theorem gives
  \[
   \limsup_{\eps\to 0} \frac{1}{\abs{\log\eps}}
   \int_{D\setminus Z_\eps}\dist^{-k}(\cdot, \, \spt S)
   \lesssim \H^{n}(\spt S\cap\overline{D}).
  \]
  On the other hand, 
  $\H^n(\spt S\cap\overline{D})\lesssim\M(S\mres\overline{D})$
  because the coefficient group~$(\GN, \, |\cdot|_*)$ is discrete
  (Proposition~\ref{prop:group_norm}). Thus, (ii) follows
  (and in particular, $w_\eps\in W^{1,k}_{\mathrm{loc}}(\Omega, \, \R^m)$).
 \end{step}
 
 \begin{step}[Proof of~(iii)]
  The inequality~\eqref{Psi_eps} implies
  \begin{equation} \label{w_eps_f}
   \limsup_{\eps\to 0}\frac{1}{\abs{\log\eps}\eps^k} \int_{W} f(w_\eps) 
   \lesssim \limsup_{\eps\to 0}\frac{\L^{n+k}(Z_\eps)}{\abs{\log\eps}\eps^k} = 0,
  \end{equation}
  so we only need to estimate the gradient terms.
  By Lemma~\ref{lemma:dipole}, $w_{|W}$ is~$\eta$-minimal,
  with nice singularity at~$((\spt S)\cap W, \, Q_*\cap W)$.
  Therefore, there exist positive numbers~$\delta$, $\gamma$, 
  a triangulation of~$(\spt S)\cap W$ and,
  for any~$n$-simplex~$K$ of the triangulation, 
  a Lipschitz map~$\phi_K\colon\SS^{k-1}\to\NN$ that satisfy
  the conditions~(i)--(iii) in Definition~\ref{def:minimal}. 
  By taking smaller~$\delta$, $\gamma$ if necessary,
  we can also assume that the interior of~$U(K, \, \delta, \, \gamma)$
  is contained in~$W$, for any~$n$-simplex~$K$ of the triangulation.
  Let~$F:= W\setminus \cup_K U(K, \, \delta, \, \gamma)$,
  where the union is taken over all $n$-simplices~$K$ of the triangulation.
  We estimate separately the energy on~$F$ and on 
  each~$U(K, \, \delta, \, \gamma)$.
  
  Let us estimate the energy on~$F$ first.
  Since~$Q_*\supseteq(\spt S)_{n-1}$, the
  definition~\eqref{U-diamond} of~$U(K, \, \delta, \, \gamma)$
  implies that
  \begin{equation} \label{outofdiamond}
   \dist(x, \, Q^*)\gtrsim\dist(x, \, \spt S)
   \qquad \textrm{for any } x\in F.
  \end{equation}
  (The proportionality constant at 
  the right-hand side depends on~$\delta$, $\gamma$.)
  Let us choose a number~$p$ with~$1 < p < 1 + 1/k$.
  Since~$w$ has a locally nice singularity
  at~$(\spt S$, $Q_*)$, we obtain
  \begin{equation*}
   \begin{split}
    \abs{\nabla w(x)} \leq C_{w,W} 
    \left(\dist^{-1}(x, \, \spt S) + \dist^{-p}(x, \, Q_*)\right)
    \stackrel{\eqref{outofdiamond}}{\leq}  C_{w,W} \dist^{-p}(x, \, Q_*)
   \end{split}
  \end{equation*}
  for a.e.~$x\in F$ and some constant~$C_{w,W}$ that 
  depends on~$w$, $W$, $p$, $\delta$ and~$\gamma$. This implies
  \begin{equation*}
   \begin{split}
    \frac{1}{k} \int_{F}\abs{\nabla w_\eps}^k
    \leq C_{w,W} \int_{F}\dist^{-kp}(x, \, Q_*)\, \d x 
    + \frac{C_{w,W}}{\eps^k} \L^{n+k}(Z_\eps).
   \end{split}
  \end{equation*}
  The right-hand side is uniformly bounded with respect to~$\eps$,
  due to \cite[Lemma~8.3]{ABO2} and~\eqref{Psi_eps}, so
  \begin{equation} \label{w_eps_F}
   \limsup_{\eps\to 0} \frac{1}{k\abs{\log\eps}}
   \int_{F}\abs{\nabla w_\eps}^k = 0.
  \end{equation}
  Next, we estimate the energy on~$U(K, \, \delta, \, \gamma)$, with~$K$
  an $n$-dimensional simplex in the triangulation of~$(\spt S)\cap W$.
  We write~$U := U(K, \, \delta, \, \gamma)$ for brevity, 
  and let~$x = (x^\prime, \, x^{\prime\prime})$ denote
  the variable in~$U$, as in~\eqref{U-diamond}. Using
  Condition~(ii) in Definition~\eqref{def:minimal},
  we can compute explicitly the gradient of~$w_\eps$, 
  and we obtain
  \[
   \begin{split}
    \abs{\nabla w_\eps(x)} 
    &\leq \left(\frac{C\zeta_\eps(x)}{\eps} + 
    \frac{1 - \zeta_\eps(x)}{\abs{x^{\prime\prime}}}\right)
    \abs{(\nablaT\phi_K)\left(\frac{x^{\prime\prime}}{\abs{x^{\prime\prime}}}\right)} 
    + \frac{C\zeta_\eps(x)}{\eps} \\
    &\leq \frac{1 - \zeta_\eps(x)}{\abs{x^{\prime\prime}}}
    \abs{(\nablaT\phi_K)\left(\frac{x^{\prime\prime}}{\abs{x^{\prime\prime}}}\right)} 
    + \frac{C_{w,W}\,\zeta_\eps(x)}{\eps}
   \end{split}
  \]
  for a.e.~$x\in U$, where~$\nablaT$
  denotes the tangential gradient on~$\SS^{k-1}$.
  (In the second inequality, 
  we use that~$\phi_K$ is Lipschitz.)
  We raise to the power~$k$ both
  sides of this inequality, integrate over~$U$, apply
  Fubini theorem and pass to polar coordinates 
  for the integral with respect to~$x^{\prime\prime}$:
  \begin{equation*}
   \begin{split}
    \frac{1}{k} \int_{U} \abs{\nabla w_\eps}^k &\leq 
     \left(\frac{1}{k}\int_{\SS^{k-1}} \abs{\nablaT\phi_K}^k \d\H^{k-1}\right)
    \left(\int_{\eps}^{\delta} \frac{\d\rho}{\rho} \right) \H^{n}(K) 
    + \frac{C_{w,W}\,\L^{n+k}(Z_\eps)}{\eps^k} \\
    &\stackrel{\eqref{Psi_eps}}{\leq} 
    \left(\frac{1}{k} \int_{\SS^{k-1}} \abs{\nablaT\phi_K}^k \d\H^{k-1}\right)
    \left(\log\frac{\delta}{\eps}\right)\H^{n}(K) + C_{w,W}.
   \end{split}
  \end{equation*}
  Using Condition~(iii) in Definition~\ref{def:minimal}, we deduce
  \begin{equation} \label{w_eps_U}
   \begin{split}
    \limsup_{\eps\to 0} \frac{1}{k\abs{\log\eps}} 
     \int_{U} \abs{\nabla w_\eps}^k 
    &\leq \left(E_{\min}(\sigma_K) + \eta\right)\H^{n}(K),
   \end{split}
  \end{equation}
  where~$\sigma_K\in\GN$ is the homotopy class of~$\phi_K$
  and~$E_{\min}(\sigma_K)$ is defined by~\eqref{I_min-intro}.
  We need to distinguish two cases, depending on whether
  the interior of~$K$ is contained~$W_{\Sg}$ or not.
  If the interior of~$K$ is contained in~$W_{\Sg}$,
  then~$\sigma_K\in\Sg$ because of~\eqref{specialS3},
  and~\eqref{w_eps_U} becomes
  \begin{equation} \label{w_eps_U_Sg}
   \begin{split}
    \limsup_{\eps\to 0} \frac{1}{k\abs{\log\eps}} 
     \int_{U} \abs{\nabla w_\eps}^k 
     \leq \left(\abs{\sigma_K}_* + \eta\right)\H^{n}(K) 
    \leq (1 + C\eta) \, \M(S\mres K)
   \end{split}
  \end{equation}
  for some constant~$C$ that depends only on~$\NN$.
  (Here again, we have used that~$\M(S\mres K)\gtrsim\H^n(K)$,
  due to Proposition~\eqref{prop:group_norm}.) 
  Suppose now that the interior of~$K$ is not contained in~$W_{\Sg}$.
  The intersection between the interior of~$K$ and~$\partial W_{\Sg}$
  has dimension~$n-1$ at most, because
  we have taken~$\partial W_{\Sg}$ to be transverse to~$\spt S$.
  Therefore, up to refining the triangulation, we may assume 
  that the interior of~$K$ is contained in~$W\setminus\overline{W}_{\Sg}$.
  Then, thanks to~\eqref{specialS1} and~\eqref{specialS2},
  $S$ agrees with~$\S(w_*)$ in the interior of~$K$.
  The chain~$\S(w_*)$ takes its multiplicity in a finite set
  that depends only on~$\NN$, $\X$, $\RR$ (by Lemma~\ref{lemma:u*}) 
  and hence, $E_{\min}(\sigma_K)\leq C$.
  Thus, \eqref{w_eps_U} becomes
  \begin{equation} \label{w_eps_U_noSg}
   \begin{split}
    \limsup_{\eps\to 0} \frac{1}{k\abs{\log\eps}} 
     \int_{U} \abs{\nabla w_\eps}^k 
    \lesssim \H^{n}(K) \lesssim \M(S\mres K).
   \end{split}
  \end{equation}
  Combining~\eqref{w_eps_f}, \eqref{w_eps_F}, 
  \eqref{w_eps_U_Sg} and~\eqref{w_eps_U_noSg}, 
  the inequality~\eqref{w_eps_sharp} follows.
  \qedhere
 \end{step}
\end{proof}

\subsubsection{Proof of Theorem~\ref{th:main}.(ii) and 
Proposition~\ref{prop:main-nobd}.(ii)}
\label{sect:proof_ii}


\begin{proof}[Proof of Theorem~\ref{th:main}.(ii)]
 Let~$S\in\mathscr{C}(\Omega, \, v)$ be a finite-mass chain, and 
 let~$\eta>0$ be a small number. Given a countable
 sequence~$\eps\to 0$, we aim to 
 construct~$u_\eps\in (L^\infty\cap W^{1,k}_v)(\Omega, \, \R^m)$,
 where~$\eps$ ranges in a non-relabelled subsequence, in such a way that
 \begin{gather} 
  \lim_{\eps\to 0}\int_{B^m(0, \, \dist(\NN, \, \X))}
   \F(\S_y(u_\eps) - S) \, \d y = 0,
   \label{limsup-goal-0} \\
  \begin{split}
    \limsup_{\eps\to 0}\frac{E_\eps(u_\eps)}{\abs{\log\eps}} \leq 
    (1 + C\eta)\,\M(S) + C\eta, \label{limsup-goal}
  \end{split}
 \end{gather}
 where~$C$ is a constant that does not depend on~$\eta$. 
 If we do so, the theorem will follow, 
 by a diagonal argument. As we have seen, thanks to
 Proposition~\ref{prop:approx-noappendix}, 
 Proposition~\ref{prop:approx_mult_noappendix} and
 a diagonal argument, it is not restrictive
 to assume that~$S$ satisfies~\eqref{specialS1},
 \eqref{specialS2}, \eqref{specialS3}. Moreover, we have
 \begin{equation} \label{limsup0}
  S\mres\partial\Omega 
  \stackrel{\eqref{specialS1},\eqref{specialS2}}{=} 
  \S(w_*)\mres\partial\Omega 
  = \S_{y_*}(u_*)\mres\partial\Omega = 0
 \end{equation}
 by~Lemma~\ref{lemma:u*} and hence, by taking a larger~$W_{\Sg}$
 if necessary, we can assume without loss of generality that
 \begin{equation} \label{limsup1}
   \M(S\mres(\overline{\Omega}\setminus W_{\Sg})) + 
   \int_{\Omega\setminus W_{\Sg}} \abs{\nabla u_*}^k \leq \eta.
  \end{equation}
 
 \setcounter{step}{0}
 \begin{step}[Definition of~$u_\eps$]
  To define the recovery sequence near the boundary 
  of~$\Omega$, we apply Lemma~\ref{lemma:extension}
  to~$u_*$ and~$y_*$, and consider the map
  \[
   w_{\eps,y_*} :=  (\xi_\eps\circ\psi)\circ(u_* - y_*) \cdot w_*
   = (\xi_\eps\circ\psi)\circ(u_* - y_*) \cdot (\RR_{y_*}\circ u_*)
  \]
  (with~$\xi_\eps$, $\psi$ as in Lemma~\ref{lemma:extension}).
  Thanks to Lemma~\ref{lemma:extension} and an averaging argument, 
  by possibly modifying the value of~$y_*$ we have
  \begin{equation} \label{limsup1.5}
   \begin{split}
    E_\eps(w_{\eps,y_*}, \, \Omega\setminus W_{\Sg}) 
     &+ \eps^{-k} \L^{n+k}\left\{x\in\Omega\setminus W_{\Sg}
     \colon w_{\eps,y_*}(x)\neq w_{*}(x)\right\} \\
     &\lesssim \abs{\log\eps}
     \int_{\Omega\setminus W_{\Sg}}\abs{\nabla u_*}^k + 1
     \stackrel{\eqref{limsup1}}{\lesssim} 
     \eta\abs{\log\eps} + 1 .
   \end{split}
  \end{equation}
  Our recovery sequence will coincide with~$w_\eps$
  given by~\eqref{w_eps} in~$W$, where~$W_{\Sg}\csubset W\csubset\Omega$ 
  is the open set introduced in Section~\ref{sect:dipole}. 
  We need to interpolate between~$w_\eps$ and~$w_{\eps,y_*}$
  near~$W$. To this end, we take a small parameter~$\theta>0$, 
  and we let $D_\theta := \{x\in\Omega\setminus\overline{ W}
  \colon\dist(x, \,  W)< \theta\}$. 
  For~$x\in D_\theta$, let~$t_\theta(x):= \theta^{-1}\dist(x, \,  W)$.
  We define 
  \[
   u_\eps(x) := \begin{cases}
                 w_\eps(x) & \textrm{if } x\in\overline{ W} \\
                 (1 - t_\theta(x))w_{\eps}(x) + t_\theta(x) w_{\eps, y_*}(x) 
                  & \textrm{if } x\in D_\theta \\              
                 w_{\eps, y_*}(x) & \textrm{if } 
                  x\in\Omega\setminus(\overline{ W}\cup D_\theta).
                \end{cases}
  \]
  We have $u_\eps\in (L^\infty\cap W^{1,k}_v)(\Omega, \, \R^m)$
  and~$\sup_{\eps}\|u_\eps\|_{L^\infty(\Omega)}<+\infty$.
 \end{step}
 
 \begin{step}[Bounds on $E_\eps(u_\eps)$]
  The energy of~$u_\eps$ on~$\Omega\setminus(\overline{ W}\cup D_\theta)$
  is bounded from above by~\eqref{limsup1.5}. 
  The energy of~$u_\eps$ is bounded from above by Lemma~\ref{lemma:w_eps_int}:
  \begin{equation} \label{limsup2}
   \begin{split}
    \limsup_{\eps\to 0} \frac{E_\eps(u_\eps, \,  W)}{\abs{\log\eps}}
    &\leq (1 +  C\eta)\,\M(S) + C\,\M(S\mres(\overline{W}\setminus W_{\Sg})) \\
    &\stackrel{\eqref{limsup1}}{\leq} 
    (1 +  C\eta)\,\M(S) + C\eta.
   \end{split}
  \end{equation}
  It remains to estimate the energy of~$u_\eps$ on~$D_\theta$.
  We first note that
  $\abs{\nabla t_\theta} = \theta^{-1}$ and hence,
  \begin{equation} \label{limsup2.1}
   \abs{\nabla u_\eps} \leq \abs{\nabla w_\eps} + \abs{\nabla w_{\eps, y_*}}
   + \theta^{-1} \abs{w_\eps - w_{\eps, y_*}}
   \leq \abs{\nabla w_\eps} + \abs{\nabla w_{\eps, y_*}}
   + C\theta^{-1}.
  \end{equation}
  By Lemma~\ref{lemma:dipole}, $w = w_*$ a.e.~in~$\Omega\setminus W$
  and in particular, $w = w_*$ a.e.~in~$D_\theta$. Therefore,
  for a.e.~$x\in D_\theta$ such that $w_\eps(x) = w(x)$
  and~$w_{\eps, y_*}(x) = w_*(x)$, we have
  $u_\eps(x) = w_*(x)\in\NN$. Since the maps~$u_\eps$
  are uniformly bounded, we deduce that
  \begin{equation} \label{limsup2.2}
   f(u_\eps) \lesssim \one_{\{w_\eps\neq w\}} 
   + \one_{\{w_{\eps, y_*}\neq w_*\}}  .
  \end{equation}
  From~\eqref{limsup2.1} and~\eqref{limsup2.2}, we obtain
  \begin{equation*} 
   \begin{split}
    E_\eps(u_\eps, \, D_\theta) 
    &\lesssim E_\eps(w_\eps, \, D_\theta) 
     + E_\eps(w_{\eps,y_*}, \, D_\theta) 
     + \theta^{-k}\L^{n+k}(D_\theta)\\
    &\qquad\qquad + \eps^{-k} \L^{n+k}\{w_\eps\neq w\}
     + \eps^{-k} \L^{n+k}(D_\theta\cap \{w_{\eps, y_*}\neq w_*\}).
   \end{split}
  \end{equation*}
  The set~$\{w_\eps\neq w\}$ is the $\eps$-neighbourhood of
  $\spt S$, which is a locally polyhedral set of codimension~$k$, so
  \begin{equation} \label{limsup3}
   \L^{n+k}\{w_\eps\neq w\}\lesssim\eps^k
  \end{equation}
  (see Lemma~\cite[Lemma~8.3]{ABO2} and~\eqref{Psi_eps}).
  Moreover, $\L^{n+k}(D_\theta)\lesssim\theta$. Then,
  \begin{equation} \label{limsup3.5}
   \begin{split}
    E_\eps(u_\eps, \, D_\theta) 
    &\lesssim E_\eps(w_\eps, \, D_\theta) 
     + E_\eps(w_{\eps,y_*}, \, D_\theta) 
     + \eps^{-k} \L^{n+k}(D_\theta\cap \{w_{\eps, y_*}\neq w_*\})
     + \theta^{1-k} +1.
   \end{split}
  \end{equation}
  We choose~$\theta = \theta(\eps)$ in such a way that 
  $\theta(\eps)\to 0$ and $\theta(\eps)^{1-k}|\!\log\eps|^{-1}\to 0$
  as~$\eps\to 0$; for instance, we take $\theta(\eps):= |\!\log\eps|^{-1/(2k - 2)}$.
  With this choice of~$\theta$,
  from~\eqref{limsup3.5}, Lemma~\ref{lemma:w_eps_int}
  and~\eqref{limsup1.5} we deduce
  \begin{equation*} 
   \begin{split}
    \frac{E_\eps(u_\eps, \, D_{\theta(\eps)})}{\abs{\log\eps}}
    &\leq C_{w,W} \, \M(S\mres\overline{D_{\theta(\eps)}}) 
     + C\eta + \mathrm{o}_{\eps\to 0}(1) 
   \end{split}
  \end{equation*}
  where~$C_{w,W}$ is a constant that depends 
  on~$w$ and~$\dist(W, \, \partial\Omega)$,
  but not on~$\eps$. By taking the limit as~$\eps\to 0$,
  and recalling that~$\partial W$ is transverse to~$\spt S$,
  we conclude that
  \begin{equation} \label{limsup4}
   \begin{split}
    \limsup_{\eps\to 0}
     \frac{E_\eps(u_\eps, \, D_{\theta(\eps)})}{\abs{\log\eps}}
    &= C_{w,W} \, \M(S\mres\partial W) + C\eta = C\eta.
   \end{split}
  \end{equation}
  Combining~\eqref{limsup1.5}, \eqref{limsup2} 
  and~\eqref{limsup4}, the inequality~\eqref{limsup-goal} follows.
 \end{step}

 \begin{step}[$u_\eps\to w$ in~$W^{1,k-1}(\Omega)$]
  To complete the proof, it only remains to check~\eqref{limsup-goal-0}.
  As an intermediate step, we prove that
  $u_\eps \to w$ strongly in~$W^{1,k-1}(\Omega)$.
  Up to extraction of a subsequence, 
  we have $w_\eps\to w$ in~$W^{1,k-1}(W)$
  and~$w_{\eps, y_*}\to w_* = w$ in~$W^{1,k-1}(\Omega\setminus W)$
  by Lemma~\ref{lemma:w_eps_int} and Lemma~\ref{lemma:extension}, 
  respectively. Thus, we only need to check that
  \begin{equation} \label{limsup5}
   \int_{D_{\theta(\eps)}}\abs{\nabla u_\eps}^{k-1} \to 0
   \qquad \textrm{as } \eps\to 0.
  \end{equation}
  From~\eqref{limsup2.1}, using that~$w = w_*$ 
  a.e. on~$D_{\theta(\eps)}$, we deduce
  \begin{equation*} 
   \begin{split}
    \int_{D_{\theta(\eps)}} \abs{\nabla u_\eps}^{k-1}  \lesssim
    \int_{D_{\theta(\eps)}} \left( \abs{\nabla w_\eps}^{k-1}
     + \abs{\nabla w_{\eps, y_*}}^{k-1}
     +\frac{\abs{w_\eps - w}^{k-1}}{\theta(\eps)^{k-1}}
     + \frac{\abs{w_{\eps, y_*} - w_*}^{k-1}}{\theta(\eps)^{k-1}} \right) \!.
   \end{split}
  \end{equation*}
  The sequences~$w_\eps$ and~$w_{\eps, y_*}$ are strongly 
  compact in~$W^{1,k-1}_{\mathrm{loc}}(\Omega)$,
  $W^{1,k-1}(\Omega)$ respectively.
  Since $\L^{n+k}(D_{\theta(\eps)})\to 0$, we have
  \begin{equation*} 
   \begin{split}
    \int_{D_{\theta(\eps)}} \left( \abs{\nabla w_\eps}^{k-1}
     + \abs{\nabla w_{\eps, y_*}}^{k-1} \right) \to 0
     \qquad \textrm{as } \eps\to 0.
   \end{split}
  \end{equation*}
  Then, keeping in mind that $w_{\eps, y_*}$, $w_\eps$ are uniformly bounded,
  and using~\eqref{limsup1.5}, \eqref{limsup3}, we obtain
  \begin{equation*} 
   \begin{split}
    \int_{D_{\theta(\eps)}} \abs{\nabla u_\eps}^{k-1} 
     \lesssim \eps^k\abs{\log\eps} \theta(\eps)^{1-k}
     + \mathrm{o}_{\eps\to 0}(1).
   \end{split}
  \end{equation*}
  Now~\eqref{limsup5} follows, because
  we have chosen~$\theta(\eps)$ in such a way that 
  $\theta(\eps)^{1-k}|\!\log\eps|^{-1}\to 0$.
 \end{step}
 
 \begin{step}[Proof of~\eqref{limsup-goal-0}]
  Let us take a larger, bounded domain~$\Omega^\prime\supset\!\supset\Omega$
  and a map $V\in(L^\infty\cap W^{1,k})
  (\Omega^\prime\setminus\overline{\Omega}, \, \R^m)$
  with trace~$v$ on~$\partial\Omega$. We define
  \[
   \tilde{u}_{\eps} := \begin{cases}
                          u_\eps & \textrm{on } \Omega \\
                          V & \textrm{on } \Omega^\prime\setminus\Omega,
                         \end{cases} \qquad
   \tilde{w} := \begin{cases}
                          w & \textrm{on } \Omega \\
                          V & \textrm{on } \Omega^\prime\setminus\Omega.
                \end{cases}
  \]
  Since the traces of~$u_\eps$, $w$ agree with that 
  of~$V$ on~$\partial\Omega$, we have 
  $\tilde{u}_\eps\in (L^\infty\cap W^{1,k})(\Omega^\prime, \, \R^m)$, 
  $\tilde{w}\in (L^\infty\cap W^{1,k-1})(\Omega^\prime, \, \R^m)$,
  $\sup_\eps\|\tilde{u}_{\eps}\|_{L^\infty(\Omega^\prime)}<+\infty$
  and~$\tilde{u}_{\eps}\to \tilde{w}$ strongly in~$W^{1,k-1}(\Omega^\prime)$.
  By continuity of~$\S$ \cite[Theorem~3.1]{CO1},
  this implies
  \begin{equation*} 
   \int_{B^m(0, \, \dist(\NN, \, \X))} \F_{\Omega^\prime}\left(
   \S_{y}(\tilde{u}_\eps) - 
   \S_{y}(\tilde{w})\right) \d y \to 0
   \qquad \textrm{as } \eps\to 0.
  \end{equation*}
  Since~$\tilde{u}_\eps=\tilde{w}$ a.e.
  on~$\Omega^\prime\setminus\overline{\Omega}$ 
  and the operator~$\S$ is local \cite[Corollary~1]{CO1},
  we have $\S_{y}(\tilde{u}_\eps) 
  \mres(\Omega^\prime\setminus\overline{\Omega}) = 
  \S_{y}(\tilde{w}) \mres(\Omega^\prime\setminus\overline{\Omega})$
  for a.e.~$y$, and hence $\S_{y}(\tilde{u}_\eps) - 
  \S_{y}(\tilde{w})$ is supported in~$\overline{\Omega}$
  for a.e.~$y$. 
  For chains supported in a compact subset of~$\Omega^\prime$,
  the relative flat norm~$\F_{\Omega^\prime}$ is equivalent to~$\F$
  (see e.g. \cite[Remark~2.2]{CO1}). Therefore, we have
  \begin{equation} \label{limsup6}
   \int_{B^m(0, \, \dist(\NN, \, \X))}
    \F\left(\S_{y}(\tilde{u}_\eps) 
    - \S_{y}(\tilde{w})\right) \d y \to 0
   \qquad \textrm{as } \eps\to 0.
  \end{equation}
  By \cite[Eq.~(3.25)]{CO1} we have
  $\S_{y}(\tilde{u}_\eps)\mres\partial\Omega = 0$ and
  $\S_{y}(\tilde{w})\mres\partial\Omega 
  = \S_{y}(V)\mres\partial\Omega = 0$ for a.e.~$y$, so
  \begin{equation*} 
   \begin{split}
    \S_{y}(\tilde{u}_\eps) - \S_{y}(\tilde{w})
    = \left(\S_{y}(\tilde{u}_\eps) - \S_{y}(\tilde{w})\right)\mres\Omega.
   \end{split}
  \end{equation*}
  Since~$\tilde{u}_\eps=u_\eps$ and~$\tilde{w} = w$ a.e. on~$\Omega$,
  \cite[Corollary~1]{CO1} implies
  \begin{equation*}
   \begin{split}
    \S_{y}(\tilde{u}_\eps) - \S_{y}(\tilde{w})
    = \left(\S_{y}(u_\eps) - \S(w) \right)\mres\Omega
    = \S_{y}(u_\eps) - S\mres\Omega 
   \end{split}
  \end{equation*}
  and finally, recalling~\eqref{limsup0}, we obtain
    \begin{equation} \label{limsup7}
   \begin{split}
    \S_{y}(\tilde{u}_\eps) - \S_{y}(\tilde{w})
     = \S_{y}(u_\eps) - S.
   \end{split}
  \end{equation}
  From \eqref{limsup6} and~\eqref{limsup7} we 
  deduce~\eqref{limsup-goal-0}, and the proof is complete. \qedhere
 \end{step}
\end{proof}

The proof of Proposition~\ref{prop:main-nobd}.(ii)
follows along the same lines, and in fact, is even simpler.

\begin{proof}[Proof of Proposition~\ref{prop:main-nobd}]
 Let~$S$ be an $n$-dimensional relative boundary of finite mass
 --- that is, $S$ has the form~$S = (\partial R)\mres\Omega$,
 where~$R$ is an $(n+1)$-chain of finite mass such that
 $\M(\partial R)<+\infty$. By a density argument, 
 we can assume without loss of generality that~$R$ is polyhedral.
 By Proposition~\ref{prop:approx_mult_noappendix}, we can also assume 
 that~$\partial R$ takes its multiplicities in the set~$\Sg\subseteq\GN$
 defined by~\eqref{generators-intro}. Finally, by translating 
 the support of~$R$ and applying Thom's transversality theorem, we
 can assume that 
 \begin{equation} \label{limsupnobd0}
  (\partial R)\mres\partial\Omega = 0.
 \end{equation}
 Let~$w_*\in\NN$ be a constant, and let~$\eta>0$
 be a small parameter. We repeat the same arguments of 
 Lemma~\ref{lemma:dipole} and modify the constant map~$w_*$
 in a neighbourhood of~$\spt R$.
 We obtain a new map~$w\colon\R^{n+k}\to\NN$ that
 \begin{enumerate}[label=(\roman*)]
  \item has a nice singularity at~$(\spt(\partial R), \, (\spt(\partial R))_{n-1})$;
  \item satisfies~$\S(w) = \S(w_*) + \partial R = \partial R$; 
  \item is~$\eta$-minimal.
 \end{enumerate}
 Let
 \[
  u_\eps(x) := \left\{\frac{\dist(x, \, \spt(\partial R))}{\eps},
  \, 1\right\} w(x) \qquad \textrm{for } x\in\R^{n+k}.
 \]
 By the same computations as in Lemma~\ref{lemma:w_eps_int}, we
 obtain that~$w_\eps\to w$ strongly in~$W^{1,k-1}(\R^{n+k})$
 and that
 \begin{equation} \label{limsupnobd1}
  \limsup_{\eps\to 0} 
  \frac{E_\eps(u_\eps, \, \Omega)}{\abs{\log\eps}}
  \leq (1 + C\eta) \, \M((\partial R)\mres\Omega^\prime),
 \end{equation}
 where~$\Omega^\prime\supset\!\supset\Omega$ is any open set,
 with polyhedral boundary, such that~$\partial\Omega$ 
 is transverse to~$\spt(\partial R)$. (The latter condition 
 is generic, by Thom's transversality theorem.)
 The continuity of~$\S$ \cite[Theorem~3.1]{CO1},
 together with the fact that the operator~$\S$
 is local~\cite[Corollary~1]{CO1}, implies 
 $\S(u_{\eps|\Omega})\to \S(w)\mres\Omega
 = (\partial R)\mres\partial\Omega = S$ in~$Y$.
 We let~$\Omega^\prime\searrow\Omega$ in~\eqref{limsupnobd1},
 and we deduce
 \[
  \limsup_{\eps\to 0} 
  \frac{E_\eps(u_\eps, \, \Omega)}{\abs{\log\eps}}
  \leq (1 + C\eta) \, \M((\partial R)\mres\overline{\Omega})
  \stackrel{\eqref{limsupnobd0}}{=}
  (1 + C\eta) \, \M(S).
 \]
 Since~$\eta$ may be taken arbitrarily small,
 Proposition~\ref{prop:main-nobd}.(ii) follows,
 by a diagonal argument.
\end{proof}

\section{Compactness and lower bounds}
\label{sect:liminf}

\subsection{A local compactness result}
\label{sect:loc_liminf}

The aim of this section is to prove Statement~(i) of Theorem~\ref{th:main}.
As an intermediate step, we will prove the following result,
which is a local version of Theorem~\ref{th:main}.(i).
We recall that we have fixed a number~$\delta^*\in (0, \, \dist(\NN, \, \X))$
and that $B^* := B^m(0, \, \delta^*)\subseteq\R^m$.

\begin{prop} \label{prop:liminf-local}
 Let~$U\csubset U^\prime$ be bounded domains in~$\R^{n+k}$.
 Let~$(u_\eps)_\eps$ be a countable sequence of maps
 in~$W^{1,k}(U^\prime, \, \R^m)$ such that
 \begin{equation} \label{H}
  \sup_{\eps>0} \frac{E_\eps(u_\eps, \, U^\prime)}{\abs{\log\eps}} < +\infty.
 \end{equation}
 Then, there exist a (non-relabelled) subsequence
 and a finite-mass chain $S\in\M_n(\overline{U}^\prime; \, \GN)$ such that
 \begin{gather}
  \lim_{\eps\to 0}\int_{B^*} \F_U(\S_y(u_\eps) - S) \, \d y = 0 
   \label{liminf:flat} \\
  \M(S\mres U) \leq \liminf_{\eps\to 0} 
   \frac{E_\eps(u_\eps, \, U^\prime)}{\abs{\log\eps}} \label{liminf:mass}
 \end{gather}
 ($\F_U$ is the relative flat norm, see~\eqref{flat_relative}).
\end{prop}

Throughout this section, we fix bounded
domains~$U\csubset U^\prime\subseteq\R^{n+k}$ 
and a countable sequence~$(u_\eps)$ in~$W^{1,k}(U^\prime, \, \R^m)$
that satisfies~\eqref{H}. By an approximation argument,
using the continuity of~$\S$ 
(Proposition~\ref{prop:S} and~\cite[Theorem~3.1]{CO1}),
we can assume without loss of generality that the maps~$u_\eps$ are \emph{smooth and bounded}.
For any~$\eps>0$ and~$y\in B^*$, we define the measure
\begin{equation} \label{mu_y,eps}
 \mu_{\eps, y} (A) := \M(\S_y(u_\eps)\mres (A\cap U^\prime))
 \qquad \textrm{for any Borel set } A\subseteq\R^{n+k}.
\end{equation}
Thanks to~\eqref{S:mass} (Proposition~\ref{prop:S}),
$\mu_{\eps,y}$ is a 
bounded Radon measure for a.e.~$y$.

\subsubsection{Choice of a grid}

As in~\cite{ABO2}, we define a grid~$\GG$ of size~$h > 0$
as a collection of closed cubes of the form 
\begin{equation} \label{grid}
 \GG = \GG(a, \, h) := \left\{a + h z + [0, \ h]^{n+k} \colon z\in\Z^{n+k}\right\} \! ,
\end{equation}
for some~$a\in\R^{n+k}$. For~$j\in\N$, $0 \leq j\leq n+k$, 
we denote by~$\GG_j$ the collection of the (closed)~$j$-cells of~$\GG$, 
and we define the~$j$-skeleton of~$\GG$, $R_j := \cup_{K\in\GG_j} K$.
We let~$\tilde{R}_{k}$ be the union of all the cells $K\in\GG_k$
that are parallel to the $k$-plane spanned by~$\{\e_{n+1}, \, \ldots, \, \e_{n+k}\}$.
Given an open set~$V\subseteq U^\prime$, we denote by~$R_k(V)$ the union of
the~$k$-cells~$K\in\GG$ such that~$K\cap V \neq\emptyset$
(so $R_k(V)\supseteq R_k\cap V$). Given~$\GG = \GG(a, \, h)$, the grid
\[
 \GG^\prime := \GG(a + (h/2, \, h/2, \, \ldots, \, h/2), \, h)
\]
will be called the dual grid of~$\GG$.
We will denote by~$\GG^\prime_k$ the collections of~$k$-cells of~$\GG^\prime$
and by~$R^\prime_k$ its~$k$-skeleton. For each $K\in\GG_k$
there exists a unique~$K^\prime\in\GG^\prime_{n}$,
called the dual cell of~$K$, such that~$K\cap K^\prime\neq\emptyset$.

We are now going to construct a sequence of grids~$\GG^\eps$
with suitable properties. The construction is
analogous to~\cite[Lemma~3.11]{ABO2}. 
Let us take a function~$h\colon (0, 1)\to\R^+$ such that
\begin{equation} \label{h}
 \eps^\alpha \ll h(\eps) \ll \abs{\log\eps}^{-1} \qquad 
 \textrm{for any } \alpha >0, \textrm{ as } \eps\to 0.
\end{equation}
For instance, we may take $h(\eps) := \abs{\log\eps}^{-2}$.

\begin{lemma} \label{lemma:grid}
 For any fixed parameter~$\delta > 0$ and any~$\eps < 1$
 there exists a grid~$\GG^\eps$ of size~$h(\eps)$ that satisfies the following properties:
 \begin{gather}
  h(\eps)^n \, E_\eps(u_\eps, \, \tilde{R}^\eps_{k}\cap U^\prime) \leq (1 +  \delta) 
   E_\eps(u_\eps, \, U^\prime) \label{grid_ktilde} \\
  h(\eps)^n \, E_\eps(u, \, R^\eps_{k}\cap U^\prime) 
   \lesssim \delta^{-1} E_\eps(u_\eps, \, U^\prime) \label{grid_k} \\
  h(\eps)^{n+1} \, E_\eps(u, \, R^\eps_{k-1}\cap U^\prime) 
   \lesssim \delta^{-1} E_\eps(u_\eps, \, U^\prime) \label{grid_k-1} \\
  h(\eps)^n \int_{B^*} \int_{U^\prime} 
   \dfrac{\d\mu_{\eps, y}(x)}{\dist^n(x, \, R^\eps_{k-1})} \, \d y
   \lesssim \delta^{-1} E_\eps(u_\eps, \, U^\prime) . \label{grid_mu}
 \end{gather}
 Here~$\mu_{\eps,y}$ is the measure defined by~\eqref{mu_y,eps}.
\end{lemma}
\begin{proof}
 We take a grid~$\GG^\eps := \GG(a, \, h(\eps))$ of the form~\eqref{grid}.
 We claim that it is possible to choose~$a\in (0, \, h(\eps))^{n+k}$ in such a way
 that~\eqref{grid_ktilde}--\eqref{grid_mu} are satisfied. 
 For \eqref{grid_ktilde}--\eqref{grid_k-1}, 
 we can repeat verbatim the arguments in~\cite{ABO2}.
 As for~\eqref{grid_mu}, let us call~$R^\eps_{k-1}(a)$ the $(k-1)$-skeleton 
 of the grid $\GG(a, \, h(\eps))$.
 Thanks to~\eqref{S:mass} in Section~\ref{sect:setting}, $\mu_{\eps,y}$
 is a finite, non-negative Radon measure for a.e.~$y\in B^*$.
 By applying \cite[Lemma~5.2]{FedererFleming}, 
 together with a scaling argument, we obtain
 \[
  h(\eps)^n \fint_{(0, \, h(\eps))^{n+k}}
  \left(\int_{U^\prime} \frac{\d\mu_{\eps,y}(x)}
  {\dist^n(x, \,R^\eps_{k-1}(a))} \right)
  \d\L^{n+k}(a) \lesssim \mu_{\eps, y}(\R^{n+k})
  = \M(\S_y(u_\eps)\mres U^\prime)
 \]
 for a.e~$y\in B^*$. 
 By integrating the previous inequality with respect to~$y$
 and applying~\eqref{S:mass}, we obtain
 \[
  h(\eps)^n\fint_{(0, \, h(\eps))^{n+k}}
  \left(\int_{B^*}\int_{U^\prime} \frac{\d\mu_{\eps,y}(x)} 
  {\dist^n(x, \,R^\eps_{k-1}(a))} \, \d y\right)
  \d\L^{n+k}(a) \lesssim \norm{\nabla u_\eps}_{L^k(U^\prime)}^k
  \lesssim E_\eps(u_\eps, \, U^\prime).
 \]
 Now the lemma follows by an averaging argument, see e.g.~\cite[Lemma~8.4]{ABO2}.
\end{proof}

Throughout the rest of this section, we suppose that~\eqref{h} is satisfied, 
we fix~$\delta\in (0, \, 1)$ and we consider the sequence of
grids~$\GG^\eps$ given by Lemma~\ref{lemma:grid}.
Without loss of generality, 
we will also assume that
\begin{equation} \label{RepsU}
 R^\eps_{n+k}(\overline{U}) \csubset U^\prime
\end{equation}
($R^\eps_{n+k}(\overline{U})$ is the union of the closed 
cubes~$K\in\GG^\eps$ such that~$K\cap\overline{U}\neq\emptyset$).

\begin{lemma} \label{lemma:grid_conv}
 For any~$\alpha \in (0, \, k/(k^2-k+2))$, there holds
  \begin{equation*} 
   \sup_{x\in R^\eps_{k-1}(U)} 
   \dist(u_\eps(x), \, \NN) \leq  
   \frac{C(\delta, \, \alpha) \, \eps^\alpha}{h(\eps)^{(n+k)/2}} 
   \left(E_\eps(u_\eps, \, U^\prime) + 1\right)^{1/2},
  \end{equation*}
  where~$C(\delta, \, \alpha)$ is a positive constant that 
  only depends on~$\NN$, $k$, $f$, $n$, $\delta$ and~$\alpha$.
\end{lemma}
\begin{proof}
 We repeat the arguments of~\cite[Lemma~3.4]{ABO2}.
 Let~$d_\eps:= \dist(u_\eps, \, \NN)$, let~$\lambda\in (0, \, 1/k)$ 
 be a parameter, and let $G\colon\R^+\to\R^+$ be defined by
 $G(t) := t^{2\lambda/(k - k\lambda) + 1}$. Thanks
 to~\eqref{hp:non-degeneracy}, we have $d_\eps^2\lesssim f(u_\eps)$.
 Therefore, by~\eqref{grid_k-1} and~\eqref{RepsU}, we obtain
 \begin{equation} \label{conv0}
  \begin{split}
   h(\eps)^{n+1}\int_{R^\eps_{k-1}(U)} \left(
   \frac{1}{k}\abs{\nabla d_\eps}^k + 
   \eps^{-k}d^2_\eps\right)\d\H^{k-1} 
   &\lesssim h(\eps)^{n+1} E_\eps(u_\eps, \, R^\eps_{k-1}(U)) \\
   &\lesssim \delta^{-1} E_\eps(u_\eps, \, U^\prime).
  \end{split}
 \end{equation}
 The Young inequality and the chain rule imply
 \begin{equation} \label{conv1}
  \begin{split}
   \delta^{-1} E_\eps(u_\eps, \, U^\prime) & \gtrsim h(\eps)^{n+1}
   \int_{R^\eps_{k-1}(U)} \left(
   \frac{1}{k}\abs{\nabla d_\eps}^k + 
   \eps^{-k}d^2_\eps\right)\d\H^{k-1} \\
   &\geq C(\lambda) \, \eps^{-k\lambda} h(\eps)^{n+1} \int_{R^\eps_{k-1}(U)}
   \abs{\nabla d_\eps}^{k - k\lambda} d^{2\lambda}_\eps \, \d\H^{k-1} \\
   & \geq C(\lambda) \, \eps^{-k\lambda} h(\eps)^{n+1} \int_{R^\eps_{k-1}(U)}
   \abs{\nabla (G\circ d_\eps)}^{k - k\lambda} \d\H^{k-1}
  \end{split}
 \end{equation}
 Since we have assumed that~$\lambda < 1/k$, we have
 $k - k\lambda > k-1$ and hence, for any $(k-1)$-cell~$K\subseteq R^\eps_{k-1}(U)$,
 we can bound the oscillation of~$G\circ d_\eps$ on~$K$ by Sobolev embedding:
 \[
  \begin{split}
   \left(\mathrm{osc}\,(G\circ d_\eps, \, K)\right)^{k-k\lambda} 
   &\leq C(\delta, \, \lambda) \, h(\eps)^{1-k\lambda} \int_{R^\eps_{k-1}(U)}
   \abs{\nabla (G\circ d_\eps)}^{k - k\lambda} \d\H^{k-1} \\
   &\stackrel{\eqref{conv1}}{\leq} C(\delta, \, \lambda)
   \, \eps^{k\lambda}h(\eps)^{-n-k\lambda}E_\eps(u_\eps, \, U^\prime).
  \end{split}
 \]
 The inverse~$G^{-1}$ of~$G$ is well-defined and H\"older continuous
 of exponent $(k-k\lambda)/(2\lambda + k - k\lambda)$, so
 \begin{equation} \label{conv2}
  \begin{split}
  \mathrm{osc}\,(d_\eps, \, K) &\leq C(\delta, \, \lambda) \left(
   \eps^{k\lambda} h(\eps)^{-n-k\lambda} E_\eps(u_\eps, \, U^\prime)
   \right)^{1/(2\lambda + k - k\lambda)} \\
   &\leq C(\delta, \, \lambda) \, \eps^{k\lambda/(2\lambda + k - k\lambda)} 
   \, h(\eps)^{-(n+k\lambda)/(2\lambda + k - k\lambda)}
   \, E_\eps(u_\eps, \, U^\prime)^{1/(2\lambda + k - k\lambda)} 
  \end{split}
 \end{equation}
 On the other hand, we can bound the integral average of~$d_\eps$
 on~$K$ thanks to~\eqref{conv0}:
 \begin{equation} \label{conv3}
  \begin{split}
   \fint_K d_\eps \, \d\H^{k-1} 
   \leq \left(\fint_K d^2_\eps \, \d\H^{k-1}\right)^{1/2}
   &\lesssim \delta^{-1/2} \, \eps^{k/2} \, h(\eps)^{-(n+k)/2}
   \, E_\eps(u_\eps, \, U^\prime)^{1/2} .
  \end{split}
 \end{equation}
 Combining~\eqref{conv2} with~\eqref{conv3}, 
 and letting~$\lambda\nearrow 1/k$, the lemma follows.
\end{proof}

\subsubsection{A polyhedral approximation of~\texorpdfstring{$\S_y(u_\eps)$}{S}}

\newcommand{\Sepsy}
{\S_{y}(u_\eps)}

Let~$y\in B^*$ be fixed in such a way that 
$\S_{y}(u_\eps)$ has finite mass for any~$\eps$.
(Thanks to~\eqref{S:mass}, the set of~$y$ such that this property
is not satisfied is negligible, because
the sequence $(u_\eps)$ is assumed to be countable.)
We are going to construct a polyhedral approximation of~$\Sepsy$,
supported on the dual $n$-skeleton~$(R^\eps)^\prime_n$ of the grid.

Thanks to Lemma~\ref{lemma:grid_conv}, there exists~$\eps_0>0$
(depending on~$\delta_*$, but not on~$y$) such that
\begin{equation} \label{grid_k-1_conv}
 \dist(u_\varepsilon(x), \, \NN) < \dist(\NN, \, \X) - \delta_* 
 < \dist(\NN, \, \X) - \abs{y}
\end{equation}
for any $x\in R^\eps_{k-1}(U)$ and any~$\eps\in (0, \, \eps_0]$. 
As a consequence, the projection~$\RR(u_\eps - y)$ is well-defined
and smooth on~$R^\eps_{k-1}(U)$ for~$\eps\in (0, \, \eps_0]$.
For any~$K\in\GG^\eps_{k}$, let $\gamma^\eps(K)\in\GN$ be the
homotopy class of~$\RR(u_\eps - y)$ on~$\partial K$. 
The quantity~$\gamma^\eps(K)$ does not depend on 
the choice of~$y\in B^*$, because 
$\RR(u_\eps - y)_{|\partial K}$ and~$\RR(u_\eps)_{|\partial K}$
are homotopic to each other, due to~\eqref{grid_k-1_conv};
a homotopy is defined by
$(x, \, t)\in\partial K\times[0, \,1]\mapsto \RR(u_\eps(x) - ty)$.
We define the polyhedral chain
\begin{equation} \label{j_eps}
 T^\eps := \sum_{K\in\GG^\eps_{k}, \ K\cap U\neq\emptyset}
 \gamma^\eps(K) \, \llbracket K^\prime \rrbracket
 \in \M_n(\overline{U^\prime}; \, \GN),
\end{equation}
where~$K^\prime\in(\GG^\eps)^\prime_n$ is the dual cell to~$K$.
The chain~$T^\eps$ depends on the choice of the grid, but not on~$y$.

\begin{lemma} \label{lemma:projection}
 For any~$\eps\in (0, \, \eps_0]$ and any~$y\in B^*$
 such that~$\Sepsy$ has finite mass, there holds
 \begin{equation} \label{proj:flat}
  \F_U\left(\Sepsy\mres U - T^\eps\right) \lesssim 
   h(\eps)^{n+1} \int_{U^\prime} \frac{\d\mu_{\eps,y}(x)}
   {\dist^n(x, \, R^\eps_{k-1})}. 
 \end{equation}
 Moreover, $\partial T^\eps\mres U = 0$.
\end{lemma}
\begin{proof}
 Essentially, this lemma is a particular instance of the 
 Deformation Theorem for flat chains~\cite[Theorem~7.3]{Fleming}
 (see also~\cite[Lemma~3.8]{ABO2} for a statement
 which is specifically tailored for application to Ginzburg-Landau functionals).
 Nevertheless, we provide details for the convenience of the reader.
 
 Let~$\eps\in (0, \, \eps_0]$ be fixed. 
 By~\cite[Lemma~3.8.(i)]{ABO2} there exists a locally Lipschitz retraction
 $\zeta^\eps\colon\R^{n+k}\setminus R^\eps_{k-1}\to (R^\eps)^\prime_{n}$,
 which maps each cube of~$\GG^\eps$ into itself and satisfies
 \begin{equation} \label{sigma}
  |\nabla\zeta^\eps(x)| \lesssim h(\eps) \dist(x, \, R^\eps_{k-1})^{-1}
  \qquad \textrm{for a.e. } x\in\R^{n+k}\setminus R^\eps_{k-1}.
 \end{equation}
 By~\eqref{grid_k-1_conv}, we have~$u_\eps(x) - y\notin\X$ 
 for any~$x\in R^\eps_{k-1}(U)$. By construction (see~\cite[Section~3]{CO1}),
 this implies $\spt(\Sepsy)\cap R^\eps_{k-1}(U) = \emptyset$, so the
 push-for\-ward~$\zeta^\eps_{*}(\Sepsy)\mres U$ is well-defined.
 Let~$\tau^\eps\colon[0, \, 1]\times(\R^{n+k}\setminus R^\eps_{k-1})\to\R^{n+k}$ 
 be given by
 \[
  \tau^\eps(t, \, x) := (1 - t) x + t\zeta^\eps(x)
 \]
 and let~$I$ be the $1$-chain, with integer multiplicity, carried
 by the interval~$[0, \, 1]$ with positive orientation. We remark that
 \begin{equation} \label{projection1}
  \tau^\eps_{*}(I\times\partial\Sepsy)\mres U = 0.
 \end{equation}
 Indeed, since~$\zeta^\eps$ maps each cell~$K$ of~$\GG^\eps$
 into itself, we have $(\tau^\eps)^{-1}(U)\subseteq
 [0, \, 1]\times R^\eps_{n+k}(\overline{U})
 \csubset [0, \, 1]\times U^\prime$ by~\eqref{RepsU}. This implies
 \[
  \begin{split}
   \tau^\eps_{*}(I\times\partial \Sepsy)\mres U &=
   \tau^\eps_{*}\left((I\times\partial \Sepsy)
       \mres ([0, \, 1]\times R^\eps_{n+k}(\overline{U}))\right)\mres U \\
   &= \tau^\eps_{*}\left(I\times\left(\partial \Sepsy
       \mres R^\eps_{n+k}(\overline{U})\right)\right)\mres U = 0
  \end{split}
 \]
 because $\partial \Sepsy\mres U^\prime = 0$
 \cite[Theorem~3.1]{CO1}.
 This proves~\eqref{projection1}.
 As a consequence, by applying the homotopy formula 
 (see e.g.~\cite[Eq.~(6.3) p.~172]{Fleming}),
 we deduce that
 \begin{equation} \label{projection2}
  \left(\zeta^\eps_{*}(\Sepsy) - \Sepsy\right)\mres U 
  = \partial\tau^\eps_{*}(I\times \Sepsy)\mres U.
 \end{equation}
 From~\cite[Eq.~(6.5) p.~172]{Fleming} 
 and~\eqref{sigma}, we obtain
 \begin{equation*} 
  \M\left(\tau^\eps_{*}(I\times \Sepsy)\right) \lesssim h(\eps)^n 
  \int_{U^\prime} \frac{\abs{\zeta^\eps(x) - x}}
          {\dist^n(x, \, R^\eps_{k-1})} \d\mu_{\eps, y}(x)
  \lesssim h(\eps)^{n+1} \int_{U^\prime} \frac{\d\mu_{\eps,y}(x)}{\dist^n(x, \, R^\eps_{k-1})}.
 \end{equation*}
 Then, by the properties of the relative flat norm 
 (see e.g.~\cite[Lemma~2]{CO1}) and~\eqref{projection2}, we deduce
 \begin{equation} \label{projection6}
  \begin{split}
   \F_U(\zeta^\eps_{*}(\Sepsy) - \Sepsy)
   \lesssim h(\eps)^{n+1} \int_{U^\prime} \frac{\d\mu_{\eps,y}(x)}
   {\dist^n(x, \, R^\eps_{k-1})}.
  \end{split}
 \end{equation}
 To conclude the proof of~\eqref{proj:flat}, it suffices to show that 
 $\zeta^\eps_{*}(\Sepsy)$ agrees with~$T^\eps$
 inside~$U$. 
 By~\cite[Lemma~7.2]{Fleming},
 $\zeta^\eps_{*}(\Sepsy)\mres U$ is a $n$-polyhedral chain of the grid~$(\GG^\eps)^\prime$;
 in particular, its multiplicity is constant on every $n$-cell 
 of~$(\GG^\eps)^\prime$. We want to compute such multiplicities.
 Let us take $K\in\GG^\eps_k$ and its dual cell~$K^\prime\in(\GG^\eps)^\prime_n$,
 and let~$x$ be the unique element of~$K\cap K^\prime$.
 By construction of~$\zeta^\eps$ (see \cite[Lemma~3.8 and Figure~3.2]{ABO2}),
 we have~$(\zeta^\eps)^{-1}(x) = K\setminus\partial K$.
 By Thom's parametric transversality theorem, we can assume
 with no loss of generality that~$K$ intersects transversally the support of~$\Sepsy$.
 Then, by definition of push-forward, we have
 \[
  \begin{split}
   \textrm{multiplicity of } \zeta^\eps_{*}(\Sepsy)
   \textrm{ at } x 
   = \I(\Sepsy, \, \llbracket (\zeta^\eps)^{-1}(x)\rrbracket)
   = \I(\Sepsy, \, \llbracket K \rrbracket)
   \stackrel{\eqref{S:intersection}}{=} \gamma^\eps(K) 
  \end{split}
 \] 
 and hence
 \begin{equation} \label{projection3}
  (\zeta^\eps_{*}(\Sepsy) - T^\eps)\mres U = 0.
 \end{equation}
 Now,~\eqref{proj:flat} follows from~\eqref{projection6} 
 and~\eqref{projection3}. Moreover, \eqref{projection3} implies
 \[
  \partial T^\eps\mres U 
  = \zeta^\eps_{*}(\partial \Sepsy)\mres U = 0,
 \]
 because $\Sepsy$ has no boundary inside~$U^\prime$
 \cite[Theorem~3.1, (S\textsubscript{3})]{CO1}.
\end{proof}

To bound the mass of~$T^\eps$, we will use the following result.

\begin{lemma} \label{lemma:lowerbounds}
 There exist positive numbers~$\delta_1=\delta_1(\NN, \, f)$,
 $C_0 = C_0(\NN, \, f)$ and, for~$r>0$,
 $\eps_r=\eps_r(r, \, \NN, \, f)$,
 $C_r = C_r(r, \, \NN, \, f)$ such that the following statement holds. 
 Let~$Q^k_h := [-h/2, \, h/2]^k$ be a cube of edge length~$h>0$.
 Suppose that~$u\in W^{1,k}(Q^k_h, \, \R^m)$ satisfies
 \begin{equation*} 
  u_{|\partial Q^k_h}\in W^{1,k}(\partial Q^k_h, \, \R^m)
  \qquad \textrm{and} \qquad 
  \dist(u(x), \, \NN)\leq\delta_1 \quad \textrm{for a.e. }
  x\in\partial\Omega.
 \end{equation*}
 Let~$\gamma\in\GN$ be the homotopy class of~$u$ on~$\partial Q^k_h$.
 Let~$0 < \eps< h^{k/2}/2$ be such that
 \[
  \frac{\eps}{h^{k/2}}\abs{\log\frac{\eps}{h^{k/2}}}|\gamma|_*\leq\eps_r.
 \]
 Then,
 \[
  E_\eps(u, \, Q_h) + C_0\, h \, r E_\eps(u, \, \partial Q_{h}) \geq
  \abs{\gamma}_*\abs{\log\frac{\eps}{h^{k/2}}} -
  C_r\abs{\gamma}_*(1 + \log\abs{\gamma}_*) .
 \]
\end{lemma}

The proof of Lemma~\ref{lemma:lowerbounds} will be given in Appendix~\ref{sect:jerrard}.

\begin{lemma} \label{lemma:mass_T}
 For any~$r$, $\delta$ and for sufficiently small~$\eps$, there holds
 \begin{equation} \label{mass_T}
  \begin{split}
   \left(1 - c_{r,\delta}(\eps)\right)\M(T^\eps\mres U) 
    &\lesssim \delta^{-1}\left(1 + r\right) 
    \frac{E_\eps(u_\eps, \, U^\prime)}{\abs{\log\eps}},
  \end{split}
 \end{equation}
 where~$c_{r,\delta}(\eps)>0$ is such that 
 $c_{r,\delta}(\eps)\to 0$ as~$\eps\to 0$. Moreover, if~$L$ is the $k$-plane
 spanned by~$\{\e_{n+1}, \, \ldots, \, \e_{n+k}\}$, then there holds
 \begin{equation} \label{mass_Tproj}      
  \begin{split}
   \left(1 - c_{r,\delta}(\eps)\right)\M(\pi_{L,*}(T^\eps\mres U))
    &\leq \left(1 + \delta + C r \, \delta^{-1}\right)
    \frac{E_\eps(u_\eps, \, U^\prime)}{\abs{\log\eps}}.
  \end{split}
 \end{equation}
\end{lemma}
\begin{proof}
 We first remark that
 \begin{equation} \label{mass-1}
  \M(T^\eps\mres U) \leq h(\eps)^n 
  \sum_{K\in\GG^\eps_k, \ K\cap U\neq\emptyset}
  |\gamma^\eps(K)|_*.
 \end{equation}
 Let~$K\in\GG^\eps_k$ be a $k$-cell such that~$K\cap U\neq\emptyset$.
 We claim that
 \begin{equation} \label{mass0}
  |\gamma^\eps(K)|_* \lesssim \delta^{-1} \,
  h(\eps)^{-n} \, E_\eps(u_\eps, \, U^\prime) 
 \end{equation}
 Indeed, thanks to \eqref{S:intersection}
 and the definition of~$\I$ (see e.g.~\cite[Section~2.1]{CO1}), we have
 \[
  |\gamma^\eps(K)|_* = |\I(\S_y(u_\eps), \, \llbracket K\rrbracket)|_*
  \leq \M(\S_{y}({u_\eps}_{|K}))
 \]
 for any $y\in B^*$.
 By averaging both sides with respect to~$y\in B^*$,
 and by applying~\eqref{S:mass} from Proposition~\ref{prop:S}, we obtain
 \[
  |\gamma^\eps(K)|_* \leq \fint_{B^*} \M(\S_y({u_\eps}_{|K})) \, \d y
  \lesssim \norm{\nabla u_\eps}_{L^k(K)}^k
  \lesssim E_\eps(u_\eps, \, R^\eps_k\cap U^\prime).
 \]
 We can bound the right-hand side from above with the help of 
 \eqref{grid_k}, so the claim~\eqref{mass0} follows.
 
 From~\eqref{H}, \eqref{h} and~\eqref{mass0}, we deduce
 \[
  \sup_{K\in \GG^\eps_k, \ K\cap U \neq\emptyset} \ \frac{\eps}{h(\eps)^{k/2}} 
  \abs{\log\left(\frac{\eps}{h(\eps)^{k/2}}\right)}
  |\gamma^\eps(K)|_* \to 0 \qquad \textrm{as } \eps\to 0
 \]
 and this fact, together with~\eqref{grid_k-1_conv}, 
 shows that the assumptions of Lemma~\ref{lemma:lowerbounds} are satisfied
 for~$\eps$ small enough. By applying Lemma~\ref{lemma:lowerbounds},
 \eqref{h} and~\eqref{mass0}, we obtain the following bound:
 \begin{equation*} 
  \begin{split}
   |\gamma^\eps(K)|_* \abs{\log\eps}
   \left(1 + \mathrm{o}_{\eps\to 0}(1)\right) \leq
   E_\eps(u_\eps, \, K) + Cr \, h(\eps)  E_\eps(u_\eps, \, \partial K).
  \end{split}
 \end{equation*}
 We multiply both sides by~$h(\eps)^n\abs{\log\eps}^{-1}$
 and sum over~$K$. Thanks to~\eqref{mass-1}, we obtain
 \begin{equation*} 
  \begin{split}
   \left(1 + \mathrm{o}_{\eps\to 0}(1)\right)\M(T^\eps\mres U) 
    &\leq h(\eps)^n \, \frac{E_\eps(u_\eps, 
       \, R^\eps_{k}\cap U^\prime)}{\abs{\log\eps}} \\
    &\qquad\qquad
    + C r \, h(\eps)^{n+1} \, \frac{E_\eps(u_\eps, 
       \, R^\eps_{k-1}\cap U^\prime)}{\abs{\log\eps}}.
  \end{split}
 \end{equation*}
 The right-hand side can now be bounded from above with 
 the help of Lemma~\ref{lemma:grid}, so~\eqref{mass_T} follows.
 The proof of~\eqref{mass_Tproj} is analougous; in this case,
 we sum over the cells~$K$ that are parallel to the $k$-plane spanned 
 by~$\{\e_{n+1}, \, \ldots, \, \e_{n+k}\}$ and use~\eqref{grid_ktilde}.
\end{proof}

\subsubsection{Proof of Proposition~\ref{prop:liminf-local}}

By combining the results in the previous section,
we prove the following lemma,
which is analougous to~\cite[Proposition~3.1]{ABO2}.
For any $n$-plane $L\subseteq\R^{n+k}$, we denote 
by $\pi_{L}\colon\R^{n+k}\to L$ the orthogonal projection onto~$L$.

\begin{lemma} \label{lemma:liminf-L}
 Let~$U\csubset U^\prime$ be bounded domains in~$\R^{n+k}$.
 Let~$(u_\eps)_\eps$ be a countable sequence of smooth, 
 bounded maps that satisfy~\eqref{H}.
 Let~$L\subseteq\R^{n+k}$ be a $n$-plane. Then, there exist a (non-relabelled) subsequence
 and a finite-mass chain $S\in\M_n(\overline{U}^\prime; \, \GN)$ such that
 \begin{gather}
  \int_{B^*} \F_U(\S_y(u_\eps) - S) \, \d y \to 0 
   \qquad \textrm{as } \eps\to 0 \label{liminf:flatL} \\
  \M(\pi_{L,*}(S\mres U)) \leq \liminf_{\eps\to 0} 
   \frac{E_\eps(u_\eps, \, U^\prime)}{\abs{\log\eps}}. \label{liminf:massL}
 \end{gather}
\end{lemma}
\begin{proof}
 Up to rotations we can assume without loss of 
 generality that~$L$ is the $k$-plane spanned by~$\{\e_{n+1}, \, \ldots, \, \e_{n+k}\}$.
 By Lemma~\ref{lemma:projection} and Lemma~\ref{lemma:mass_T}, 
 we know that $\partial T^\eps\mres U = 0$ and $\M(T^\eps\mres U)$
 is uniformly bounded with respect to~$\eps$. Then, by applying
 compactness results for the flat norm (see e.g.
 \cite[Lemma~5 and~6]{CO1} for a statement in terms of the relative flat norm),
 we find a (non-relabelled) subsequence and a finite-mass chain 
 $S\in\M_n(\overline{U}; \, \GN)$ such that
 \begin{gather}
  \F_U(T^\eps - S) \to 0 \qquad \textrm{as } \eps\to 0 \label{liminf1} \\
  \M(\pi_{L,*}(S\mres U)) \leq (1 + \delta + C\delta^{-1} r)
  \liminf_{\eps\to 0} \frac{E_\eps(u_\eps, \, U^\prime)}{\abs{\log\eps}}.
  \label{liminf2}
 \end{gather}
 The triangle inequality and Lemma~\ref{lemma:projection} imply
 \begin{equation*} \label{liminf3}
  \begin{split}
   \int_{B^*} \F_U(\S_y(u_\eps) - S) \, \d y
   &\leq \int_{B^*} \F_U(\S_y(u_\eps) - T^\eps) \, \d y
   + \L^n(B^*) \, \F_U(T^\eps - S) \\
   &\leq h(\eps)^{n+1} \int_{B^*} \int_{U^\prime} \frac{\d\mu_{\eps,y}(x)}
   {\dist^n(x, \, R^\eps_{k-1})} \d y 
   + \L^n(B^*) \, \F_U(T^\eps - S) \\
   &\stackrel{\eqref{grid_mu}}{\lesssim}
   \delta^{-1} h(\eps) \frac{E_\eps(u_\eps, \, U^\prime)}{\abs{\log\eps}}
   + \L^n(B^*) \, \F_U(T^\eps - S)
  \end{split}
 \end{equation*}
 and the right-hand side tends to zero as~$\eps\to 0$, due
 to~\eqref{H} and~\eqref{liminf1}. Thus, \eqref{liminf:flatL} follows.
 By passing to the limit in~\eqref{liminf2} first as~$r\to 0$,
 then as~$\delta\to 0$, we obtain~\eqref{liminf:massL}. 
\end{proof}

Now, Proposition~\ref{prop:liminf-local} can be deduced 
from Lemma~\ref{lemma:liminf-L} by a localisation argument,
with the help of the following lemma.

\begin{lemma} \label{lemma:mass_chain-nointro}
 Let~$S\in\M_n(\R^{n+k}; \, \GN)$ be a chain of finite mass.
 Then, there holds
 \[
  \M(S) = \sup_{(U_i, \, L_i)_{i\in\N}} \,
  \sum_{i=0}^{+\infty} \M(\pi_{L_i,*}(S\mres U_i)),
 \]
 where the supremum is taken over all sequences of
 pairwise disjoint open sets $U_i$ and $n$-planes $L_i\subseteq\R^{n+k}$.
\end{lemma}

The proof will be given in Appendix~\ref{sect:mass_chains}.
Once Lemma~\ref{lemma:mass_chain-nointro} is proved,
Proposition~\ref{prop:liminf-local} follows by repeating
verbatim the arguments of~\cite[Theorem~1.1.(i)]{ABO2}, so we skip
the proof of Proposition~\ref{prop:liminf-local}.

\subsection{Compactness and lower bounds for the boundary value problem}

The aim of this section is to complete the proof of
Theorem~\ref{th:main}.(i). We will deduce 
Theorem~\ref{th:main}.(i)
from its local counterpart, i.e. Proposition~\ref{prop:liminf-local},
with the help of the extension result, Lemma~\ref{lemma:extension}.

\begin{proof}[Proof of Theorem~\ref{th:main}.(i)]
 Let~$(u_\eps)_\eps\subseteq W^{1,k}_v(\Omega, \, \R^m)$ 
 be such that 
 $\sup_\eps {\abs{\log\eps}^{-1}}{E_\eps(u_\eps)} <+\infty$.
 Let~$\tilde{u}\in (L^\infty\cap W^{1,k})(\R^{n+k}, \,\R^m)$ 
 be such that~$\tilde{u} = v$ on~$\partial\Omega$. 
 Let~$\Omega^\prime$, $\Omega^{\prime\prime}$ be bounded
 domains in~$\R^{n+k}$, such that~$\Omega\csubset\Omega^\prime
 \csubset\Omega^{\prime\prime}$.
 By applying Lemma~\ref{lemma:extension},
 we find $y\in B^*$, a subsequence~$\eps\to 0$ and maps
 $w_{\eps,y}\in (L^\infty\cap W^{1,k})
 (\Omega^{\prime\prime}\setminus\overline{\Omega}, \, \R^m)$
 that agree with~$v$ on~$\partial\Omega$ and satisfy
 \begin{equation}\label{mliminf1}
  \sup_{\eps} \left(
   \|w_{\eps,y}\|_{L^\infty(\Omega^{\prime\prime}\setminus\overline{\Omega})}
   + \frac{E_\eps(w_{\eps,y}, \, \Omega^{\prime\prime}\setminus\overline{\Omega})}
  {\abs{\log\eps}} \right) < + \infty. 
 \end{equation}
 Lemma~\ref{lemma:extension} also implies that the sequence $(w_{\eps, y})$ converges 
 $W^{1,k-1}(\Omega^{\prime\prime}\setminus\overline{\Omega})$-strongly
 to a limit~$w_y$, and that~$\S(w_y) = \S_y(\tilde{u})$.
 Then, the continuity of~$\S$ \cite[Theorem~3.1]{CO1} implies
 \begin{equation}  \label{mliminf2}
  \int_{B^*} \F_{\Omega^{\prime\prime}\setminus\overline{\Omega}} 
   \left(\S_{y^\prime}(w_{\eps,y}) - \S_{y}(\tilde{u})\right) 
   \d y^\prime \to 0 \qquad \textrm{as } \eps\to 0.
 \end{equation}
 We define the map~$\tilde{u}_\eps\in (L^\infty\cap W^{1,k})
 (\Omega^{\prime\prime}, \, \R^m)$ by setting~$\tilde{u}_\eps := u_\eps$
 on~$\Omega$ and~$\tilde{u}_\eps := w_{\eps,y}$ on 
 $\Omega^{\prime\prime}\setminus\overline{\Omega}$.
 Since the operator~$\S$ is local~\cite[Corollary~1]{CO1},
 we have
 \[
  \S_{y^\prime}(w_{\eps,y})\mres(\Omega^{\prime\prime}\setminus\overline{\Omega})
  = \S_{y^\prime}(\tilde{u}_\eps)\mres(\Omega^{\prime\prime}\setminus\overline{\Omega})
 \]
 for a.e.~$y^\prime\in B^*$.
 Therefore, from~\eqref{mliminf2} and~\cite[Lemma~3]{CO1} we obtain
 \begin{equation}  \label{mliminf2bis}
  \int_{B^*} \F_{\Omega^{\prime}\setminus\overline{\Omega}} 
   \left(\S_{y^\prime}(\tilde{u}_\eps) - \S_{y}(\tilde{u})\right) 
   \d y^\prime \to 0 \qquad \textrm{as } \eps\to 0.
 \end{equation}
 We are now in the position to apply our local result,
 Proposition~\ref{prop:liminf-local}, to the sequence~$\tilde{u}_\eps$
 and the open sets~$\Omega^\prime\csubset\Omega^{\prime\prime}$.
 As a result, we obtain a finite-mass chain~$\tilde{S}$
 such that, up to subsequences,
 \begin{equation}\label{mliminf3}
  \int_{B^*} \F_{\Omega^\prime}(\S_{y^\prime}(\tilde{u}_\eps)
    - \tilde{S}) \, \d y^\prime \to 0 \qquad
    \textrm{as }\eps\to 0. 
 \end{equation}
 By~\cite[Lemma~3]{CO1}, 
 \begin{equation*}\label{mliminf3bis}
  \int_{B^*} \F_{\Omega^\prime\setminus\overline{\Omega}}
    (\S_{y^\prime}(\tilde{u}_\eps)
    - \tilde{S}) \, \d y^\prime \to 0 \qquad
    \textrm{as }\eps\to 0. 
 \end{equation*}
 This condition, combined with~\eqref{mliminf2bis}, implies that
 $\S_{y}(\tilde{u})\mres(\Omega^{\prime}\setminus\overline{\Omega}) 
 = \tilde{S} \mres(\Omega^{\prime}\setminus\overline{\Omega})$
 and hence, the chain
 \[
  S := \tilde{S} - \S_{y}(\tilde{u})\mres
  (\Omega^{\prime}\setminus\overline{\Omega}) 
  = \tilde{S}\mres\overline{\Omega}
 \]
 is supported in~$\overline{\Omega}$. At the same time,
 we have $\S_{y^\prime}(u_\eps) = 
 \S_{y^\prime}(\tilde{u}_\eps)\mres\overline{\Omega}$
 for a.e.~$y^\prime\in B^*$. For chains supported in a compact subset of~$\Omega^\prime$,
 the relative flat norm~$\F_{\Omega^\prime}$ is equivalent to~$\F$
 (see e.g. \cite[Remark~2.2]{CO1})
 and hence, \eqref{mliminf3} implies
 \[
  \int_{B^*} \F\left(\S_{y^\prime}(u_\eps)
    - S\right) \d y^\prime \to 0 \qquad
    \textrm{as }\eps\to 0.
 \]
 By~\eqref{S:cobord}, $\S_{y^\prime}(u_\eps)\in\mathscr{C}(\Omega, \, v)$ for 
 any~$\eps$ and a.e.~$y^\prime\in B^*$. 
 The set~$\mathscr{C}(\Omega, \, v)$ is closed with respect to the
 $\F$-norm (this follows from the isoperimetric inequality, see e.g.
 \cite[Statement~(7.6)]{Fleming}). Therefore,
 $S\in\mathscr{C}(\Omega, \, v)$.
 
 It only remain to prove the upper bound on the mass of~$S$.
 Let~$A\subseteq\R^{n+k}$ be an open set. We extract a (non-relabelled) subsequence,
 in such a way that $\liminf_{\eps\to 0}\abs{\log\eps}^{-1} E_\eps(u_\eps, A\cap\Omega)$
 is achieved as a limit. For any integer~$j\geq 1$, let
 $A_j :=\{x\in A\colon \dist(x, \, \partial A)\geq 1/j\}$.
 By applying Proposition~\ref{prop:liminf-local} and a diagonal argument,
 we find a subsequence such that
 \[
  \begin{split}
  \M(S\mres (A_j\cap\Omega^\prime)) 
  \leq \limsup_{\eps\to 0} \frac{E_\eps(\tilde{u}_\eps, \, A\cap\Omega^{\prime\prime})}{\abs{\log\eps}}
  \qquad \textrm{for any } j\geq 1.
  \end{split}
 \]
 By construction, $S$ is supported in~$\overline{\Omega}$,
 so~$S\mres (A_j\cap\Omega^\prime) = S\mres A_j$. Then, by applying 
 Lemma~\ref{lemma:extension}, we obtain
 \[
  \begin{split}
  \M(S\mres A_j) 
  \leq \lim_{\eps\to 0} \frac{E_\eps(u_\eps, \, A\cap\Omega)}{\abs{\log\eps}} 
  + C \int_{\Omega^{\prime\prime}\setminus\Omega} \abs{\nabla\tilde{u}}^{k} 
  \qquad \textrm{for any } j\geq 1,
  \end{split}
 \]
 for some constant~$C$ that does not depend  
 on~$\eps$, $j$, $\Omega^{\prime\prime}$.
 Letting~$j\to+\infty$, $\Omega^{\prime\prime}\searrow\Omega$,
 we conclude that
 \[
  \M(S\mres A) \leq \lim_{\eps\to 0} \frac{E_\eps(u_\eps, \, A\cap\Omega)}{\abs{\log\eps}}
 \]
 and the proof is complete.
\end{proof}

Statement~(i) in Proposition~\ref{prop:main-nobd} also
follows by Proposition~\ref{prop:liminf-local}, in a similar way.

\section{Proof of Theorem~\ref{th:mainmin}}
\label{sect:mainmin}

Let~$u_{\eps,\min}$ be a minimiser of the functional~$E_\eps$
subject to the boundary condition~$u = v$ on~$\partial\Omega$, and let
\begin{equation*}
 \mu_{\eps,\min} := \left( \frac{1}{k}|\nabla u_{\eps,\min}|^k
 + \frac{1}{\eps^k} f(u_{\eps,\min})\right) \frac{\d x\mres\Omega}{\abs{\log\eps}} .
\end{equation*}
We have~$\sup_\eps \mu_{\eps,\min}(\R^{n+k}) < +\infty$
by Remark~\ref{rk:extension} and hence,
up to a subsequence, $\mu_{\eps,\min}$ converges
weakly$^*$ to a limit~$\mu_{\min}$,
in the sense of measures on~$\R^{n+k}$. 
By applying Theorem~\ref{th:main}.(i),
we find a chain~$S_{\min}\in\mathscr{C}(\Omega, \, v)$ such that
\begin{equation} \label{cor1}
 \M(S_{\min}\mres A) \leq \liminf_{\eps\to 0} \mu_{\eps, \min}(A)
 \qquad \textrm{for any open set } A\subseteq\R^{n+k}.
\end{equation}
Theorem~\ref{th:main}.(ii) implies that~$S_{\min}$ is
mass-minimising in~$\mathscr{C}(\Omega, \, v)$. Moreover, by
the properties of weak$^*$ convergence, from~\eqref{cor1} we obtain
\begin{equation} \label{cor2}
 \M(S_{\min}\mres A) \leq \mu_{\min}(A)
 \qquad \textrm{for any open set } A\subseteq\R^{n+k} \textrm{ such that }
 \mu_{\min}(\partial A) = 0.
\end{equation}
Let~$E\subseteq\R^{n+k}$ be a Borel set, let~$U\subseteq\R^{n+k}$
be an open set and let~$K\subseteq\R^{n+k}$ be a compact set
such that $K\subseteq E \subseteq U$. 
For any~$t\in (0, \, \dist(K, \, \partial U))$,
let~$U_t := \{x\in U\colon \dist(x, \, \partial U)> t\} \supseteq K$.
Since~$\mu_{\min}$ is a finite measure, we have
$\mu_{\min}(\partial U_{t}) = 0$ for all but countably 
many~$t\in (0, \, \dist(K, \, \partial U))$.
Therefore, there holds
\[
 \M(S_{\min}\mres K) \leq \M(S_{\min}\mres U_t)
 \stackrel{\eqref{cor2}}{\leq}
 \mu_{\min}(U_t) \leq \mu_{\min}(U) \qquad \textrm{for a.e. } t.
\]
Letting~$U\searrow K$, $K\nearrow E$, we conclude that
$\M(S_{\min}\mres E)\leq \mu_{\min}(E)$.
(The measure~$\M(S_{\min}\mres\cdot)$ is Radon, because by construction,
it is the weak$^*$ limit of a sequence of Radon measures, associated
with polyhedral approximations of~$S_{\min}$; see \cite[Section~4]{Fleming}.)
As a consequence, $\mu_{\min} - \M(S_{\min}\mres\cdot)$ is a non-negative measure.
However, Theorem~\ref{th:main}.(ii) implies that 
$\mu_{\min}(\R^{n+k}) = \lim_{\eps\to 0} \mu_{\eps,\min}(\R^{n+k})\leq \M(S_{\min})$, 
so $\mu_{\min} = \M(S_{\min}\mres\cdot)$.

\paragraph*{Acknowledgements.}
{\BBB The authors are grateful to the referees for their careful reading of the
manuscript and their insightful comments.
The authors were partially supported by GNAMPA-INdAM.}

\begin{appendix}

\section{The norm on~\texorpdfstring{$\GN$}{the coefficient group}: Proof of Proposition~\ref{prop:group_norm}}
\label{sect:norm}

The aim of this section is to prove Proposition~\ref{prop:group_norm}.
In Section~\ref{sect:setting}, we have defined 
\begin{equation} \label{I_min}
 E_{\min}(\sigma) := \inf\left\{
 \frac{1}{k}\int_{\SS^{k-1}}\abs{\nablaT v}^k
 \colon v\in W^{1,k}(\SS^{k-1}, \, \NN)\cap\sigma \right\} 
\end{equation}
for any~$\sigma\in\GN$, with~$\nablaT$ the tangential
gradient on~$\SS^{k-1}$ (i.e.~the restriction of~$\nabla$
to the tangent plane to~$\SS^{k-1}$). The compact Sobolev
embedding~$W^{1,k}(\SS^{k-1}, \, \NN)\hookrightarrow C(\SS^{k-1}, \, \NN)$
implies that~$W^{1,k}(\SS^{k-1}, \, \NN)\cap\sigma$ is 
sequentially $W^{1,k}$-weakly closed, so the infimum at the 
right-hand side is achieved. We must have
\begin{equation} \label{discrete_I}
 \inf_{\sigma\in\GN\setminus\{0\}} E_{\min}(\sigma) > 0,
\end{equation}
for otherwise there would exist a sequence of non-null-homotopic maps
$v_j\in W^{1,k}(\SS^{k-1}, \, \NN)$ that converge $W^{1,k}$-strongly,
and hence uniformly, to a constant. Moreover, there holds
\begin{equation} \label{inversion}
 E_{\min}(\sigma) = E_{\min}(-\sigma) \qquad 
 \textrm{for any } \sigma\in\GN.
\end{equation}
Indeed, for any~$v\in W^{1,k}(\SS^{k-1}, \, \NN)\cap\sigma$ and any~$x\in\SS^{k-1}$,
define $\bar{v}(x) := v(-x_1, \, x_2, \, \ldots, \, x_k)$. The
map that sends~$v\mapsto\bar{v}$
is a bijection $W^{1,k}(\SS^{k-1}, \, \NN)\cap\sigma\to W^{1,k}(\SS^{k-1}, \, \NN)\cap(-\sigma)$
that preserves the $L^k$-norm of the gradient, and hence~\eqref{inversion} follows.

Our candidate norm~$|\cdot|_*$ on~$\GN$, 
which was also introduced in Section~\ref{sect:setting},
is defined for any~$\sigma\in\GN$ by
\begin{equation} \label{group_norm}
 |\sigma|_* := 
   \inf\left\{\sum_{i=1}^q E_{\min}(\sigma_i)\colon q\in\N, \ (\sigma_i)_{i=1}^q\in\GN^q,
   \ \sum_{i=1}^q \sigma_i = \sigma\right\} \!.
\end{equation}

\begin{prop}
 The function $|\cdot|_*$ is a norm on~$\GN$ that satisfies
 \begin{equation} \label{discrete_norm}
  \inf_{\sigma\in\GN\setminus\{0\}} |\sigma|_* > 0
 \end{equation}
 and
 \begin{equation} \label{I_norm}
   \abs{\sigma}_* \leq E_{\min}(\sigma) \qquad 
  \textrm{for any } \sigma\in\GN.
 \end{equation}
 The infimum in~\eqref{discrete_norm} is achieved, for any~$\sigma\in\GN$.
 Moreover, the set
 \begin{equation} \label{generators}
  \Sg := \left\{\sigma\in\GN\colon 
  \abs{\sigma}_* = E_{\min}(\sigma) \right\}
 \end{equation}
 is finite, and for any~$\sigma\in\GN$ there exists a decomposition~$\sigma = \sum_{i=1}^q\sigma_i$
 such that $\abs{\sigma}_* = \sum_{i=1}^q\abs{\sigma_i}_*$ and~$\sigma_i\in\Sg$ for any~$i$.
\end{prop}
\begin{proof}
 The function~$|\cdot|_*$ is certainly non-negative, and its definition~\eqref{group_norm}
 immediately implies the triangle inequality, $|\sigma_1+\sigma_2|_*\leq |\sigma_1|_* + |\sigma_2|_*$.
 The property~$|\sigma|_* = |-\sigma|_*$ follows by~\eqref{inversion}, 
 while~\eqref{discrete_I} yields~\eqref{discrete_norm} (in particular,
 $|\sigma|_* = 0$ only if~$\sigma=0$). The property~\eqref{I_norm} is
 immediate from the definition of~$|\cdot|_*$.
 
 We check now that the set~$\Sg$ is finite.
 Under the assumption~\eqref{hp:N}, Hurewicz theorem 
 (see e.g.~\cite[Theorem~4.37 p.~371]{Hatcher}) implies that
 $\GN$ is isomorphic to the homology group $H_{k-1}(\NN)$.
 The latter is Abelian and finitely generated, because the manifold~$\NN$ is compact
 and hence, homotopically equivalent to a finite cell complex. Therefore, we have 
 \[
  \GN\simeq H_{k-1}(\NN) \simeq \Z^{p} \oplus T,
 \]
 where~$p\geq 0$ is an integer and~$T$ is a finite group. Let~$(g_i)_{i=1}^p$
 be a basis for the torsion-free part of~$H_{k-1}(\NN)$ (i.e., the quotient $H_{k-1}(\NN)/T\simeq\Z^p$).
 By de Rham theorem, there exist closed, smooth 
 $(k-1)$-forms~$\omega_1, \, \ldots, \,  \omega_p$ that satisfy
 \[
  \int_{c_i}  \omega_j = \delta_{ij} \qquad \textrm{for any } i, \, j,
 \]
 where~$c_i$ is a smooth $(k-1)$-cycle in the homology class~$g_i$.
 Let~$\sigma\in\GN$. By abusing of notation, and identifying $g_i$
 with its image under the Hurewicz isomorphism, we can write uniquely
 \[
  \sigma = \sum_{i=1}^p d_ig_i + \sigma_T,
 \]
 where~$d_i \in\Z$ and~$\sigma_T\in T$. Then, for
 any~$v\in W^{1,k}(\SS^{k-1}, \, \NN)\cap\sigma$, we have
 \[
  \begin{split}
   |d_i| = \abs{\int_{\SS^{k-1}} v^* \omega_i} \leq 
   \|\omega_i\|_{L^\infty(\NN, \, \Lambda^{k-1}T^*\NN)} 
   \int_{\SS^{k-1}} \abs{\nablaT v}^{k-1}
   \leq C_{k, \NN}
   \left(\frac{1}{k} \int_{\SS^{k-1}} 
   \abs{\nablaT v}^k\right)^{\frac{k-1}{k}}
  \end{split}
 \]
 where~$C_{k, \, \NN}>0$ is a constant depending only on~$k$ and the $\omega_i$'s.
 This implies
 \begin{equation} \label{I_superquadratic}
  E_{\min}(\sigma) \geq 
  C_{k, \NN}^{\prime} \left(\sum_{i=1}^p |d_i|\right)^{\frac{k}{k-1}}
 \end{equation}
 for a different constant~$C_{k, \NN}^\prime$.
 On the other hand, the definition of~$|\cdot|_*$ immediately gives the upper bound
 \begin{equation} \label{norm_sublinear}
  |\sigma|_* \leq \left(\max_{i=1,\ldots,p} E_{\min}(g_i)\right) 
  \sum_{i=1}^p |d_i| + \max_{\sigma_T\in T} E_{\min}(\sigma_T).
 \end{equation}
 If~$\sigma\in\Sg$ then, by comparing~\eqref{I_superquadratic}
 and~\eqref{norm_sublinear}, we obtain
 $\sum_i |d_i|\leq M$ for some constant~$M>0$ 
 depending only on~$k$, $\NN$. Therefore,
 $\Sg$ is a finite set.
 
 For any~$\sigma\in\GN$ there exists a finite 
 decomposition~$\sigma = \sum_{i=1}^q\sigma_i$ which achieves the infimum in
 the right-hand side of~\eqref{group_norm}. 
 Indeed, it suffices to minimise among the decompositions with 
 $q \leq(\inf_{g\in\GN\setminus\{0\}} E_{\min}(g))^{-1} E_{\min}(\sigma)$
 and~$E_{\min}(\sigma_i) \leq E_{\min}(\sigma)$ for any~$i$, 
 and there are only finitely many such decompositions 
 because of~\eqref{discrete_I}, \eqref{I_superquadratic}.
 Let~$\sigma = \sum_{i=1}^q\sigma_i$ be a decomposition that achieves the minimum
 in~\eqref{group_norm}. Then, the triangle inequality implies
 \[
  \sum_{i=1}^q E_{\min}(\sigma_i) = 
  \abs{\sigma}_* \leq 
  \sum_{i=1}^q \abs{\sigma_i}_*
 \]
 and, since $\abs{\sigma_i}_*\leq E_{\min}(\sigma_i)$ for any~$i$, 
 we must have~$\abs{\sigma_i}_* = E_{\min}(\sigma_i)$, i.e. $\sigma_i\in\Sg$,
 for any~$i$. 
\end{proof}

\begin{example} \label{example:sphere}
 Let~$k=2$, $\NN = \SS^{1}$. Then, $\GN\simeq\Z$ and~$E_{\min}(d) = \pi d^2$
 for any~$d\in\Z$, since the infimum in~\eqref{I_min} is achieved by
 a curve that parametrises the unit circle~$|d|$ times, with constant speed 
 and orientation depending on the sign of~$d$.
 Therefore,~$\Sg = \{-1, \, 0, \, 1\}$ and
 $\abs{d}_* = \pi\abs{d}$ for any~$d\in\Z$.
 
 More generally, when~$\NN=\SS^{k-1}$ the constant that appears in 
 the lower bound~\eqref{I_superquadratic} can be computed explicitely,
 and we have
 \begin{equation} \label{superlin}
  E_{\min}(d)\geq \beta_k |d|^{\frac{k}{k-1}} 
  \qquad \textrm{for any } d\in\pi_{k-1}(\SS^{k-1})\simeq\Z,
 \end{equation}
 where~$\beta_k := (k-1)^{k/2}\L^k(B^k_1)$.
 On the other hand, by using the identity as a comparison map for~\eqref{I_min},
 we see that $E_{\min}(1) \leq (k-1)^{k/2}\mathrm{vol}(\SS^{k-1})/k = \beta_k$,
 hence $E_{\min}(1) = E_{\min}(-1) = \beta_k$. It follows that
 \[
  E_{\min}(d) \stackrel{\eqref{superlin}}{>}
  \beta_k\abs{d} \geq \abs{d}_*
  \qquad \textrm{if } \abs{d}>1.
 \]
 Therefore, also in case~$\NN=\SS^{k-1}$ we have~$\Sg = \{-1, \, 0, \, 1\}$.
 By Proposition~\ref{prop:group_norm}, we conclude that
 $\abs{d}_* = \beta_k\abs{d}$ for 
 any~$d\in\pi_{k-1}(\SS^{k-1})\simeq\Z$. 
\end{example}

\section{The operator~\texorpdfstring{$\overline{\S}$}{S bar}: Proof of Proposition~\ref{prop:S}}
\label{appendix:S}

The aim of this section is to prove Proposition~\ref{prop:S}, 
which we recall here for the convenience of the reader.
We recall that~$\delta^*\in (0, \, \dist(\NN, \, \X))$ is a fixed constant,
$B^* := B^m(0, \, \delta^*)\subseteq\R^m$, and
$\overline{Y} := L^1(B^*, \, \F_{n}(\overline{\Omega}; \, \GN))$ 
is equipped with the norm
\[
 \norm{S}_{\overline{Y}} := \int_{B^*} \F(S_y) \, \d y < +\infty.
\]

\begin{prop} \label{prop:S-app}
 There exists a continuous operator $\overline{\S}\colon W^{1,k}(\Omega, \, \R^*)\to \overline{Y}$ 
 that satisfies the following properties:
 \begin{enumerate}
 [label=\emph{(P\textsubscript{\arabic*})}, ref=P\textsubscript{\arabic*}]
  \setcounter{enumi}{-1}
  \item \label{S:intersection-app} 
  for any smooth~$u$, a.e.~$y\in B^*$ and any~$R\in\F_{k}(\R^{n+k}; \, \Z)$
  such that $\M(R)+\M(\partial R)<+\infty$,
  $\spt(R)\subseteq\Omega$, $\spt(\partial R)
  \subseteq\Omega\setminus\spt\S_y(u)$, there holds
  \[
   \I(\S_y(u), \, R) = 
   \textrm{homotopy class of } \RR\circ(u - y) \textrm{ on } \partial R.
  \]
  
  \item \label{S:Sbar-app}
  For any~$u\in (L^\infty\cap W^{1,k})(\Omega, \, \R^m)$ and a.e~$y\in B^*$,
  $\overline{\S}_y(u) = \S_y(u)$ (more precisely, the chain~$\overline{\S}_y(u)$
  belongs to the equivalence class~$\S_y(u)\in\F_{n}(\Omega; \, \GN)$).
  
  \item \label{S:mass-app} For any~$u\in W^{1,k}(\Omega, \, \R^m)$
  and any Borel subset~$E\subseteq\overline{\Omega}$, there holds
  \[
   \int_{B^*} \M(\overline{\S}_y(u)\mres E) \,\d y \lesssim \int_{E} \abs{\nabla u}^k.
  \]
  
  \item \label{S:cobord-app} Let~$u_0$, $u_1\in W^{1,k}(\Omega, \, \R^m)$
  be such that $u_{0|\partial\Omega} = 
  u_{1|\partial\Omega}\in W^{1-1/k, k}(\partial\Omega, \, \NN)$
  (in the sense of traces). Then, for a.e.~$y_0$, $y_1\in B^*$
  there exists~$R\in\M_{n+1}(\overline{\Omega}; \, \GN)$ such that
  $\overline{\S}_{y_1}(u_1) - \overline{\S}_{y_0}(u_0) = \partial R$. 
 \end{enumerate}
\end{prop}

In the proof, we will use the following
\begin{lemma} \label{lemma:flat_restr}
 Let~$\rho>0$, and let~$\Gamma_\rho := \{x\in\R^{n+k}\setminus\overline{\Omega}\colon
 \dist(x, \, \Omega)<\rho\}$.
 Then, for any finite-mass chain~$T\in\M_n(\R^{n+k}; \, \GN)$, there holds
 \[
  \F(T\mres\overline{\Omega}) \leq (1 + \rho^{-1}) \F(T)
  + \M(T\mres\Gamma_\rho).
 \]
\end{lemma}
\begin{proof}
 For any~$t>0$, let~$\Omega_t := \{x\in\R^{n+k}\colon\dist(x, \, \Omega)<t\}$.
 There holds
 \[
  \int_0^{\rho}\F(T\mres\Omega_t) \, \d t \leq (1 + \rho) \F(T)
 \]
 (see e.g.~\cite[Lemma~4, Eq.~(2.8)]{CO1}). By an averaging argument,
 we can find~$t\in (0, \, \rho)$ such that
 \begin{equation} \label{restr1}
  \F(T\mres\Omega_t) \leq (1 + \rho^{-1}) \F(T).
 \end{equation}
 Now, there holds~$T\mres\Omega_t = T\mres\overline{\Omega} + T\mres\Gamma_t$
 and hence,
 \[
  \begin{split}
   \F(T\mres\overline{\Omega}) \leq \F(T\mres\Omega_t) + \F(T\mres\Gamma_t) 
   \stackrel{\eqref{restr1}}{\leq} (1 + \rho^{-1}) \F(T) + \M(T\mres\Gamma_t), 
  \end{split}
 \]
 so the lemma follows.
\end{proof}

\begin{proof}[Proof of Proposition~\ref{prop:S-app}]
 For the sake of clarity, we split the proof into steps.
 \setcounter{step}{0}
 \begin{step}[Construction of~$\overline{\S}$]
 First, we consider a smooth map~$u\in C^\infty_{\mathrm{c}}(\R^{n+k}, \, \R^m)$
 and the topological singular operator, $\S_y(u)$, as defined
 in~\cite[Section~3.2, Eq.~(3.4)]{CO1}.
 By definition, we can write
 \begin{equation} \label{S}
  \S_y(u) = \sum_{K} \gamma(K) \llbracket (u-y)^{-1}(K)\rrbracket 
  \qquad \textrm{for a.e. } y\in B^*.
 \end{equation}
 Here, the sum is taken over all~$(m-k)$-dimensional 
 polyhedra~$K$ in~$\X$. The coefficient~$\gamma(K)\in\GN$
 is the homotopy class of~$\RR$ restricted to a small 
 $(k-1)$-sphere~$\Sigma$ around~$K$, 
 $\RR_{|\Sigma}\colon\Sigma\simeq\SS^{k-1}\to\NN$.
 For a.e.~$y\in B^*$, the set~$(u-y)^{-1}(K)$
 is a smooth, compact $n$-dimensional manifold (as a consequence of
 Thom's transversality theorem, see e.g.~\cite[Theorem~2.7 p.~79]{Hirsch})
 and $\llbracket (u-y)^{-1}(K)\rrbracket$ denotes the smooth chain
 carried by~$(u-y)^{-1}(K)$, with unit multiplicity and a suitable orientation
 (see~\cite[Section~3.2]{CO1} for more details).
 
 We claim that, for any $u$, $u_0$, $u_1\in C^\infty_{\mathrm{c}}(\R^{n+k}, \, \R^m)$
 and any open set~$U\subseteq\R^{n+k}$, there holds
 \begin{gather} 
  \int_{B^*} \M(\S_y(u)\mres U) \lesssim \int_U \abs{\nabla u}^k \label{mass} \\
  \int_{B^*} \F(\S_y(u_1) - \S_y(u_0)) \lesssim
  \norm{u_1 - u_0}_{L^k(\R^{n+k})} 
  \left(\norm{\nabla u_0}_{L^k(\R^{n+k})}^{k-1} + 
  \norm{\nabla u_1}_{L^k(\R^{n+k})}^{k-1} \right) \! .\label{flat}
 \end{gather}
 These inequalities differ from the corresponding ones in
 \cite[Theorem~3.1]{CO1} because the multiplicative constants in
 front of the right-hand sides do not depend on the
 $L^\infty$-norm of~$u$, $u_0$, $u_1$. We postpone the proof of~\eqref{mass}--\eqref{flat}.
 As a consequence of~\eqref{flat}, by a density argument
 we can extend~$\S$ to a continuous operator
 $W^{1,k}(\R^{n+k}, \, \R^m)\to L^1(B^*, \, \F(\R^{n+k}, \, \GN))$,
 still denoted~$\S$ for simplicity. The property~\eqref{mass}
 is preserved for any~$u\in W^{1,k}(\R^{n+k}, \, \R^m)$,
 by the lower semi-continuity of~$\M$ (see e.g.~\cite[Lemma~3 and Lemma~5]{CO1}).
 
 Since the domain~$\Omega\subseteq\R^{n+k}$ is bounded and smooth, 
 by reflection about~$\partial\Omega$ and multiplication with a cut-off
 function we can define a linear extension
 operator~$T\colon W^{1,k}(\Omega, \, \R^m)\to W^{1,k}(\R^{n+k}, \, \R^m)$,
 such that
 \begin{equation} \label{ext}
  \norm{Tu}_{L^k(\R^{n+k})} \lesssim \norm{u}_{L^k(\Omega)}, \quad
  \norm{\nabla(Tu)}_{L^k(\R^{n+k})} \lesssim
   \norm{\nabla u}_{L^k(\Omega)} + \norm{u}_{L^k(\Omega)} \! .
 \end{equation}
 For any~$u\in W^{1,k}(\Omega, \, \R^m)$ and a.e.~$y\in B^*$,
 the chain~$\S_y(Tu)$ has finite mass, due to~\eqref{mass}.
 Therefore, the restriction
 \[
  \overline{\S}_y(u) := \S_y(Tu)\mres\overline{\Omega}
 \]
 is well-defined and belongs to~$\M_n(\overline{\Omega}; \, \GN)$.
 \end{step}
 
 \begin{step}[$\overline{\S}$ is continuous]
  Let~$(u_j)_{j\in\N}$ be a sequence in~$W^{1,k}(\Omega, \, \R^m)$
  such that $u\to u$ in~$W^{1,k}(\Omega)$.
  From~\eqref{flat} and~\eqref{ext}, we deduce
  \begin{equation} \label{cont1}
   \begin{split}
    \int_{B^*} \F(\S_y(Tu_j) - \S_y(Tu))\,\d y &\lesssim
    \norm{u_j - u}_{L^k(\Omega)} 
    \left(\norm{\nabla u_j}_{L^k(\Omega)}^{k-1} + 
    \norm{\nabla u}_{L^k(\Omega)}^{k-1} \right)
    + \norm{u_j - u}_{L^k(\Omega)}^k 
   \end{split}
  \end{equation}
  Let~$\rho>0$ and~$\Gamma_\rho := \{x\in\R^{n+k}
  \setminus\overline{\Omega}\colon \dist(x, \, \Omega)<\rho\}$.
  By applying Lemma~\ref{lemma:flat_restr} and~\eqref{mass}, \eqref{cont1},
  we obtain
  \begin{equation*}
   \begin{split}
    \|\overline{\S}(u_j) - \overline{\S}(u)\|_{\overline{Y}}
    &\lesssim (1 + \rho^{-1}) \int_{B^*} \F(\S_y(Tu_j) - \S_y(Tu))\,\d y \\
    &\qquad\qquad + \int_{B^*} \M(\S_y(Tu_j)\mres\Gamma_\rho) \, \d y 
    + \int_{B^*} \M(\S_y(Tu)\mres\Gamma_\rho) \, \d y \\
    &\lesssim (1 + \rho^{-1}) \norm{u_j - u}_{L^k(\Omega)} 
    \left(\norm{\nabla u_j}_{L^k(\Omega)}^{k-1} + 
    \norm{\nabla u}_{L^k(\Omega)}^{k-1} \right) \\
    &\qquad\qquad + (1 + \rho^{-1}) \norm{u_j - u}_{L^k(\Omega)}^k 
    + \norm{\nabla(Tu_j)}_{L^k(\Gamma_\rho)}^k
    + \norm{\nabla(Tu)}_{L^k(\Gamma_\rho)}^k \! .
   \end{split}
  \end{equation*}
  By taking the limit in the inequality above first as~$j\to+\infty$,
  then as~$\rho\to 0$, we conclude that $\overline{\S}(u_j)\to\overline{\S}(u)$
  in~$\overline{Y}$.
 \end{step}
 
 \begin{step}[Proof of~\eqref{S:Sbar}]
  By construction, $\overline{\S}_y(u) = \S_y(u)$ for 
  any~$u\in C^\infty_{\mathrm{c}}(\R^{n+k}, \, \R^m)$ and a.e.~$y\in B^*$.
  By continuity of both~$\overline{\S}$ and~$\S$ \cite[Theorem~3.1]{CO1},
  we deduce that $\overline{\S} = \S$ on~$(L^\infty\cap W^{1,k})(\Omega, \, \R^m)$.
 \end{step}
 
 \begin{step}[Proof of~\eqref{S:mass}]
  Let~$E\subseteq\overline{\Omega}$ be a Borel set and~$U\supseteq E$
  be open. By~\eqref{mass}, we have
  \[
   \int_{B^*} \M(\overline{\S}_y(u)\mres E) \, \d y \leq 
   \int_{B^*} \M(\S_y(u)\mres U) \, \d y \lesssim \int_U \abs{\nabla u}^k
  \]
  and~\eqref{S:mass} follows by letting~$U\searrow E$.
 \end{step}
 
 \begin{step}[Proof of~\eqref{S:cobord}]
  Take~$u_0$, $u_1\in W^{1,k}(\Omega, \, \R^m)$. 
  For~$i\in\{0, \, 1\}$ and~$M >0$, we define
  \[
   u_{i,M} := \begin{cases}
               u_i                    & \textrm{if } \abs{u_i}\leq M \\
               \dfrac{Mu_i}{\abs{u_i}} &\textrm{otherwise.}
              \end{cases}
  \]
  Since~$u_{i,M}\to u_i$ strongly in~$W^{1,k}(\Omega)$ as~$M\to 0$,
  the continuity of~$\overline{\S}$ gives, upon extraction of
  a (non-relabelled) subsequence,
  \begin{equation} \label{Scobord1}
   \F(\S_{y}(u_{i, M}) - \S_y(u_{i})) \to 0 \qquad 
   \textrm{as } M\to +\infty,  \textrm{ for a.e. } y\in B^*
   \textrm{ and } i\in\{0, \, 1\}.
  \end{equation}
  Let~$\mathbb{B} := \{\partial R\colon R\in \M_{n+1}(\overline{\Omega}; \, \R^{n+k})\}$.
  By~\cite[Proposition~2]{CO1},
  we have~$\overline{\S}_{y_0}(u_{1,M}) - \overline{\S}_{y_1}(u_{0, M})\in\mathbb{B}$
  for any~$M>0$ and a.e.~$y_0$, $y_1\in B^*$. On the other hand,
  the set~$\mathbb{B}$ is closed with respect to the $\F$-norm,
  as a consequence of the isoperimetric inequality (see e.g.~\cite[7.6]{FedererFleming}).
  Therefore, \eqref{S:cobord} follows by~\eqref{Scobord1}.
 \end{step}

 \begin{step}[Proof of~\eqref{mass}]
  Let~$u\in C^\infty_{\mathrm{c}}(\R^{n+k}, \, \R^m)$ and
  let~$E\subseteq\R^{n+k}$ be a Borel set.
  Since~$\X$ contains finitely many $(m-k)$-cells~$K$,
  there exists a constant~$C$ such that~$\abs{\gamma(K)}_*\leq C$
  for any~$K$. Then, using the definition~\eqref{S} of~$\S_y(u)$,
  we deduce
  \begin{equation} \label{Smass1}
   \M(\S_y(u)\mres E) \lesssim \sum_{K} \H^n\left((u-y)^{-1}(K)\cap E\right) \! ,
  \end{equation}
  where the sum is taken over all the $(m-k)$-dimensional polyhedra~$K$ in~$\X$.
  We fix~$K$ and assume, without loss of generality,
  that~$K\subseteq\left\{y\in\R^m\colon y_1=\ldots = y_k=0 \right\}\simeq\R^{m-k}$.
  Let~$\zeta^\perp$ be the orthogonal projection
  $\R^m\to\left\{y\in\R^m\colon y_{m-k+1}=\ldots = y_m=0 \right\}\simeq\R^{k}$.
  Then,
  \begin{equation*} 
   (u-y)^{-1}(K)\subseteq (\zeta^\perp\circ u)^{-1}(\zeta^\perp(y))
  \end{equation*}
  If we integrate this inequality over~$y\in B^*$, and 
  use the variable~$y = (z, \, z^\perp)\in\R^{m-k}\times\R^m$,
  we obtain
  \[
   \begin{split}
    \int_{B^*} \H^n\left((u-y)^{-1}(K) \cap E\right)
    &\leq \int_{[-\delta^*, \, \delta^*]^{m-k}\times [-\delta^*, \, \delta^*]^k}
    \H^n\left((\zeta^\perp\circ u)^{-1}(z^\perp)\cap E\right) \d (z, \, z^\perp) \\
    &\leq (2\delta^*)^{m-k} \int_{\R^k} 
    \H^n\left((\zeta^\perp\circ u)^{-1}(z^\perp)\cap E\right) \d z^\perp.
   \end{split}
  \]
  The right-hand side can be estimated by applying the coarea formula:
  \begin{equation} \label{Smass2}
   \begin{split}
    \int_{B^*} \H^n\left((u-y)^{-1}(K) \cap E\right)
    &\lesssim \int_{E} \abs{\nabla(\zeta^\perp\circ u)}^k 
    \lesssim \int_{E} \abs{\nabla u}^k.
   \end{split}
  \end{equation}
  Combining~\eqref{Smass1} and~\eqref{Smass2}, \eqref{S:mass-app} follows.
 \end{step}
 
 \begin{step}[Proof of~\eqref{flat}]
  Let~$u_0$, $u_1\in C^\infty_{\mathrm{c}}(\R^{n+k}, \, \R^m)$, and let
  $u\colon [0, \, 1]\times\R^{n+k}\to\R^m$ be defined by
  $u(t, \, x) := (1-t)u_0(x) + tu_1(x)$.
  Let~$\pi\colon [0, \, 1]\times\R^{n+k}\to\R^{n+k}$ be the
  canonical projection, $\pi(t, \, x) := x$. By~\cite[Proposition~4]{CO1},
  we have
  \begin{equation*}
   \S_y(u_1) - \S_y(u_0) = \partial\left(\pi_{*}\S_y(u)\right) \! .
  \end{equation*}
  Therefore, using~\eqref{S}, we obtain
  \begin{equation} \label{Sflat0}
   \F(\S_y(u_1) - \S_y(u_0)) \leq \M(\pi_{*}\S_y(u)) 
   \leq \sum_{K} \M\left(\pi_{*}\llbracket (u-y)^{-1}(K)\rrbracket\right) \! ,
  \end{equation}
  where the sum is taken over all the~$(m-k)$-polyhedra~$K$ in~$\X$.
  Fix such a~$K$. As above, we assume that
  that~$K\subseteq\left\{y\in\R^m\colon y_1=\ldots = y_k=0 \right\}$.
  Let $\zeta$, $\zeta^\perp$ be the orthogonal projections of~$\R^m$ onto
  $\{y\in\R^m\colon y_1=\ldots = y_{m-k}=0\}\simeq\R^{m-k}$,
  $\{y\in\R^m\colon y_{m-k+1}=\ldots = y_m=0\}\simeq\R^{k}$, respectively.
  We wite~$z := \zeta(y)$, $z^\perp := \zeta^\perp(y)$
  and identify~$y = (z, \, z^\perp)$. Then, for a suitable choice 
  of orientation of~$K$, we obtain
  \begin{equation} \label{Sflat1}
   \llbracket (u-y)^{-1}(K)\rrbracket = 
   \llbracket (\zeta^\perp\circ u)^{-1}(z^\perp) \rrbracket
    \mres \left((\zeta\circ u - z)^{-1}(K)\right) \!.
  \end{equation}
  where~$\llbracket (\zeta^\perp\circ u)^{-1}(z^\perp) \rrbracket$
  is the chain carried by the set~$(\zeta^\perp\circ u)^{-1}(z^\perp)$,
  with unit multiplicity, oriented by the Jacobian of~$\zeta^\perp\circ u$
  (see e.g. \cite[p.~72]{CO1}). Let us define~$v := \zeta^\perp\circ u$,
  $K_z := (\zeta\circ u - z)^{-1}(K)$.
  By integrating~\eqref{Sflat1} with respect to~$y\in B^*$,
  we obtain
  \begin{equation*}
   \begin{split}
    \int_{B^*} \M\left(\pi_{*}\llbracket (u-y)^{-1}(K)\rrbracket\right) \d y
    \leq \int_{[-\delta^*, \delta^*]^{m-k}\times\R^k}
     \M\left(\pi_{*}(v^{-1}(z^\perp)\mres K_{z})\right) \d (z, \, z^\perp). 
   \end{split}
 \end{equation*}
 We may write~$v(t, \, x) = (1-t)v_0(x) + tv_1(x)$, where~$v_0 := \zeta^\perp\circ u_0$,
 $v_1 := \zeta^\perp\circ u_1$. By applying~\cite[Lemma~15]{CO1}, we deduce
 \begin{equation} \label{Sflat2}
   \begin{split}
    \int_{B^*} \M\left(\pi_{*}\llbracket (u-y)^{-1}(K)\rrbracket\right) \d y
    &\lesssim (2\delta^*)^{m-k}
    \int_{\R^{n+k}} \abs{v_0 - v_1} 
    \left(\abs{\nabla v_0}^{k-1} + \abs{\nabla v_1}^{k-1}\right) \! .
   \end{split}
 \end{equation}
 Combining~\eqref{Sflat0} and~\eqref{Sflat2}, using that
 the function~$\zeta$ is $1$-Lipschitz, and applying 
 the H\"older inequality, \eqref{flat} follows. \qedhere
 \end{step}
\end{proof}

\section{Energy lower bounds when~\texorpdfstring{$n=0$}{n = 0}}
\label{sect:jerrard}

The aim of this section is to prove energy lower bounds in the 
critical dimension, i.e. when~$n=0$.
In the contest of the Ginzburg-Landau theory, i.e. when~$\NN=\SS^{k-1}$,
energy bounds of this type were proved 
by Jerrard~\cite{Jerrard} and, in case~$k=2$, by Sandier~\cite{Sandier}.

Let~$\delta_0>0$, $r>0$ be small numbers. 
Suppose that a map~$u\in W^{1,k}(\Omega, \, \R^m)$ satisfies
\begin{equation} \label{bd-close-N}
  \dist(u(x), \, \NN)\leq\delta_0 \qquad \textrm{for a.e. }
  x\in\Omega \textrm{ such that } \dist(x, \, \partial\Omega)<r.
\end{equation}
Then, we can define the homotopy class of~$u$ 
(or, more precisely, of~$\RR\circ u$)
on~$\partial\Omega$ as an element of~$\GN$.
This is immediate in case~$u$ is continuous on~$\partial\Omega$ 
and~$\overline{\Omega}$ is homeomorphic to a disk. If~$\Omega$
has not the topology of a disk, this is still possible due to
the Hurewicz isomorphism $\pi_{k-1}(\NN)\simeq H_{k-1}(\NN)$,
which holds true thanks to~\eqref{hp:N}
(see e.g. \cite[Theorem~4.37 p.~371]{Hatcher} and~\eqref{hc} below).
If~$u$ is not continuous 
we can define its homotopy class by approximating~$\RR\circ u$
with smooth functions $\Omega\to\NN$,
as in 
\cite{BN1} (see also~\eqref{hc} below for more details).

\begin{prop} \label{prop:lowerbounds}
 Let~$\Omega\subseteq\R^k$ be a bounded, Lipschitz domain and let~$r>0$.
 There exist a number~$\delta_0>0$, depending only on~$\NN$, and 
 positive constants~$\eps_0$,~$M$, depending only 
 on~$\Omega$, $r$, $\NN$, $k$ and~$f$, such that the following statement holds.
 Suppose that~$u\in W^{1,k}(\Omega, \, \R^m)$ satisfies~\eqref{bd-close-N},
 and let~$\sigma\in\GN$ be the homotopy class of~$u$ on~$\partial\Omega$.
 Let~$\eps\in (0, \, 1/2)$ be such that
 $\eps\abs{\log\eps}|\sigma|_*\leq\eps_0$. Then,
 \[
  E_\eps(u) \geq 
  \abs{\sigma}_*\abs{\log\eps} -
  M\abs{\sigma}_*(1 + \log\abs{\sigma}_*) .
 \]
\end{prop}

The aim of this section is to prove Proposition~\ref{prop:lowerbounds}.
Once Proposition~\ref{prop:lowerbounds} is proved,
Proposition~\ref{prop:JS} follows by an extension argument
in a neighbourhood of~$\partial\Omega$ 
(see e.g.~\cite[Theorem~2]{BethuelDemengel}).
Lemma~\ref{lemma:lowerbounds} 
also follows from Proposition~\ref{prop:lowerbounds},
by exactly the same arguments as in~\cite[Lemma~3.10]{ABO2}. 

\subsection{Reduction to the cone-valued case}

For the purposes of this section, it will be convenient to consider the 
nearest-point projection onto~$\NN$.
If~$z\in\R^m$ is sufficiently close
to~$\NN$, there exists a unique $\pi(z)\in\NN$ such that
$\abs{z-\pi(z)} \leq \abs{z - w}$ for any $w\in\NN$.
Moreover, the map~$z\mapsto \pi(z)$ is a smooth in a neighbourhood of~$\NN$.
Throughout the rest of the section, we fix
a small parameter~$\theta_0$ and assume that 
$\pi$ is well-defined and smooth in a $\theta_0$-neighbourhood of~$\NN$.

\begin{lemma} \label {lemma:proj}
 If~$u\colon\Omega\to\R^m$ is a smooth map that satisfies
 $\dist(u(x), \, \NN)<\theta_0$ for any~$x\in\Omega$
 and if~$d := \dist(u, \, \NN)$, then there holds
 \[
  \abs{\nabla u}^2 \geq C_1\abs{\nabla d}^2 + 
  (1 - C_2d)\abs{\nabla (\pi\circ u)}^2 \qquad \textrm{on } \Omega,
 \]
 where~$C_1$, $C_2$ are positive constants that only depend on~$\NN$.
\end{lemma}
\begin{proof}
 Let~$x_0\in\Omega$ be arbitrarily fixed.
 Let~$\nu_1, \, \ldots, \, \nu_p$ be a smooth orthonormal frame for the normal
 space to~$\NN$, locally defined in a neighbourhood of~$(\pi\circ u)(x_0)$.
 Then, for each~$x$ in a neighbourhood of~$x_0$ there exist numbers 
 $\alpha_1(x), \, \ldots, \, \alpha_p(x)$ such that
 \[
  u(x) = (\pi\circ u)(x) + \sum_{i=1}^p \alpha_i(x) \, (\nu_i\circ\pi\circ u)(x).
 \]
 The functions~$\alpha_i$ are as regular as~$u$. By differentiating this equation,
 raising both sides to the square, using the fact that $\nabla(\pi\circ u)$
 is tangent to~$\NN$ and that $(\nabla\nu_i)\nu_i = 0$, we obtain
 \[
  \begin{split}
   \abs{\nabla u}^2 - \abs{\nabla (\pi\circ u)}^2 = \sum_{i=1}^p \left( 
   \abs{\nabla\alpha_i}^2 + \alpha_i^2 \abs{\nabla(\nu_i\circ\pi\circ u)}^2
   + 2\alpha_i\nabla(\pi\circ u):\nabla(\nu_i\circ\pi\circ u)\right) \!.
  \end{split}
 \]
 Since~$\NN$ is smooth and compact, we have $|\nabla\nu_i|\leq C$ for some 
 constant~$C$ that only depends on~$\NN$ and not on~$u$. Therefore, 
 setting $d := \dist(u, \, \NN) = (\sum_i\alpha_i^2)^{1/2}$, we obtain
 \begin{equation} \label{proj1}
  \abs{\nabla u}^2 - \abs{\nabla (\pi\circ u)}^2 
  \geq \sum_{i=1}^p \abs{\nabla\alpha_i}^2 - Cd \abs{\nabla (\pi\circ u)}^2.
 \end{equation}
 On the other hand, by differentiating the identity
 $d = (\sum_i\alpha_i^2)^{1/2}$, we see that
 \begin{equation} \label{proj2}
  \abs{\nabla d}^2 \leq \left(\sum_{i=1}^p \abs{\nabla\alpha_i}\right)^2 
  \lesssim \sum_{i=1}^p \abs{\nabla\alpha_i}^2.
 \end{equation}
 By combining~\eqref{proj1} and~\eqref{proj2}, the lemma follows.
\end{proof}

\begin{lemma} \label{lemma:cone}
 Suppose that~$f\colon\R^m\to\R$ satisfies the
 assumptions~\eqref{hp:first}--\eqref{hp:non-degeneracy}.
 Then, there exist positive constants~$\alpha$, $\beta$ and a smooth
 function~$\phi\colon \R^m\to [0, \, 1]$ such that
 the following holds:
 \begin{enumerate}[label=(\roman*)]
  \item $\phi(y) = 1$ for any~$y\in\NN$;
  \item $\phi(y) = 0$ if~$\dist(y, \, \NN)\geq \theta_0$, and in particular
  $\pi(y)$ is well-defined for any~$y\in\R^m$ such that~$\phi(y) >0$;
  \item for any~$u\in W^{1,k}(\Omega, \, \R^m)$, there holds
   \[
     \frac{1}{k}\abs{\nabla u}^k + \frac{1}{\eps^k} f(u)
     \geq \alpha\abs{\nabla(\phi\circ u)}^k 
     + \frac{1}{k}(\phi\circ u)^k\abs{\nabla(\pi\circ u)}^k
     + \frac{\beta}{\eps^k} (1 - \phi\circ u)^2
   \]
   pointwise a.e. on~$\Omega$.
 \end{enumerate}
\end{lemma}
\begin{proof}
 Let~$u\in W^{1,k}(\Omega, \, \R^m)$ be given.
 By a density argument, we can assume without loss of generality
 that~$u$ is smooth. Let~$d:=\dist(u, \, \NN)$, and let~$x_0\in\Omega$ 
 be such that $d(x_0)<\theta_0$. By applying Lemma~\ref{lemma:proj},
 and using the convexity of the function $t\mapsto t^{k/2}$, we see that the inequality
 \begin{equation} \label{cone1}
  \abs{\nabla u}^k \geq 
  C_1 \abs{\nabla d}^k + (1 - C_2 d)\abs{\nabla (\pi\circ u)}^k
 \end{equation}
 holds pointwise in a neighbourhood of~$x_0$
 (though we may need to re-define
 the constants~$C_1$, $C_2$.)
 Let~$\xi\in C^\infty_{\mathrm{c}}[0, \, +\infty)$
 be a non-increasing function,
 such that $\xi = 1$ in a neighbourhood of~$0$ 
 and~$\xi(\min\{\theta_0/2, 1/(2C_2)\}) = 0$. We set
 \begin{equation*} 
  \phi(y) := \left(1 - C_2\dist(y, \, \NN)\right)^{1/k} \, \xi(\dist(y, \, \NN))
 \end{equation*}
 for any~$y\in\R^m$. This defines a smooth 
 function~$\phi\colon\R^m\to [0, \, 1]$ which satisfies~(i) and~(ii).
 Since~$(\phi\circ u)^k\leq 1 - C_2 d$ and 
 $|\nabla(\phi\circ u)| \lesssim |\nabla d|$, 
 from~\eqref{cone1} we deduce that
 \begin{equation} \label{cone2}
  \frac{1}{k}\abs{\nabla u}^k \geq \alpha\abs{\nabla(\phi\circ u)}^k
  + \frac{1}{k}(\phi\circ u)^k\abs{\nabla (\pi\circ u)}^k
 \end{equation}
 pointwise in the open set~$\{d<\theta_0\}$. Here~$\alpha$ is a positive constant
 that only depends on~$\NN$, $k$ and~$\xi$. Because the function~$\phi\circ u$
 is identically equal to zero on the open set~$\{d>\theta_0/2\}$, 
 the inequality~\eqref{cone2} actually holds in the whole of~$\Omega$.
 
 We consider now the potential term~$f(u)$. 
 Due to the assumption~\eqref{hp:non-degeneracy},
 $f(u)\gtrsim d^2$ and hence, 
 there exists a positive number~$\beta>0$ such that
 \begin{equation} \label{cone3}
  f(u) \geq \beta\left(1 - (\phi\circ u)^k\right)^2
  \qquad \textrm{in } \Omega.
 \end{equation}
 By combining \eqref{cone2} and~\eqref{cone3}, and using the elementary inequality
 $1 - x^k 
 \geq 1 - x$ for~$0 \leq x \leq 1$, the lemma follows.
\end{proof}

\subsection{Proof of Proposition~\ref{prop:lowerbounds}}

Throughout this section, we fix a bounded, smooth map~$u\colon\Omega\to\R^m$
and we let~$s:= \phi\circ u$, $v:=\pi\circ u$, where~$\phi$ is the function
given by Lemma~\ref{lemma:cone} and~$\pi$ is the nearest-point projection
onto~$\NN$. Thanks to Lemma~\ref{lemma:cone},
in order to provide lower bounds for~$E_\eps(u)$ it suffices to 
bound from below the functional
\begin{equation} \label{G_eps}
 G_\eps(s, \, v) = G_\eps(s, \, v; \, \Omega) := 
   \int_\Omega\left( \alpha\abs{\nabla s}^k 
   + \frac{s^k}{k}\abs{\nabla v}^k + \frac{\beta}{\eps^k} (1 - s)^2 \right) \!.
\end{equation}
To this end, we adapt Jerrard's approach in~\cite{Jerrard}.
We explain here the main steps of the construction and 
point out the differences, referring the reader to~\cite{Jerrard} for more details.

Let us fix a small number~$\eta_0$, such that
\begin{equation} \label{eta0}
 0 < \eta_0 < \dist(\NN, \, \X) - \theta_0.
\end{equation}
Let~$V\subseteq\Omega$ be an open set such that~$s>0$ on~$\partial V$.
By Lemma~\ref{lemma:cone}, we have $\dist(u(x), \, \NN) < \theta_0$ for 
any $x\in\partial V$. Therefore, by~\eqref{eta0}, we have 
$\spt(\S_y(u))\cap \partial V = \emptyset$
for a.e.~$y\in\R^m$ such that~$|y|\leq\eta_0$.
In fact, $\S_y(u)$ is a $0$-chain of finite mass, and hence we can write
\begin{equation} \label{0chain}
 \S_y(u)\mres\overline{V} := \sum_{i=1}^q \sigma_i\llbracket x_i\rrbracket
\end{equation}
where~$\sigma_i\in\GN$ and~$x_i\in V$. The quantity
$\I(\S_y(u), \llbracket V\rrbracket) := \sum_{i=1}^q \sigma_i\in\GN$
plays the r\^ole of the topological degree 
and indeed, it coincides with the homotopy class of~$\pi\circ u$ on~$\partial V$
because of Proposition~\ref{prop:S-app}.\eqref{S:intersection-app} 
and~\eqref{eta0} (see \cite[Section~2.4 and Theorem~3.1]{CO1} for the
details on the case~$u$ is not smooth).
In particular, $\I(\S_y(u),  \llbracket V\rrbracket)$
is independent of the choice of~$y$.

As in~\cite{Jerrard}, we define an ``approximate homotopy class'',
which allows us to disregard sets where~$s$ is small
but $u$ does not carry topological obstruction. We 
let~$S := \{x\in\Omega\colon s(x)\leq 1/2\}$.
The components~$\tilde{S}$ of~$S$ are closed sets and it is possible
to define~$\I(\S_y(u), \, \llbracket \tilde{S}\rrbracket)$ as above.
We define the ``essential part'' of~$S$:
\[
 S_E := \cup \left\{\tilde{S} \textrm{ component of }S \colon 
 \I(\S_y(u), \, \llbracket \tilde{S}\rrbracket) \neq 0 \right\} \!.
\]
For~$V\subseteq\Omega$, 
we define the ``approximate homotopy class'' of~$u$ on~$\partial V$ as
\begin{equation} \label{hc}
 \hc(u, \, \partial V) := \sum_{\tilde{S}} \I(\S_y(u), 
    \, \llbracket \tilde{S}\rrbracket)\in\GN,
\end{equation}
where the sum is taken over all the
components~$\tilde{S}$ of~$S_E$ such that~$\tilde{S}\csubset V$.
If~$V\subseteq\Omega$ is an open disk and~$s>1/2$ on~$\partial V$, 
then~$\hc(u, \, \partial V)$ is the homotopy class 
of~$v\colon\partial V\simeq\SS^{k-1}\to\NN$.

For any~$\rho>0$, we define
\begin{equation*} 
 \lambda_{\eps}(\rho) := \min_{0\leq \mu \leq 1}
 \left\{\frac{\mu^k}{\rho}+ \frac{C_0}{\eps}(1 - \mu)^N \right\}
\end{equation*}
and
\[
 \Lambda_{\eps}(\rho) := \int_0^{\rho} \min\left\{\lambda_{\eps}(s),
 \, \frac{C_1}{\eps}\right\} \d\rho,
\]
where~$C_0>0$, $C_1>0$ and~$N>1$ are parameters that 
we will need to choose, depending on~$k$, $\alpha$ and~$\beta$.
It can be shown (see \cite[Theorem~2.1, proof of~(2.2)]{Jerrard}) that 
\begin{equation*} 
 \lambda_{\eps}(\rho) \geq 
 \frac{1}{\rho} \left(1 - C \frac{\eps^\nu}{\rho^\nu}\right) \!,
\end{equation*}
where~$C>0$ only depends on~$k$, $C_0$ and~$\nu := 1/(N-1)>0$.
As a consequence, $\lambda_\eps(\rho)\geq C_1/\eps$ if~$\rho\geq C_2\eps$,
for some constant~$C_2>0$ depending on~$k$, $C_0$, $C_1$ and~$N$.
Therefore, after integration we obtain that
\begin{equation} \label{lb-log}
 \Lambda_{\eps}(\rho) \geq 
 \log\frac{\rho}{\eps} - C
\end{equation}
for any~$\rho\geq 0$, where the constant~$C$ only depends on~$k$, $\alpha$, $\beta$.
We have the following analogue of \cite[Proposition~3.2]{Jerrard}.

\begin{lemma} \label{lemma:annulus}
 Let~$\eps\leq \rho_1\leq\rho_2$ and let
 $u\in W^{1,k}(B^k_{\rho_2}\setminus B^k_{\rho_1}, \, \R^m)$ be smooth.
 Suppose that $\hc(u, \, \partial B^k_\rho) = \sigma$ 
 for any $\rho\in (\rho_1, \, \rho_2)$. Then, there holds
 \[
  E_\eps(u, \, B^k_{\rho_2}\setminus B^k_{\rho_1}) \geq
  \abs{\sigma}_*\left(\Lambda_\eps\left(\frac{\rho_2}{\abs{\sigma}_*}\right) 
  - \Lambda_\eps\left(\frac{\rho_1}{\abs{\sigma}_*}\right)\right) \! .
 \]
\end{lemma}
\begin{proof}
 First of all, given $\rho>0$ and a map
 $v\in W^{1,k}(\partial B^k_\rho, \, \NN)$ in
 the homotopy class~$\sigma\in\GN$, there holds
 \begin{equation} \label{start}
  \frac{1}{k}\int_{\partial B^k_\rho} \abs{\nabla v}^k 
  \geq \frac{E_{\min}(\sigma)}{\rho}
 \end{equation}
 where $E_{\min}(\sigma)$ is defined by~\eqref{I_min}.
 This inequality follows immediately from the definition of $E_{\min}(\sigma)$,
 combined with a scaling argument.
 
 Now, for any~$\rho\geq\eps>0$ and any smooth~$u\colon\partial B^k_\rho\to\R^m$
 such that $\mu := \min_{\partial B^k_\rho} \phi\circ u > 1/2$, there holds
 \begin{equation*} 
  G_\eps(s, \, v; \, \partial B^k_\rho) \geq 
  \frac{E_{\min}(\sigma)}{\rho}\mu^k + \frac{C_0}{\eps}(1 - \mu)^N.
 \end{equation*}
 Here $G_\eps$ is defined by~\eqref{G_eps}, $s:=\phi\circ u$, 
 $v:=\pi\circ u$, and~$\sigma$ denotes the homotopy class of $v$ on~$\partial B^k_\rho$.
 The constants~$C_0$, $N$ are suitably chosen at this stage.
 The proof of this claim follow by repeating, almost word by word,
 the arguments in~\cite[Theorem~2.1]{Jerrard}; the only difference is that 
 we need to apply~\eqref{start} instead of~\cite[Lemma~2.4]{Jerrard}.
 Due to~\eqref{I_norm}, we obtain
 \begin{equation} \label{lb-circle1}
  G_\eps(s, \, v; \, \partial B^k_\rho) \geq 
   \frac{\mu^k}{\rho/\abs{\sigma}_*}  + \frac{C_0}{\eps}(1 - \mu)^N
  \geq \lambda_{\eps}\left(\frac{\rho}{\abs{\sigma}_*}\right) \! .
 \end{equation}
 On the other hand, in case~$0\leq\mu\leq 1/2$, \cite[Lemma~2.3]{Jerrard}
 implies that
 \begin{equation} \label{lb-circle2}
  G_\eps(s, \, v; \, \partial B^k_\rho) \geq \frac{C_1}{\eps}
 \end{equation}
 for some~$C_1>0$ that depends on~$k$, $C_0$ and~$N$. Therefore,
 by integrating the inequalities~\eqref{lb-circle1}--\eqref{lb-circle2}
 with respect to~$\rho$, we deduce that
 \[
  \begin{split}
   G_\eps(s, \, v; \, B^k_{\rho_2}\setminus B^k_{\rho_1}) \geq
   \int_{\rho_1}^{\rho_2} \min\left(\lambda_{\eps}
      \left(\frac{\rho}{|\sigma|_*}\right), \, \frac{C_1}{\eps}\right) \d\rho
   = |\sigma|_* \int_{\rho_1/|\sigma|_*}^{\rho_2/|\sigma|_*} \min\left(\lambda_{\eps}
      (s), \, \frac{C_1}{\eps}\right) \d s
  \end{split}
 \]
 and, thanks to Lemma~\ref{lemma:cone}, the lemma follows.
\end{proof}

We also have an analogue of \cite[Proposition~3.3]{Jerrard}.

\begin{lemma} \label{lemma:SE}
 Suppose that~$u\in W^{1,k}(\Omega, \, \R^m)$ is smooth and
 that~$S_E\csubset\Omega$. Then, there exists a finite collection 
 of closed, pairwise disjoint balls~$(B_i)_{i=1}^p$, of radii~$\rho_i\geq\eps$,
 such that $S_E\subseteq \cup_{i=1}^p B_i$, 
 $B_i\cap S_E\neq\emptyset$ for any~$i$, and
 \[
  E_\eps(u; \, B_{\rho_i}\cap\Omega) \geq \frac{C_1}{\eps}\rho_i. 
 \]
\end{lemma}
\begin{proof}
 We claim that, if~$\tilde{S}$ is a connected component of~$S_E$
 such that~$\tilde{S}\csubset\Omega$, then 
 \begin{equation} \label{SE1}
  \int_{\tilde{S}} \abs{\nabla u}^k \gtrsim \abs{\hc(u, \, \partial \tilde{S})}_* \! .
 \end{equation}
 This inequality parallels \cite[Lemma~3.2]{Jerrard}; once \eqref{SE1}
 is established, the rest of the proof follows exactly as in~\cite{Jerrard}.
 The definition~\eqref{hc} of~$\hc$ and~\eqref{0chain}
 imply that 
 \begin{equation} \label{SE2}
  \abs{\hc(u, \, \partial \tilde{S})}_* = 
  \abs{\I(\S_y(u), \, \llbracket \tilde{S}\rrbracket)}_* 
  \leq \M(\S_y(u)\mres \tilde{S})
 \end{equation}
 for a.e.~$y\in\R^m$ such that~$|y|\leq\eta_0$. On the other hand, 
 \eqref{S:mass-app} gives that 
 \[
  \int_{\R^m} \M(\S_y(u)\mres \tilde{S})\, \d y \lesssim
  \int_{\tilde{S}} \abs{\nabla u}^k,
 \]
 so there exists (a set of positive measure of) $y$ such that $|y|\leq\eta_0$
 and $\M(\S_y(u)\mres \tilde{S}) \lesssim\int_{\tilde{S}}|\nabla u|^k$. Then, 
 \eqref{SE1} follows from~\eqref{SE2}.
\end{proof}

Lemma~\ref{lemma:SE} and the definition of~$\Lambda_\eps$
imply that
\begin{equation} \label{SE}
 E_\eps(u; \, B_{i}) \geq \abs{\hc(u, \, \partial B_i)}_*
 \Lambda_\eps \, \left(\frac{\rho_i}{\abs{\hc(u, \, \partial B_i)}_*}\right) 
 \qquad \textrm{for any } i.
\end{equation}

The last step in the proof of Proposition~\ref{prop:lowerbounds}
is the so-called ``ball construction'' \cite[Proposition~4.1]{Jerrard}. 
If~$u$ satisfies~\eqref{bd-close-N} for some~$r>0$ then,
by choosing~$\delta_0 = \delta_0(\NN) < \theta_0$ sufficiently 
small, we obtain as a consequence
\begin{equation} \label{bd-s}
 |s(x)|\geq \frac{1}{2} \qquad \textrm{for any }
 x\in\Omega \textrm{ such that } \dist(x, \, \partial\Omega)<r.
\end{equation}
Moreover, we can assume without loss of generality that~$u$ satisfies
\begin{equation} \label{upper_bd}
 E_\eps(u) \leq 
 \abs{\hc(u, \, \partial\Omega)}_*\abs{\log\eps} + C
\end{equation}
for some~$\eps$-independent constant~$C$, for otherwise 
Proposition~\ref{prop:lowerbounds} holds trivially.

\begin{lemma} \label{lemma:ballconstruction}
 There exists a constant~$\eps_0>0$ such that the following statement holds.
 Let~$u\in W^{1,k}(\Omega, \, \R^m)$ be a smooth function
 that satisfies~\eqref{bd-s} for some~$r>0$ and~\eqref{upper_bd}.
 For any~$\tau > 0$ and any~$\eps\in (0, \, 1/2)$ such that
 \begin{equation} \label{sigma-eps}
  4\tau\abs{\hc(u, \, \partial\Omega)}_* < r, \qquad 
  \eps\abs{\log\eps}\abs{\hc(u, \, \partial\Omega)}_*\leq \eps_0,
 \end{equation}
 there exists a finite collection of closed ball~$(\tilde{B}_i)_{i=1}^q$,
 of radii~$r_i$, that satisfy the following properties:
 \begin{enumerate}[label=(\roman*)]
  \item the interiors of the balls are pairwise disjoint;
  \item $S_E\csubset \cup_{i=1}^q \tilde{B}_i$ and~$\tilde{B}_i\cap S_E\neq\emptyset$ for any~$i$;
  \item letting~$s:= \min_{i} r_i/|\hc(u, \, \partial \tilde{B}_i)|_*$, we have
  \[
   E_\eps(u, \, \tilde{B}_i\cap\Omega) \geq \frac{r_i}{s} \Lambda_\eps(s);
  \]
  \item $\tau/2 \leq s \leq\tau$;
  \item $\abs{\hc(u, \, \partial\Omega)}_* = 
  \sum_{i=1}^q |\hc(u, \, \partial \tilde{B}_i)|_*$.
 \end{enumerate}
\end{lemma}
Lemma~\ref{lemma:ballconstruction} follows by repeating the 
arguments of~\cite[Proposition~4.1]{Jerrard} (see 
also~\cite[Remark at p.~22]{ABO2}), and using 
Lemmas~\ref{lemma:annulus}, \ref{lemma:SE} and~\eqref{SE}. 

\begin{proof}[Proof of Proposition~\ref{prop:lowerbounds}]
 We assume that~$u$ is smooth, 
 satisfies~\eqref{bd-s} (as a consequence of our 
 assumption~\eqref{bd-close-N}) and~\eqref{upper_bd}. We apply 
 Lemma~\ref{lemma:ballconstruction} and use the fact that, by definition of~$s$,
 $|\hc(u, \, \partial \tilde{B}_i)|_*\leq r_i/s$ for any~$i$:
 \[
  \begin{split}
  E_\eps(u, \, \Omega) &\stackrel{(i)-(iii)}{\geq} 
  \sum_{i=i}^p \frac{r_i}{s} \Lambda_\eps(s) \geq 
  \sum_{i=i}^p \abs{\hc(u, \, \partial \tilde{B}_i)}_* \Lambda_\eps(s)
  \stackrel{(v)}{=} \abs{\hc(u, \, \partial\Omega)}_* \Lambda_\eps(s)\\
  &\stackrel{(iv)}{\geq} \abs{\hc(u, \, \partial\Omega)}_* 
  \Lambda_\eps\left(\frac{\tau}{2}\right) 
  \stackrel{\eqref{lb-log}}{\geq} 
  \abs{\hc(u, \, \partial\Omega)}_* \log\frac{\tau}{2\eps} - C.
  \end{split}
 \]
 The constant~$C$ here only depends on~$k$, $\alpha$, $\beta$.
 Now, we choose $\tau := r/(8|\hc(u, \, \partial\Omega)|_*)$
 (which is admissible in view of~\eqref{sigma-eps}).
 Taking~\eqref{discrete_norm} into account, we obtain the desired estimate
 in case~$u$ is smooth. Now the proposition follows
 by a density argument.
\end{proof}

\section{Technical results about flat chains}

Throughout this appendix, we consider chains with coefficients in
a normed Abelian group~$(\G, \, |\cdot|)$ such that
\begin{equation} \label{G_discrete}
 \inf_{g\in\G\setminus\{0\}} \abs{g} > 0.
\end{equation}
This assumption is satisfied by~$(\GN, \, |\cdot|_*)$,
due to Proposition~\ref{prop:group_norm}.

\subsection{Approximation results for flat chains}
\label{sect:approximation}

We give the proof of the approximation results,
Proposition~\ref{prop:approx-noappendix} 
and~\ref{prop:approx_mult_noappendix}, we have used 
in Section~\ref{sect:reduction}.
For convenience, we recall the statements here.
Let~$\Sg\subseteq\G$ be a set of generators for~$\G$. We assume that,
for any~$g\in\G$, there exist $g_1, \, \ldots, \, g_p\in\Sg$
such that
\begin{equation} \label{hp:mathfrakS}
 g = \sum_{i=1}^p g_i \qquad
 \textrm{ and } \qquad \abs{g} = \sum_{i=1}^p \abs{g_i}.
\end{equation}
The set defined by~\eqref{generators} satisfies this
assumption, by Proposition~\ref{prop:group_norm}.

\begin{prop} \label{prop:approx_mult}
 Let $S\in\M_{n}(\R^{n+k}; \, \G)$ be a polyhedral chain.
 Let~$W_{\Sg}\subseteq\R^{n+k}$ be an open set, with polyhedral boundary,
 such that $\partial W_{\Sg}$ is transverse to~$\spt S$
 (i.e., there exist triangulations of~$\partial W_{\Sg}$
 and~$\spt S$ such that any simplex of the triangulation of~$\partial W_{\Sg}$
 is transverse to any simplex of the triangulation of~$\spt S$).
 Then, there exists a sequence of polyhedral 
 $(n+1)$-chains~$R_j$, supported in $\overline{W_{\Sg}}$,
 such that the following hold:
 \begin{enumerate}[label=(\roman*)]
  \item $S+\partial R_j \to S$, with respect 
  to the~$\F$-norm, as~$j\to+\infty$;
  \item $\M(S+\partial R_j)\to\M(S)$ as~$j\to+\infty$;
  \item for any~$j$, $(S + \partial R_j)\mres\partial W_{\Sg} = 0$;
  \item for any~$j$, the chain $(S + \partial R_j)\mres W_{\Sg}$ takes multiplicities 
  in the set~$\Sg\subseteq\GN$ defined by~\eqref{generators-intro}.
 \end{enumerate}
\end{prop}
Proposition~\ref{prop:approx_mult} implies Proposition~\ref{prop:approx_mult_noappendix}.
\begin{proof}
 Since~$\partial W_{\Sg}$ is transverse to~$\spt S$,
 the intersection~$\spt S \cap\partial W_{\Sg}$ has dimension~$n-1$ 
 at most and hence, $S\mres\partial W_{\Sg} = 0$.
 By triangulating~$S$, we can write~$S\mres W_{\Sg}$ as a finite sum
 \[
  S\mres W_{\Sg} = \sum_{K} \sigma_K\llbracket K\rrbracket,
 \]
 where~$\sigma_K\in\GN$ and the~$K$'s are closed $n$-simplices,
 whose interiors are contained in~$W_{\Sg}$ and pairwise disjoint.
 We fix positive parameters~$\delta$, $\gamma$ and,
 for any~$n$-simplex~$K$ of~$S\mres W_{\Sg}$, we consider the
 set~$U(K, \, \delta, \, \gamma)$ defined by~\eqref{U-diamond}.
 We choose~$\delta$, $\gamma$ small enough, so that
 the interiors of the~$U(K, \, \delta, \, \gamma)$'s are
 pairwise disjoint and contained in~$W_{\Sg}$.
 By assumption~\eqref{hp:mathfrakS}, we can write 
 $\sigma_K = \sum_{i=1}^p \sigma_{K,i}$ where $\sigma_{K,i}\in\Sg$ and
 \begin{equation} \label{mult1}
  \abs{\sigma_K} = \sum_{i=1}^p \abs{\sigma_{K,i}} .
 \end{equation}
 Take distinct vectors~$y^{K,1}, \, \ldots, \, y^{K,p}\in\R^{n+k}$ that are orthogonal 
 to~$K$ and satisfy $|y^{K,1}| = \ldots = |y^{K,p}| = 1$.
 For each~$i\in\{1, \, \ldots, \, p\}$, 
 we define $h^{K,i}\colon [0, \, 1]\times K\to \R^{n+k}$ by
 \[
  h^{K,i}(t, \, x^\prime) := x^\prime + t\min\left\{\delta , \,
  \gamma\dist(x^\prime, \, \partial K)\right\} y^{K,i}
  \qquad \textrm{for any } (t, \, x^\prime)\in [0, \, 1]\times K.
 \]
 For any integer~$j\geq 1$, we define
 \[
  R_j := \sum_K \sum_{i=1}^p \sigma_{K,i} \ h^{K,i}_{*}( 
  \llbracket[0, \, 1/j]\rrbracket\times \llbracket K\rrbracket)
 \]
 (see Figure~\ref{fig:ratp}).
 \begin{figure}[t]
	\centering
    \includegraphics[height=.23\textheight]{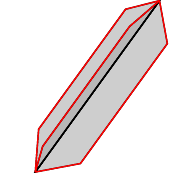}
    
    \bigskip
	\caption{The chain~$R_j$, in case~$n=1$, $k=2$ and~$S$
	consists of a segment only, $S = \sigma_K\llbracket K \rrbracket$.
	The chain~$S$ is in black, $R_j$ is in gray, and
	$S + \partial R_j$ is in red.}
	\label{fig:ratp}
\end{figure}
 The chain~$R_j$ is polyhedral, because 
 the~$h^{K,i}$'s are piecewise affine,
 and supported in~$\overline{W_{\Sg}}$. 
 The support of~$R_j$ may intersect~$\partial W_{\Sg}$ 
 only along its $(n-1)$-skeleton, so
 $(\partial R_j)\mres\partial W_{\Sg} = 0$.
 We compute the mass of~$R_j$.
 Since the maps~$h^{K,i}$ are Lipschitz, and their Lipschitz constant
 only depends on~$\gamma$, which is fixed,
 the area formula implies
 \begin{equation*} 
  \M(R_j) \lesssim \sum_{K} \sum_{i=1}^p \abs{\sigma_{K,i}}_*
  \H^{n+1}([0, \, 1/j]\times K)
  \stackrel{\eqref{mult1}}{\le} j^{-1} \, \M(S\mres W_{\Sg})\to 0
  \qquad \textrm{as } j\to+\infty.
 \end{equation*}
 Thus, (i) follows. Now, we compute the boundary of~$R_j$.
 For each simplex~$K$ and each~$i$, we have
 $h^{K,i}(t, \, x^\prime) = x^\prime$ if~$x^\prime\in\partial K$
 and~$h^{K,i}(0, \, x^\prime) = x^\prime$ for any~$x^\prime\in K$.
 As a consequence,
 \[
  \begin{split}
   \partial h^{K,i}_{*}\left(\llbracket[0, \, 1/j]\rrbracket\times \llbracket K\rrbracket\right) 
   &= h^{K,i}_{*}\left(\llbracket1/j\rrbracket\times \llbracket K\rrbracket
   - \llbracket 0\rrbracket\times \llbracket K\rrbracket\right) - 
   h^{K,i}_{*}\left(\llbracket[0, \, 1/j]\rrbracket\times \partial\llbracket K\rrbracket\right) \\
   &= h^{K,i}(j^{-1}, \, \cdot)_{*}\llbracket K\rrbracket - \llbracket K\rrbracket.
  \end{split}
 \]
 By multiplying this identity by~$\sigma_{K,i}$,
 and taking the sum over~$i$, $K$,
 we obtain
 \begin{equation} \label{partialRj}
  \partial R_j = \sum_K \sum_{i=1}^p \sigma_{K,i} \,  
  h^{K,i}(j^{-1}, \, \cdot)_{*}\llbracket K\rrbracket - S\mres W_{\Sg}.
 \end{equation}
 In particular, $S\mres W_{\Sg} + \partial R_j = (S + \partial R_j)\mres W_{\Sg}$
 takes multiplicities in~$\Sg$.
 Finally, by applying the area formula to~\eqref{partialRj},
 we deduce
 \[
  \M(S\mres W_{\Sg} + \partial R_j) \to 
  \sum_K \sum_{i=1}^p \abs{\sigma_{K,i}}_* \, \H^n(K) 
  \stackrel{\eqref{mult1}}{=} \M(S\mres W_{\Sg})
  \qquad \textrm{as } j\to+\infty,
 \]
 and~(ii) follows.
\end{proof}

Let~$\Omega\subseteq\R^{n+k}$ be a domain
and let~$S\in\M_n(\overline{\Omega}; \, \G)$. 
Recall that~$S$ is called
locally polyhedral if, for any compact set 
$K\subseteq\Omega$, there exists a polyhedral 
chain~$T$ such that $(S - T)\mres K = 0$. 
We write~$S_0\sim_{\overline{\Omega}} S_1$ if there exists 
$R\in\M_{n+1}(\overline{\Omega}; \, \G)$ such that $S_1 = S_0 + \partial R$.

\begin{prop} \label{prop:approx}
 Let~$\Omega\subseteq\R^{n+k}$ be a bounded, Lipschitz domain.
 Let~$S_0\in\M_n(\overline{\Omega}; \, \G)$ be a locally polyhedral 
 chain such that~$S_0\mres\partial\Omega = 0$. 
 Let~$S\in\M_n(\overline{\Omega}; \, \G)$ be such that $S \sim_{\overline{\Omega}} S_0$.
 Then, there exists a sequence of polyhedral $(n+1)$-chains~$R_j$,
 with compact support in~$\Omega$, such that
 $S_0 + \partial R_j\to S$ (with respect to the $\F$-norm)
 and $\M(S_0 + \partial R_j)\to\M(S)$ as~$j\to+\infty$.
\end{prop}

Proposition~\ref{prop:approx-noappendix} follows from
Proposition~\ref{prop:approx}, with the help
of~\eqref{S:cobord} and Lemma~\ref{lemma:u*}.
We split the proof of Proposition~\ref{prop:approx} into 
several lemmas. The first one is a straightforward consequence
of the deformation theorem for flat chains;
we provide a proof for completeness.

\begin{lemma} \label{lemma:def}
 Let~$q\in\{0, \, \ldots, \, n+k-1\}$,
 $T\in\M_q(\R^{n+k}; \, \G)$ anq~$\eta > 0$ be given.
 Suppose that~$T$ is compactly supported and~$\partial T$ is polyhedral.
 Then, there exist a polyhedral $q$-chain~$P$ 
 and a finite-mass chain~$C\in\M_{q+1}(\R^{n+k}; \, \G)$,
 supported in the $\eta$-neighbourhood of~$\spt T$, that satisfy
 \begin{gather*}
  T = P + \partial C, \\
  \M(P) \lesssim \M(T) + \eta \, \M(\partial T), 
   \qquad \M(C) \lesssim \eta \, \M(T).
 \end{gather*}
\end{lemma}
\begin{proof}
 We apply the deformation theorem (see e.g. \cite[Theorem~7.3]{Fleming}
 or \cite[Theorem~1.1]{White-Deformation}) to~$T$.
 We find a polyhedral $q$-chain $A$, a finite-mass $q$-chain $B$
 and a finite-mass $(q+1)$-chain~$C$
 that satisfy the following properties:
 \begin{enumerate}[label=(\alph*)]
  \item $T = A + B + \partial C$;
  \item $\M(A)\lesssim \M(T) + \eta\,\M(\partial T)$, 
  $\M(B)\lesssim\eta\,\M(\partial T)$ and~$\M(C)\lesssim\eta\,\M(T)$;
  \item $A$, $B$, $C$ are supported in the $\eta$-neighbourhood of~$\spt T$.
 \end{enumerate}
 Since we have assumed that~$\partial T$ is polyhedral,
 we can take~$B$ to be polyhedral, too
 (see e.g.~\cite[Theorem~1.1.(7)]{White-Deformation}). Then,
 the chains~$P:=A+B$ and~$C$ have all the required properties.
\end{proof}

\begin{lemma} \label{lemma:compactspt}
 Let~$\Omega\subseteq\R^{n+k}$ be a bounded, Lipschitz domain.
 Let~$S_0\in\M_n(\overline{\Omega}; \, \G)$ be such that~$S_0\mres\partial\Omega= 0$.
 Let~$R\in\M_{n+1}(\overline{\Omega}; \, \G)$ be
 such that~$\M(\partial R)<+\infty$.
 Then, there exists a sequence of chains
 $R_j\in\M_{n+1}(\overline{\Omega}; \, \G)$,
 compactly supported in~$\Omega$, such that
 $\partial R_j\to\partial R$ (with respect to the~$\F$-norm)
 and $\M(S_0 + \partial R_j)\to\M(S_0 + \partial R)$ as~$j\to +\infty$.
\end{lemma}

The proof of Lemma~\ref{lemma:compactspt} is identical to 
that of~\cite[Proposition~8.6]{ABO2}.
In~\cite{ABO2}, the authors work in the setting of currents;
however, the arguments used in the proof of Proposition~8.6 
carry over to the setting of flat chains, thanks to the results 
in~\cite[Sections~5 and~6]{Fleming}.

\begin{lemma} \label{lemma:approxcspt}
 Let~$\Omega\subseteq\R^{n+k}$ be a bounded, Lipschitz domain.
 Let~$S_0\in\M_n(\overline{\Omega}; \, \G)$ be a \emph{locally polyhedral} chain
 such that~$S_0\mres\partial\Omega = 0$, and
 let~$S\in\M_n(\overline{\Omega}; \, \G)$ be such that~$S \sim_{\overline{\Omega}} S_0$.
 Then, there exists a sequence of \emph{locally polyhedral} chains
 $S_j\in\M_n(\overline{\Omega}; \, \G)$ with the following properties:
 \begin{enumerate}[label=(\roman*)]
  \item $\F(S_j-S)\to 0$ as~$j\to+\infty$;
  \item $\M(S_j)\to\M(S)$ as~$j\to+\infty$;
  \item for any~$j$, we can write~$S_j = S_0 + \partial R_j$
  for some finite-mass $(n+1)$-chain~$R_j$ with compact support in~$\Omega$.
 \end{enumerate}
\end{lemma}
\begin{proof}
 By assumption, there exists~$R\in\M_n(\overline{\Omega}; \, \G)$ such that
 $S = S_0 + \partial R$.  
 Thanks to Lemma~\ref{lemma:compactspt} and a diagonal
 argument, we can assume without loss of generality that $R$
 is compactly supported in~$\Omega$. For any positive~$t$, let 
 $\Omega_t := \{x\in\Omega\colon\dist(x, \, \partial\Omega)>t\}$.
 We take a positive number~$t_0$ such that $\spt R\subseteq\Omega_{2t_0}$,
 and an open set~$U$, with polyhedral boundary, such that
 $\Omega_{2t_0}\csubset U \csubset\Omega_{t_0}$. 
 Because $S$ and~$S_0$ differ by a boundary, we have 
 \[
  \begin{split}
   \partial(S\mres U) + \partial(S\mres (\R^{n+k}\setminus U)) = \partial S 
   = \partial S_0 = \partial(S_0\mres U) + \partial(S_0\mres (\R^{n+k}\setminus U)).
  \end{split}
 \]
 However, $S$ and~$S_0$ agree out of~$U$, so 
 $\partial (S\mres U) = \partial(S_0\mres U)$. 
 In particular, since~$S_0$ is locally polyhedral
 in~$\Omega$ and~$U$ is polyhedral, 
 $\partial (S\mres U)$ is a polyhedral chain.
 Thanks to, e.g., \cite[Theorem~5.6 and~7.7]{Fleming}, there exists a
 sequence of polyhedral $n$-chains~$T_j$ that $\F$-converges to~$S\mres U$,
 satisfies $\spt T_j \subseteq \Omega_{t_0}$ for any $j$ and
 \begin{equation} \label{acspt0}
  \partial T_j=\partial (S\mres U) \quad \textrm{for any } j \in\N,
  \qquad \M(T_j)\to\M(S\mres U) \quad \textrm{as }j\to+\infty.
 \end{equation}
%
 By definition of the $\F$-norm, there exist sequences~$P_j\in\M_{n+1}(\R^{n+k}; \, \G)$
 and $Q_j\in\M_{n}(\R^{n+k}; \, \G)$ such that
 \begin{gather}
  S\mres U - T_j = \partial P_j + Q_j \qquad 
   \textrm{for any } j \label{acspt1} \\
  \M(P_j) \to 0, \quad \M(Q_j) \to 0 
   \qquad \textrm{as } j\to +\infty. \label{acspt2}
 \end{gather}
 We do not know a priori whether the chains~$P_j$, $Q_j$ are 
 supported in~$\Omega$, so we perform a truncation argument. 
 Define
 \[
  P_{j,t} := (\partial P_j)\mres\Omega_t
   - \partial (P_j\mres\Omega_t) 
 \]
 for~$t\in (0, \, t_0)$ and~$j\in\N$.
 By applying Fatou's lemma and~\cite[Theorem~5.7]{Fleming}, we obtain that
 \[
  \int_0^{t_0} \liminf_{j\to+\infty} \M(P_{j,t}) \, \d t \leq 
  \liminf_{j\to+\infty} \int_0^{t_0} \M(P_{j,t}) \, \d t \leq 
  \liminf_{j\to+\infty} \M(P_j) \stackrel{\eqref{acspt2}}{=} 0.
 \]
 Therefore, for a.e.~$t\in (0, \, t_0)$ there exists a (non-relabelled) 
 subsequence $j\to+\infty$ such that $\M(P_{j,t})\to 0$.
 By taking the restriction of~\eqref{acspt1} to~$\Omega_t$,
 we obtain
 \begin{equation} \label{acspt3}
  S\mres U - T_j 
  = \partial ( \ \underbrace{P_j\mres\Omega_t}_{=:P^\prime_j} \ ) 
  + \ \underbrace{P_{j, t} + Q_t\mres\Omega_t}_{=:Q^\prime_j}.
 \end{equation}
 By construction, $P^\prime_j$ and~$Q^\prime_j$ are supported 
 in~$\overline{\Omega_t}\subseteq\Omega$, and there holds
 \begin{equation} \label{acspt4}
  \M(P^\prime_j) \to 0, \quad \M(Q^\prime_j) \to 0 
   \qquad \textrm{as } j\to +\infty.
 \end{equation}
 Moreover, by taking the boundary of both sides of~\eqref{acspt3}, we deduce that 
 \begin{equation*} 
  \partial Q^\prime_j 
  = \partial (S\mres U) - \partial T_j \stackrel{\eqref{acspt0}}{=} 0.
 \end{equation*}
 By applying Lemma~\ref{lemma:def} to~$Q_j^\prime$, we 
 find a decomposition
 \begin{equation} \label{acspt5}
  Q^\prime_j = Q_j^{\prime\prime} + \partial C_j,
 \end{equation}
 where 
 \begin{enumerate}[label=(\alph*)]
  \item $Q_j^{\prime\prime}$ is a polyhedral $n$-chain
  such that $\M(Q_j^{\prime\prime})\lesssim \M(Q_j^\prime)$;
  \item $C_j$ is a $(n+1)$-chain of finite mass 
  and~$\M(C_j)\lesssim j^{-1}\M(Q_j^\prime)$;
  \item $Q_j^{\prime\prime}$ and~$C_j$ are supported in~$\overline{\Omega_{t_0-1/j}}$.
 \end{enumerate}
 From~(a), (b) 
 and~\eqref{acspt4},  we deduce that
 \begin{equation} \label{acspt6}
  \M(C_j)\to 0, \quad \M(Q_j^{\prime\prime})\to 0 
  \qquad \textrm{as } j\to+\infty.
 \end{equation}
 Now, we define 
 \begin{equation} \label{acspt6.5}
  S_j := T_j + Q_j^{\prime\prime} + S\mres(\R^{n+k}\setminus U).
 \end{equation}
 By construction, $S_j$ is locally polyhedral. We have
 \begin{equation} \label{acspt7}
  \begin{split}
   S_j - S 
   = T_j + Q_j^{\prime\prime} - S\mres U
   \stackrel{\eqref{acspt3}}{=}
   Q_j^{\prime\prime} - \partial P_j^\prime - Q_j^\prime
   \stackrel{\eqref{acspt5}}{=} - \partial (P^\prime_j + C_j)
  \end{split}
 \end{equation}
 and hence, $\F(S_j-S)\to 0$ due to~\eqref{acspt4} and~\eqref{acspt6}.
 By the lower semi-continuity of the mass, we deduce that
 $\M(S)\leq\liminf_{j\to+\infty}\M(S_j)$. On the other hand, 
 if we apply the triangle inequality to~\eqref{acspt6.5}
 and use the identity
 $\M(S) = \M(S\mres U) + \M(S\mres(\R^{n+k}\setminus U))$, we obtain
 \[
  \begin{split}
   \M(S_j) - \M(S) \leq \M(T_j) + \M(Q_j^{\prime\prime}) - \M(S\mres U).
  \end{split}
 \]
 The right hand side converges to zero as~$j\to+\infty$,
 due to~\eqref{acspt0} and~\eqref{acspt6}. Thus, we deduce that
 $\limsup_{j\to+\infty}\M(S_j)\leq\M(S)$, and hence $\M(S_j)\to\M(S)$ as~$j\to+\infty$.
 Finally, we define $R_j := R - P^\prime_j - C_j$. Then,
 \eqref{acspt7} gives~$S_j - S_0 = \partial R_j$ and the lemma follows.
\end{proof}

\begin{proof}[Proof of Proposition~\ref{prop:approx}]
 Let~$S_0$, $S$ be given, as in the statement.
 By applying Lemma~\ref{lemma:approxcspt}, we find a sequence of 
 locally polyhedral chains $S_j\in\M_{n}(\overline{\Omega}; \, \G)$
 and a sequence of finite-mass~$(n+1)$-chains~$\tilde{R}_j$, 
 compactly supported in~$\Omega$,
 such that $S_j\to S$ in the~$\F$-norm, 
 $\M(S_j)\to\M(S)$ and $S_j = S_0 + \partial \tilde{R}_j$ 
 for any~$j$. Since~$S_0$, $S_j$ are locally polyhedral in~$\Omega$, 
 $\partial \tilde{R}_j$ is polyhedral. We apply Lemma~\ref{lemma:def} 
 to each $\tilde{R}_j$. We find polyhedral $(n+1)$-chains~$R_j$,
 compactly supported in~$\Omega$, and $(n+2)$-chains~$C_j$ of finite mass,
 such that
 \begin{gather*}
  \tilde{R}_j = R_j + \partial C_j.
 \end{gather*}
 Then, $S_j = S_0 + \partial(R_j + \partial C_j) = S_0 + \partial R_j$,
 and the proposition follows.
\end{proof}

\subsection{A characterisation of the mass of a rectifiable chain}
\label{sect:mass_chains}

For any linear subspace~$L\subseteq\R^{n+k}$, 
we let~$\pi_L\colon\R^{n+k}\to L$ be the orthogonal projection onto~$L$.
A $n$-chain of class~$C^1$ is a chain~$S$ that can be written in the form~$S = f_{*}P$,
with~$f$ a map of class~$C^1$ and~$P$ a polyhedral chain.
The set of rectifiable $n$-chains is defined as the closure 
of $n$-chains of class~$C^1$ with respect to the~$\M$-norm.

\begin{lemma} \label{lemma:mass_chain} 
 Let~$S\in\M_n(\R^{n+k}; \, \G)$ be a rectifiable $n$-chain. Then, 
 \[
  \M(S) = \sup_{(U_i, \, L_i)_{i\in\N}} \,
  \sum_{i=0}^{+\infty} \M(\pi_{L_i,*}(S\mres U_i)),
 \]
 where the supremum is taken over all sequences of
 pairwise disjoint open sets $U_i$ and $n$-planes $L_i\subseteq\R^{n+k}$.
\end{lemma}

If the coefficient group satisfies~\eqref{G_discrete},
as is the case for~$\G = \pi_{k-1}(\NN)$, then \emph{any} chain
of finite mass is rectifiable, by White's Rectifiability Theorem
\cite[Theorem~7.1]{White-Rectifiability}. Therefore, 
Lemma~\ref{lemma:mass_chain} implies
Lemma~\ref{lemma:mass_chain-nointro}.

\begin{proof}
 Let~$(U_i)_{i\in\N}$ be a sequence of
 pairwise disjoint open sets, and let~$(L_i)_{i\in\N}$ 
 be a sequence of $n$-planes in~$\R^{n+k}$.
 For any~$i$, the projection $\pi_{L_i}$ is a $1$-Lipschitz
 map and hence $\M(\pi_{L_i,*}(S\mres U_i))\leq \M(S\mres U)$
 (see e.g. \cite[Eq.~(5.1)]{Fleming}). Since the~$U_i$'s
 are assumed to be pairwise disjoint, we obtain
 \begin{equation} \label{step0}
  \sum_{i=0}^{+\infty} \M(\pi_{L_i,*}(S\mres U_i)) \leq
  \sum_{i=0}^{+\infty} \M(S\mres U_i) \leq \M(S).
 \end{equation}
 This proves one of the inequalities. To prove the opposite inequality,
 we first suppose that~$S$ is a $C^1$-polyhedron, 
 then a $C^1$-chain, and finally
 we extend the result to an arbitrary rectifiable chain.
 We denote by $\mathrm{int}\, A$ the interior of a set 
 $A\subseteq\R^{n+k}$, and by~$\mathrm{diam}\, A$ its diameter.
 
 \setcounter{step}{0}
 \begin{step}[$S$ is a~$C^1$-polyhedron]
  We suppose that $S=f_*(\sigma\llbracket K\rrbracket)$, 
  where~$\sigma\in\G$, $K$ is a convex, compact $n$-polyhedra, and
  $f\colon\R^{n+k}\to\R^{n+k}$ is a $C^1$-diffeomorphism.
  Let~$\eta > 0$ be arbitrarily fixed. Since $f$ is~$C^1$
  and~$K$ is compact, there exists~$\rho>0$ such that
  \begin{equation} \label{step1-1}
   \norm{\nabla f(x) - \nabla f(y)} \leq\eta
   \qquad \textrm{if } (x, \, y)\in K\times K
   \textrm{ and } \abs{x-y}\leq\rho,
  \end{equation}
  where $\|\cdot\|$ denotes the operator norm on the space 
  of real $(n+k)\times (n+k)$-matrices.
  Let~$(T_i)_{i=1}^q$ be a collection of~$n$-simplices that
  triangulate~$K$, such that
  \begin{equation} \label{step1-2}
   \max_{1\leq i\leq q} \mathrm{diam}\,T_i \leq\rho.
  \end{equation}
  Let~$V_i := \mathrm{int}\, U(T_i, \, \rho/2, \, \rho/2)$
  where $U(T_i, \, \rho/2, \, \rho/2)$ is
  defined as in~\eqref{U-diamond}, and $U_i:=f(V_i)$.
  The $U_i$'s are pairwise disjoint open sets,
  because $f$ is a diffeomorpism. Let~$L$ be the~$n$-plane
  passing through the origin that is parallel to~$K$.
  For any~$i$, we choose a point~$x_i\in\mathrm{int}\, T_i$
  and we define $L_i := \nabla f(x_i)(L)$. The $L_i$'s are indeed 
  $n$-planes, because $\nabla f(x_i)$ is an invertible linear map.
  For any~$x\in K\cap T_i$ and any~$y\in L$, we have
  \[
   \begin{split}
    \abs{(\pi_{L_i}\circ\nabla f)(x)\,y - \nabla f(x)\,y} &\leq 
    \abs{(\pi_{L_i}\circ\nabla f)(x_i)\,y - \nabla f(x_i)\,y}
    + 2\norm{\nabla f(x_i) - \nabla f(x)} \abs{y} \\
    &\stackrel{\eqref{step1-1}-\eqref{step1-2}}{\leq} 2\eta \abs{y}.
   \end{split}
  \]
  Therefore, by applying the area formula we obtain
  \[
   \begin{split}
    \abs{\M(\pi_{L_i,*}(S\mres U_i)) - \M(S\mres U_i)} &\leq
    \abs{\sigma}_* \abs{\H^n((\pi_{L_i,*}\circ f)(K\cap V_i))
      - \H^n(f(K\cap V_i))} \\
    &\leq C\eta\abs{\sigma}_* \H^n(K\cap V_i)
   \end{split}
  \]
  for some constant~$C$ depending only on~$n$, $k$. This implies
  \[
   \begin{split}
    \sum_{i=1}^q \M(\pi_{L_i,*}(S\mres U_i)) &\geq 
    \sum_{i=1}^q \M(S\mres U_i) - 
       C\eta \abs{\sigma}_* \sum_{i=1}^q \H^n(K\cap V_i) \\
    &\geq \M(S\mres \cup_i U_i) - C\eta \abs{\sigma}_* \H^n(K),
   \end{split}
  \]
  where~$\eta$ is arbitrarily small.
  To complete the proof in this case, it only remains
  to notice that $\M(S) = \M(S\mres \cup_i U_i)$,
  because $\H^n(K\setminus\cup_i V_i) = 0$ and
  $\H^n(S\setminus\cup_i U_i)=0$ by the area formula.
 \end{step}
 
 \begin{step}[$S$ is a $C^1$-chain]
  We suppose that $S=f_*P$, where~$P$ is a polyhedral~$n$-chain and
  $f\colon\R^{n+k}\to\R^{n+k}$ is a $C^1$-diffeomorphism.
  This case follows easily from the previous one, by additivity.
  Indeed, let us write $S = \sum_{j=1}^p \sigma_j\llbracket K_j\rrbracket$
  with~$\sigma_j\in\G$, $K_j$ a convex, compact $n$-polyhedra. 
  Given  positive parameters $\delta$, $\gamma$, let 
  $W^j := f(\mathrm{int}\,U(K_j, \, \delta, \, \gamma))$.
  For~$\delta$, $\gamma$ small enough,
  the $W^j$ have pairwise disjoint interiors.
  Let~$\eta>0$ be fixed. By applying Step~1, for any~$j$
  we find a sequence~$(V_i^j)_{i\in\N}$ of pairwise disjoint open sets
  and a sequence~$(L_i^j)_{i\in\N}$ of $n$-planes such that
  \begin{equation} \label{step-2}
   \sum_{i=0}^{+\infty} \M(\pi_{L_i^j,*}
   (f_*(\sigma_j\llbracket K_j\rrbracket)\mres U_i^j))
   \geq \M(f_*(\sigma_j\llbracket K_j\rrbracket)) - \eta.
  \end{equation}
  We define $U_i^j := V_i^j\cap W^j$.
  The $U_i^j$'s are pairwise disjoint open sets.
  We have $\H^n(K_j\setminus\mathrm{int}\,
  U(K_j, \, \delta, \, \gamma)) = 0$
  and hence, by the area formula,
  $\M(f_*(\sigma_j\llbracket K_j\rrbracket)\mres
  (\R^{n+k}\setminus W^j)) = 0$.
  Therefore, we obtain
  \[
   \sum_{i,j} \M(\pi_{L_i^j,*}
   (f_*(\sigma_j\llbracket K_j\rrbracket)\mres U_i^j))
   = \sum_{i,j}  \M(\pi_{L_i^j,*}
   (f_*(\sigma_j\llbracket K_j\rrbracket)\mres V_i^j))
   \stackrel{\eqref{step-2}}{\geq} \M(S) - p\eta
  \]
  and the lemma is proved also in this case, because
  $\eta$ may be taken arbitrarily small.
 \end{step}
 
 \begin{step}[$S$ is a rectifiable chain]
  Since~$S$ is rectifiable, for any~$\eta>0$ there exist
  a polyhedral $n$-chain~$P$ and a $C^1$-diffeomorphism 
  $f\colon\R^{n+k}\to\R^{n+k}$ such that
  \begin{equation} \label{step3-1}
   \M(S - f_*P)\leq\eta
  \end{equation}
  (see e.g. \cite[1.2 at p.~169]{White-Rectifiability}
  and the references therein).
  By Step~2, there exists a sequence $(U_i, \, L_i)_{i\in\N}$ such that
  \begin{equation} \label{step3-2}
   \sum_{i=0}^{+\infty} \M(\pi_{L_i,*}(f_*P\mres U_i))
   \geq \M(f_*P) - \eta.
  \end{equation}
  Let us set~$Q := S - f_*P$. Then, using the linearity of~$\pi_{L_i,*}$,
  $\cdot\, \mres U_i$, and the triangle inequality for~$\M$,
  we obtain
  \[
   \begin{split}
    \sum_{i=0}^{+\infty} \M(\pi_{L_i,*}(S\mres U_i)) 
    &\geq \sum_{i=0}^{+\infty} \M(\pi_{L_i,*}(f_*P\mres U_i))
    - \sum_{i=0}^{+\infty} \M(\pi_{L_i,*}(Q\mres U_i)) \\
    &\stackrel{\eqref{step0}, \, \eqref{step3-2}}{\geq} \M(f_*P) - \eta - \M(Q)
    \stackrel{\eqref{step3-1}}{\geq} \M(S) - 3\eta
   \end{split}
  \]
  so the lemma follows. \qedhere
 \end{step}
\end{proof}

\end{appendix}

\bibliographystyle{plain}
\bibliography{singular_set}

\end{document}